\documentclass{birkjour}

\usepackage[utf8]{inputenc}
\usepackage[T1]{fontenc}
\usepackage[english]{babel}
\usepackage{mathtools,amsmath,amssymb,amsfonts,amsthm,mathrsfs}
\usepackage{bbm}
\mathtoolsset{showonlyrefs}
\usepackage{enumitem}

\usepackage{graphicx,color}
\usepackage{tikz,pgfplots,pgf}
\usepackage{epsfig}
\usepackage{subcaption}
\usepackage[font=small]{caption}

\newtheorem{theorem}{Theorem}[section]

\newtheorem{lemma}{Lemma}[section]

\newtheorem{corollary}{Corollary}[section]

\theoremstyle{definition}
\newtheorem{remark}{Remark}[section]

\numberwithin{equation}{section}

\begin{document}

\title[Positive periodic solutions to a Minkowski-curvature equation]{Positive periodic solutions to an indefinite \\ Minkowski-curvature equation}

\author[A.~Boscaggin]{Alberto Boscaggin}

\address{
Department of Mathematics, University of Torino\\
Via Carlo Alberto 10, 10123 Torino, Italy}

\email{alberto.boscaggin@unito.it}

\author[G.~Feltrin]{Guglielmo Feltrin}

\address{
Department of Mathematics, University of Torino\\
Via Carlo Alberto 10, 10123 Torino, Italy}

\email{guglielmo.feltrin@unito.it}

\thanks{Work written under the auspices of the 
Grup\-po Na\-zio\-na\-le per l'Anali\-si Ma\-te\-ma\-ti\-ca, la Pro\-ba\-bi\-li\-t\`{a} e le lo\-ro
Appli\-ca\-zio\-ni (GNAMPA) of the Isti\-tu\-to Na\-zio\-na\-le di Al\-ta Ma\-te\-ma\-ti\-ca (INdAM).
The authors are supported by the project ERC Advanced Grant 2013 n.~339958 ``Complex Patterns for Strongly Interacting Dynamical Systems - COMPAT''.
\\
\textbf{Preprint -- May 2018}} 



\date{}

\dedicatory{}

\begin{abstract}
We investigate the existence, non-existence, multiplicity of positive periodic solutions, both harmonic (i.e., $T$-periodic) and subharmonic 
(i.e., $kT$-periodic for some integer $k \geq 2$) to the equation 
\begin{equation*}
\Biggl{(} \dfrac{u'}{\sqrt{1-(u')^{2}}} \Biggr{)}' + \lambda a(t) g(u) = 0,
\end{equation*}
where $\lambda > 0$ is a parameter, $a(t)$ is a $T$-periodic sign-changing weight function and $g \colon \mathopen{[}0,+\infty\mathclose{[} \to \mathopen{[}0,+\infty\mathclose{[}$ is a continuous function having superlinear growth at zero. In particular, we prove that for both $g(u)=u^{p}$, with $p>1$, and $g(u)= u^{p}/(1+u^{p-q})$, with $0 \leq q \leq 1 < p$, the equation has no positive $T$-periodic solutions for $\lambda$ close to zero and two positive $T$-periodic solutions (a ``small'' one and a ``large'' one) for $\lambda$ large enough. Moreover, in both cases the ``small'' $T$-periodic solution is surrounded by a family of positive subharmonic solutions with arbitrarily large minimal period.
The proof of the existence of $T$-periodic solutions relies on a recent extension of Mawhin's coincidence degree theory for locally compact operators in product of Banach spaces, while subharmonic solutions are found by an application of the Poincar\'{e}--Birkhoff fixed point theorem, after a careful asymptotic analysis of the $T$-periodic solutions for $\lambda \to +\infty$. 
\end{abstract}

\maketitle

\section{Introduction}\label{section-1}

In this paper, we investigate the existence of positive periodic solutions to the equation
\begin{equation}\label{eq-main}
\Biggl{(} \frac{u'}{\sqrt{1-(u')^{2}}}\Biggr{)}' + \lambda a(t) g(u) = 0,
\end{equation}
where $\lambda > 0$ is a parameter, $a(t)$ is a $T$-periodic and locally integrable weight function and $g: \mathbb{R}^{+} := \mathopen{[}0,+\infty\mathclose{[}\to \mathbb{R}^{+}$ is a continuous function satisfying the sign-condition
\begin{equation*}
g(0) = 0, \qquad g(u) > 0 \quad \text{for $u>0$}.
\leqno{(g_{*})}
\end{equation*}
Incidentally, let us observe that, by the above condition, equation \eqref{eq-main} has the trivial solution $u(t) \equiv 0$ and, unless $a(t) \equiv 0$, no other constant (periodic) solutions.

Equation \eqref{eq-main} is driven by a strongly nonlinear differential operator of $\varphi$-laplacian type, precisely
\begin{equation}\label{phi-L}
u \mapsto - (\varphi(u'))', \quad \text{where } \; \varphi(\xi) := \frac{\xi}{\sqrt{1-\xi^{2}}}.
\end{equation}
As is well known, this is the one-dimensional version of the partial differential operator
\begin{equation*}
u \mapsto -\mathrm{div}\,\Biggl{(} \dfrac{\nabla u}{\sqrt{1- | \nabla u |^{2}}}\Biggr{)},
\end{equation*}
which in turn is usually meant as a mean-curvature operator in Lorentz--Minkowski spaces (cf.~\cite{BaSi-8283,Fl-79,Ge-83}); interestingly, it also plays a role in the theory of nonlinear electromagnetism, being known in this context as Born--Infeld operator (cf.~\cite{BoCoFo-pp,BdAP-16} and the references therein). Recently, there has been a significant interest in the study of the existence and multiplicity issues of the associated boundary value problems, both in the ODE and in the PDE cases (see, among many others, \cite{Az-14,Az-16,BeJeMa-09,BeJeTo-13,BeJeTo-13a,BeMa-07,CoCoObOm-12,CoCoRi-14,CoObOmRi-13,Ma-13}). However, the interplay between the nonlinear differential operator, the periodic boundary conditions and the specific form of the nonlinear term $f_{\lambda}(t,u) = \lambda a(t) g(u)$ makes the analysis of \eqref{eq-main} rather new with respect to the existing literature.

To motivate this assertion, let us start by discussing the solvability picture of equation \eqref{eq-main} when paired with Dirichlet boundary conditions $u(0) = u(T) = 0$. A first simple observation is that, in this case, due to the natural bound $|u'(t)| < 1$, any solution to 
\eqref{eq-main} is a priori bounded, precisely $|u(t)|< T/2$ for every $t$; hence, the solvability of the Dirichlet boundary value problem is not affected by the value of $g(u)$ for $u \geq T/2$. On the other hand, it has been proved in \cite{CoCoObOm-12} that, when $g(u)$ has superlinear growth at zero, namely 
\begin{equation}\label{super0}
\lim_{u \to 0^{+}} \frac{g(u)}{u} = 0,
\end{equation}
then two positive solutions to the Dirichlet problem associated with \eqref{eq-main} exist when $\lambda$ is sufficiently large, provided that $a(t)$ is positive somewhere. Such a two-solution theorem has been later extended to the radial Dirichlet problem on a ball in \cite{BeJeTo-13,CoCoRi-14}, as well as to a genuine PDE setting in \cite{CoObOmRi-13}. We also mention that very recently, via a dynamical systems approach, it has been proved that more and more pairs of sign-changing Dirichlet solutions to \eqref{eq-main} appear as $\lambda$ becomes larger and larger (see \cite{BoGa-ccm}).

The above two-solution geometry, which is quite popular in Nonlinear Functional Analysis (see \cite{Am-72}), can be roughly explained as follows. Condition \eqref{super0} implies that the nonlinear term $\lambda a(t)g(u)/u$ in \eqref{eq-main} stays, for $u$ near zero, below the principal eigenvalue of the linear operator $u \mapsto -u''$ with Dirichlet boundary conditions (here, $-u''$ appears as the linearization of the Minkowski-curvature operator near zero). This fact, together with the global a priori bound, can be used to show that a suitably defined topological degree associated with \eqref{eq-main} is equal to $1$ on small balls as well as on large balls centered at the origin (of a Banach space of $T$-periodic functions). On the other hand, the largeness of the parameter $\lambda$ implies the existence of a ball with intermediate radius on which the topological degree is equal to $0$. Therefore, the existence of two positive solutions (a ``small'' one and a ``large'' one) follows from the additivity property of the degree. A variational viewpoint can also be adopted, providing the two positive solutions as a global minimum and a mountain pass type critical point of the associated action functional.

When dealing with positive $T$-periodic solutions to \eqref{eq-main}, the situation differs substantially. Indeed, now the principal eigenvalue of $-u''$ is null, so that \eqref{eq-main} experiences a resonance near zero; moreover, contrarily to Dirichlet boundary conditions, the obvious global a priori bound is no longer available, so that growth assumptions at infinity for $g(u)$ are expected to play a role.

As a matter of fact, some other important considerations about the periodic problem associated with \eqref{eq-main} can be deduced with elementary arguments. First, as it can be shown just by integrating the equation on $\mathopen{[}0,T\mathclose{]}$, positive $T$-periodic solutions cannot exist if $a(t) \geq 0$ for every $t$ and, consequently, one is led to deal with a sign-changing weight $a(t)$ (often named ``indefinite weight'', starting with \cite{HeKa-80}). Second, as observed various times in the related literature dealing with the second order linear operator $- u''$ (compare with \cite{BaPoTe-88,BoZa-12,Fe-18}), the mean value condition 
\begin{equation*}
\int_{0}^{T} a(t) \,dt <0
\leqno{(a_{\#})}
\end{equation*}
also naturally appears. Indeed, by dividing equation \eqref{eq-main} by $g(u(t)) (\neq 0)$ and integrating (by parts) on a period, one finds
\begin{equation*}
- \int_{0}^{T} a(t)\,dt = \int_{0}^{T} \biggl{(} \frac{u'(t)}{g(u(t))}\biggr{)}^{2} \frac{g'(u(t))}{\sqrt{1-u'(t)^{2}}}\,dt, 
\end{equation*}
implying that $(a_{\#})$ is a necessary condition when $g(u)$ is a strictly increasing function of class $\mathcal{C}^{1}$.

For the semilinear equation
\begin{equation}\label{eq-sem}
u'' + \lambda a(t) g(u) = 0
\end{equation}
some results in this indefinite setting have been obtained in \cite{BoFeZa-16,FeZa-15ade}. To summarize the main conclusions obtained therein, we consider the two model equations
\begin{equation}\label{eq-sem1}
u'' + \lambda a(t) u^{p} = 0, \quad \text{with $p >1$,}
\end{equation}
and
\begin{equation}\label{eq-sem2}
u'' + \lambda a(t) \dfrac{u^{p}}{1+u^{p-q}}= 0, \quad \text{with $0 \leq q < 1 < p$.}
\end{equation}
We notice that in both cases the nonlinear term is superlinear near zero (that is, \eqref{super0} is satisfied), while the behavior at infinity
differs, being of superlinear type for \eqref{eq-sem1} and of sublinear one for \eqref{eq-sem2}. Assuming the mean value condition $(a_{\#})$ (together with a mild condition on the nodal behavior of $a(t)$, see $(a_{*})$ in Section~\ref{section-1.1}), the existence of a positive $T$-periodic solution to \eqref{eq-sem1} is guaranteed for every $\lambda > 0$ (cf.~\cite[Theorem~3.2]{FeZa-15ade}), while two positive $T$-periodic solutions to \eqref{eq-sem2} exist, but only for $\lambda$ sufficiently large (cf.~\cite[Theorem~4.4 and Theorem~4.6]{BoFeZa-16}).

Roughly speaking, we can say that the superlinearity at zero, when paired with $(a_{\#})$ (and further suitable technical assumptions, see $(g_{0})$ and $(g_{0}')$ in Section~\ref{section-1.1}), still provides, in the indefinite periodic setting, the desired geometry near zero, thus implying that the topological degree is equal to $1$ on small balls (it is impressive to interpret $(a_{\#})$ as a non-resonance condition pushing the term $\lambda a(t)g(u)/u$ below the principal eigenvalue for $u \to 0^{+}$). Then, depending on the behavior of the nonlinear term $g(u)$ at infinity, two very different scenarios arise. On one hand, when $g(u)$ behaves like $u^{q}$ with $0 \leq q < 1$ for $u \to +\infty$, the topological degree on large balls is equal to $1$. In such a case, using the largeness of the parameter $\lambda$ in order to create a ball of intermediate radius on which the degree equal to $0$, one can conclude that a two-solution result holds true, similarly as for \eqref{eq-main} with Dirichlet boundary conditions. On the other hand, when $g(u)$ is superlinear at infinity the topological degree is equal to $0$ on large balls and, consequently, the existence of a positive solution follows for any $\lambda > 0$. 

Motivated by the above discussion, it seems interesting to investigate the existence of positive $T$-periodic solutions to \eqref{eq-main}: indeed, while it is reasonable to expect that the mean value condition $(a_{\#})$ is still going to play a role, it is not clear to what extent the behavior of $g(u)$ at infinity can interfere with the nonlinear differential operator, thus giving rise (or not) to different existence/multiplicity patterns. In particular, one could wonder about the solvability of Minkowski-curvature counterparts of \eqref{eq-sem1} and \eqref{eq-sem2}, namely 
\begin{equation}\label{eq-min1}
(\varphi(u'))' + \lambda a(t) u^{p} = 0, \quad \text{with $p > 1$,}
\end{equation}
and
\begin{equation}\label{eq-min2}
(\varphi(u'))' + \lambda a(t) \dfrac{u^{p}}{1+u^{p-q}} = 0, \quad \text{with $0 \leq q < 1 < p$,}
\end{equation}
with $\varphi$ defined in \eqref{phi-L}.
As a consequence of our first main result, we can prove that both the above problems behave in the same way, according to the following two-solution theorem.

\begin{theorem}\label{th-1.1}
Let $a(t)$ be a sign-changing, continuous and $T$-periodic function, having a finite number of zeros in $\mathopen{[}0,T\mathclose{]}$ and satisfying condition $(a_{\#})$. Then, there exist $\lambda_{*}$ and $\lambda^{*}$, with 
$0 < \lambda_{*} \leq \lambda^{*}$, such that both equation \eqref{eq-min1} and equation \eqref{eq-min2} have no positive $T$-periodic solutions for $\lambda \in \mathopen{]}0,\lambda_{*}\mathclose{[}$ and two positive $T$-periodic solutions for $\lambda > \lambda^{*}$.
\end{theorem}

The above result, showing that the behavior at infinity is essentially governed by the nonlinear differential operator even for the superlinear power $g(u) = u^{p}$ with $p > 1$, will be proved as a corollary of Theorem~\ref{th-main-ex} and Theorem~\ref{th-main-nex}, dealing with equation \eqref{eq-main} under more general conditions on $a(t)$ and $g(u)$, both for $u$ near zero and for $u$ near infinity. In particular, we will show that the existence of two positive $T$-periodic solutions to \eqref{eq-min2} is valid (when $\lambda$ is large) also for $q=1$, that is, when $g(u)$ has linear growth at infinity.

The proof of Theorem~\ref{th-main-ex} relies on topological degree theory, according to the aforementioned strategy of showing that a suitable topological degree associated with \eqref{eq-main} is equal to $1$ on small and large balls centered at the origin and is equal 
to $0$ (for large values of the parameter $\lambda$) on a ball of intermediate radius. While the needed technical estimates are inspired by the ones in \cite{BoFeZa-16,FeZa-15ade} for the semilinear equation \eqref{eq-sem}, the most delicate point in our strongly nonlinear setting is actually the definition of the degree. In \cite{MaMa-98,Ma-13}, some continuations theorems have been proposed and used for equations driven by nonlinear operators: they all rely on the formulation of \eqref{eq-main} as a fixed point problem on a Banach space of $T$-periodic functions and on a direct use of the Leray--Schauder degree theory. We choose to adopt a different approach: precisely (after having extended the nonlinear term $f_{\lambda}(t,u) = \lambda a(t)g(u)$ to the whole real line in such a way that the positivity of $T$-periodic solutions can be recovered by maximum principle arguments) we write the equation $(\varphi(u'))'+ f_{\lambda}(t,u) = 0$ as the first order planar system 
\begin{equation*}
u' = \varphi^{-1}(v), \quad v' = -f_{\lambda}(t,u),
\end{equation*}
so as to apply Mawhin coincidence degree theory on the product space $\mathcal{C}(\mathopen{[}0,T\mathclose{]}) \times \mathcal{C}(\mathopen{[}0,T\mathclose{]})$. It is worth mentioning two subtleties of this approach. First, we need to evaluate the degree on sets of the type 
$B(0,d) \times \mathcal{C}(\mathopen{[}0,T\mathclose{]})$, with $B(0,d)$ the ball centered at the origin and having radius $d$ in the space $\mathcal{C}(\mathopen{[}0,T\mathclose{]})$; consequently, we need to use the extension of the Leray--Schauder degree theory for locally compact operators. Second, since the natural homotopy $(\varphi(u'))'+ \vartheta f_{\lambda}(t,u) = 0$ (with $0 \leq \vartheta \leq 1$) leads to the system
\begin{equation}\label{sys-intro}
u' = \varphi^{-1}(v), \quad v' = -\vartheta f_{\lambda}(t,u), 
\end{equation}
standard continuation theorems do not apply (notice, indeed, that the parameter $\vartheta$ appears only in the second equation) and we need to use an extension of the classical theory, recently developed in \cite{FeZa-17tmna} and refined in Appendix~\ref{appendix-B}. 

Some numerical simulations illustrating Theorem~\ref{th-main-ex} are depicted in Figure~\ref{fig-01} and Figure~\ref{fig-02}.
It is worth noticing that the bifurcation branches on varying of the parameter $\lambda$ strongly suggest that the ``small'' solution goes to $0$ for $\lambda \to +\infty$, while the ``large'' solutions remains bounded and bounded away from zero. We will indeed carefully discuss this in Section~\ref{section-4}, by identifying suitably topologies in which the convergence of both ``small'' and ``large'' solutions to their limit profiles takes place: in particular, we will prove that the ``large'' solution converge, for $\lambda \to +\infty$ to a (non-zero) Lipschitz function having derivative equal to $0$, $1$ or $-1$ for almost every $t \in \mathbb{R}$.

\begin{figure}[htb]
\centering
\begin{tikzpicture}[scale=1]
\begin{axis}[
  tick label style={font=\scriptsize},
          scale only axis,
  enlargelimits=false,
  xtick={0,1,2,3,4},
  xticklabels={$0$, $\frac{\pi}{2}$, , , $2\pi$},
  ytick={0,4},
  xlabel={\small $t$},
  ylabel={\small $u(t)$},
  max space between ticks=50,
                minor x tick num=1,
                minor y tick num=7,  
every axis x label/.style={
below,
at={(3.5cm,0cm)},
  yshift=-3pt
  },
every axis y label/.style={
below,
at={(0cm,2.5cm)},
  xshift=-3pt},
  y label style={rotate=90,anchor=south},
  width=7cm,
  height=5cm,  
  xmin=0,
  xmax=4,
  ymin=0,
  ymax=4]
\addplot graphics[xmin=0,xmax=4,ymin=0,ymax=4] {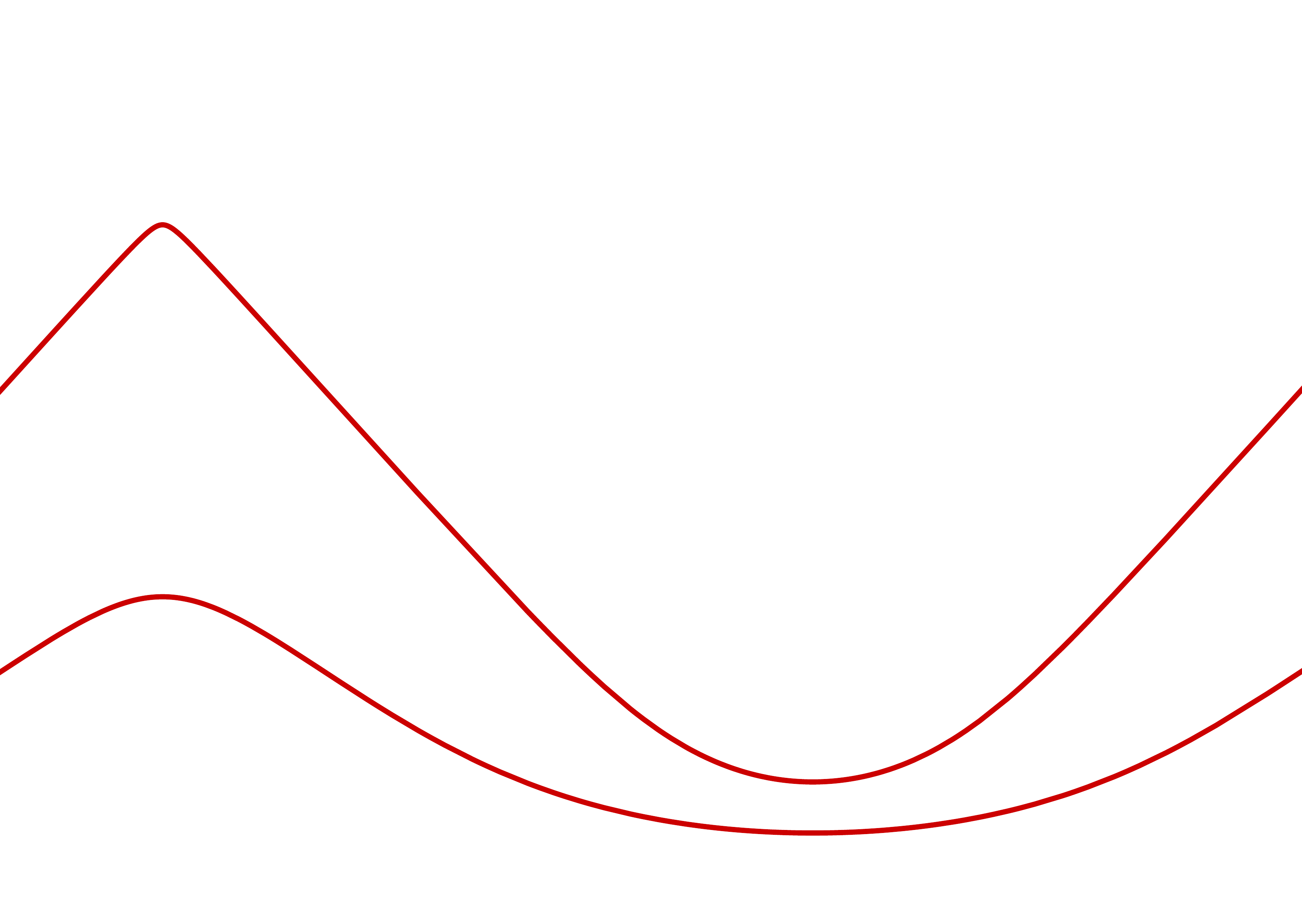};
\end{axis}
\end{tikzpicture} 
\caption{Graphs of two positive $2\pi$-periodic solutions to \eqref{eq-min1} with $p=3$, $a(t)=\cos(t-\pi/4)-\sqrt{2}/2$ and $\lambda=2$. We notice that $a(t)>0$ on $\mathopen{]}0,\pi/2{[}$ and $a(t)<0$ on $\mathopen{]}\pi/2,2\pi{[}$.}         
\label{fig-01}
\end{figure}

\begin{figure}[!htb]
\centering
\begin{tikzpicture}[scale=1]
\begin{axis}[
  tick pos=left,
  tick label style={font=\scriptsize},
          scale only axis,
  enlargelimits=false,
  xtick={0,30},
  ytick={0,4},
  xlabel={\small $\lambda$},
  ylabel={\small $\|u\|_{\infty}$},
  max space between ticks=30,
                minor x tick num=5,
                minor y tick num=7,  
every axis x label/.style={
below,
at={(4cm,0cm)},
  yshift=-3pt
  },
every axis y label/.style={
below,
at={(0cm,2.5cm)},
  xshift=-3pt},
  y label style={rotate=90,anchor=south},
  width=8cm,
  height=5cm,  
  xmin=0,
  xmax=30,
  ymin=0,
  ymax=4]
\addplot graphics[xmin=0,xmax=30,ymin=0,ymax=4] {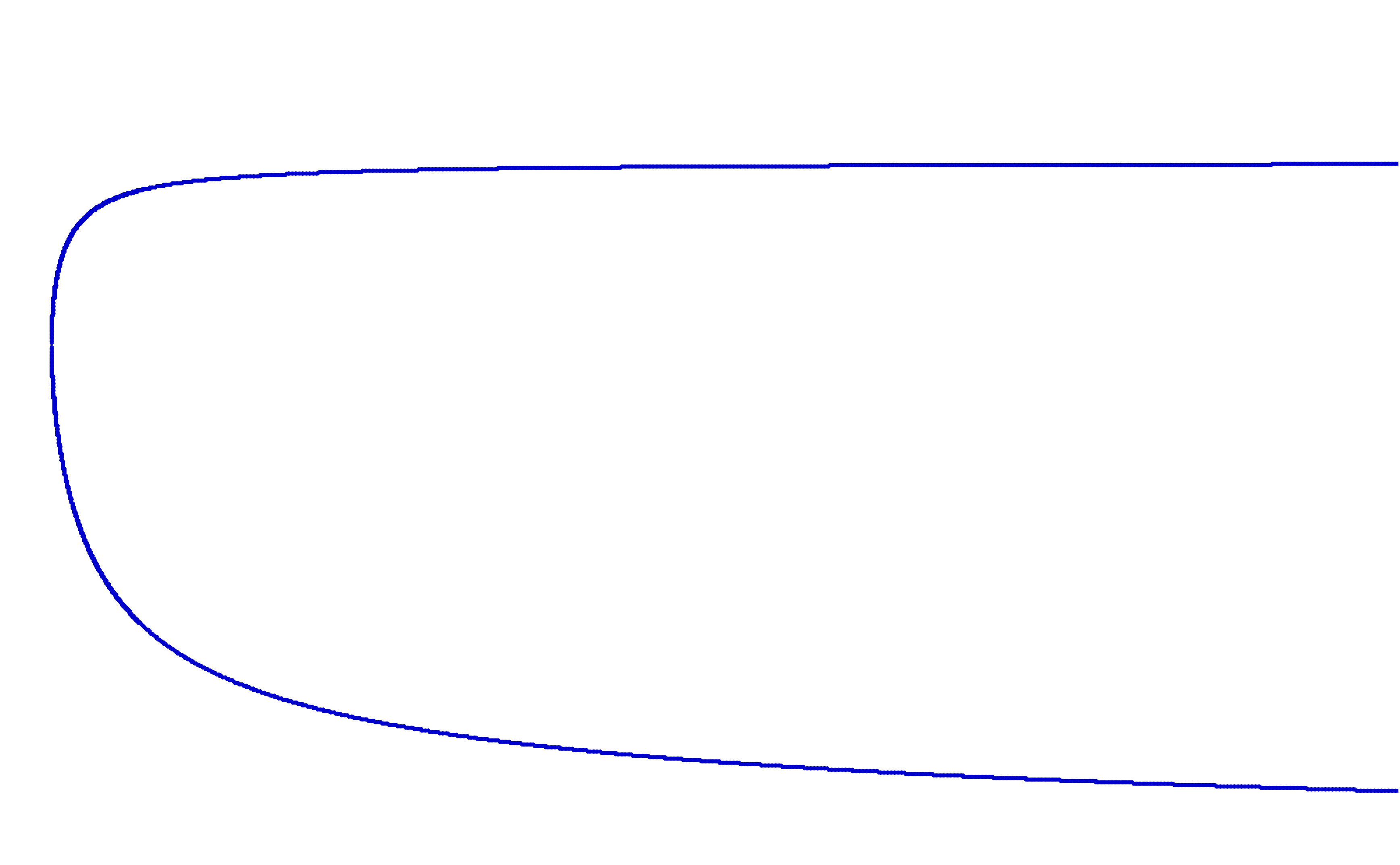};
\end{axis}
\end{tikzpicture}
\caption{Bifurcation diagram with bifurcation parameter $\lambda\in\mathopen{[}0,30\mathclose{]}$ for the $2\pi$-periodic problem associated with \eqref{eq-min1}, setting $p=3$, $a(t)=\cos(t-\pi/4)-\sqrt{2}/2$. 
The $x$-axis represents the parameter $\lambda$ and the $y$-axis the $\sup$-norm of the $2\pi$-periodic solutions to \eqref{eq-min1}. The trivial identically zero solution, which exists for all $\lambda>0$, is not represented.}     
\label{fig-02}
\end{figure}

We mention some very recent results, dealing with equations driven by a nonlinear differential operator and involving an indefinite weight, which are somehow related to Theorem~\ref{th-1.1}. In \cite{BeZa-18} the $T$-periodic problem for equation \eqref{eq-main} is considered, and existence of solutions is proved (via topological degree theory) for the singular nonlinearity $g(u) = -1/u^{\mu}$ for $\mu > 1$ (on a line of research previously developed in \cite{BrTo-10,HaZa-17,Ur-16} in the semilinear case). On the other hand, 
in \cite{LGOmRi-17na,LGOmRi-17jde}, for the Neumann boundary value problem associated with the mean-curvature equation in Euclidean spaces, namely
\begin{equation*}
\Biggl{(} \frac{u'}{\sqrt{1+(u')^{2}}}\Biggr{)}' + \lambda a(t) g(u) = 0,
\end{equation*}
existence/multiplicity of solutions is provided (via bifurcation theory and critical point theory) when $g(u)$ is a power-like function both at zero and at infinity. Notice that in all these results the fact that $a(t)$ is sign-changing is a necessary condition and the mean value condition $(a_{\#})$ plays a role.
\medskip

The second part of our investigation concerns positive subharmonic solutions to \eqref{eq-main}, i.e., solutions with minimal period $kT$ for some integer $k \geq 2$; this is often a quite delicate issue, the crucial point being indeed the proof of the minimality of the period. As far as functional analytic methods are considered, the typical strategy is that of using variational techniques, by finding $kT$-periodic solutions as critical points of the associated action functional and then trying to gain information about their minimal period by level estimates and/or Morse index estimates (see \cite{FoWi-89,SeTaTe-00} and the references therein). Even if we do not exclude that this approach could be successful in our setting, we find more convenient to adopt a completely different method, still relying on the variational structure of the equation but using a celebrated result of Hamiltonian dynamical systems theory: the Poincar\'{e}--Birkhoff fixed point theorem. 
More precisely, we follow the strategy introduced in \cite{BoZa-13} (to which we also refer for a general bibliography on the subject of subharmonic solutions, as well as for a description of the Poincar\'{e}--Birkhoff theorem and its applications in the theory of ODEs) and later refined in \cite{BoFe-18,BoOrZa-14}, dealing with second order equations like $u'' + h(t,u) = 0$ and finding subharmonic solutions once a $T$-periodic solution $u^{*}(t)$ is available. Roughly, the solution $u^{*}(t)$ is required to have a non-trivial Morse index and, under suitable conditions on the behaviour of $h(t,u)$ at infinity, for $k$ large enough many $kT$-periodic solutions $u_{k}(t)$ are found, with a precise number of zeros for $u_{k}(t) - u^{*}(t)$. Using this latter extra information (which intrinsically comes from the Poincar\'{e}--Birkhoff theorem) yields the minimality of the period for a large sub-family of the above $kT$-periodic solutions. As observed many times (compare again with the introduction in \cite{BoZa-13}) such a technique, when applicable, gives much sharper results, compared with the ones obtained with the variational strategy.

In our situation, the ``pivot'' $T$-periodic solution to be chosen is the ``small'' $T$-periodic solution given by degree theory, say $u_{s,\lambda}(t)$. Then, apart from some technical difficulties arising from the presence of the nonlinear differential operator, the key point of the argument consists in proving that the principal eigenvalue $\mu_{0}$ of the $T$-periodic Sturm--Liouville problem 
\begin{equation*}
\bigl{(} \varphi'(u'_{s,\lambda}(t)) w' \bigr{)}' + \bigl{(} \mu + \lambda a(t)g'(u_{s,\lambda}(t))\bigr{)} w = 0
\end{equation*}
is strictly negative when $\lambda$ is large enough. This in turn can be proved (via an algebraic trick inspired by the one introduced in \cite{BrHe-90} and already used in \cite{BoFe-18}) using in an essential way the asymptotic analysis for the solution $u_{s,\lambda}(t)$ as $\lambda \to +\infty$ developed in Section~\ref{section-4}.

Referring for instance to the model equations \eqref{eq-min1} and \eqref{eq-min2}, we have the following result (see Theorem~\ref{th-sub} for a more general version dealing with equation \eqref{eq-main}).

\begin{theorem}\label{th-1.2}
Let $a(t)$ be a sign-changing, continuous and periodic function with minimal period $T$, having a finite number of zeros in $\mathopen{[}0,T\mathclose{]}$ and satisfying condition $(a_{\#})$. Then, there exists $\Lambda^{*} \geq \lambda^{*}$ such that, for every $\lambda > \Lambda^{*}$, both equation \eqref{eq-min1} and equation \eqref{eq-min2} have (two positive $T$-periodic solutions and) a positive periodic solution of minimal period $kT$, for every integer $k$ large enough.
\end{theorem} 

\medskip

The plan of the paper is the following. In Section~\ref{section-2} we introduce the abstract degree setting and we prove two semi-abstract lemmas for the computation of the degree. By taking advantage of these results, in Section~\ref{section-3} we state and prove our results for the existence of positive $T$-periodic solutions; a non-existence result is proposed too. Section~\ref{section-4} deals with the asymptotic properties of the $T$-periodic solutions when $\lambda \to +\infty$. In Section~\ref{section-5} we finally prove the existence of positive subharmonic solutions, via the Poincar\'{e}--Birkhoff fixed point theorem. The paper ends with three appendices: in Appendix~\ref{appendix-A} we provide some maximum principles for equations driven by general $\varphi$-laplacian operators, in Appendix~\ref{appendix-B} we present a continuation theorem for the planar system \eqref{sys-intro}, in Appendix~\ref{appendix-C} we prove a dynamical characterization of the principal eigenvalue of Sturm--Liouville operators.

\subsection{Basics facts, main assumptions and notation}\label{section-1.1}

We conclude this introductory section with a list of notation and hypotheses that will be assumed along the paper.

We use the standard notation $\mathcal{C}^{k}$, $L^{p}$, $W^{k,p}$ (where $k$ is a non-negative integer and $1 \leq p \leq \infty$) for functions spaces, endowed with the usual norms; in all cases, a subscript $T$ means that the corresponding functions are $T$-periodic.  

Using the definition in \eqref{phi-L}, equation \eqref{eq-main} is written in a compact notation as 
\begin{equation*}
(\varphi(u'))' + \lambda a(t) g(u) = 0.
\leqno{(\mathscr{E}_{\lambda})}
\end{equation*}
Let us recall that $a\colon\mathbb{R}\to\mathbb{R}$ is a $T$-periodic and locally integrable weight, and $g\colon \mathbb{R}^{+} \to \mathbb{R}^{+}$ is a continuous function. Accordingly, solutions to $(\mathscr{E}_{\lambda})$ are meant in the Carath\'{e}odory sense, namely as continuously differentiable functions $u(t)$ such that the map $t \mapsto \varphi(u'(t))$ is locally absolutely continuous and the differential equation is satisfied almost everywhere. 
Actually, since $\varphi^{-1}$ is smooth, using the chain rule for Sobolev functions (see, for instance, \cite[Corollary~8.11]{Br-11}) it is easy to see that $u'(t)$ is locally absolutely continuous, with 
\begin{equation}\label{eq-usecondo}
u''(t) = - \lambda a(t) g(u(t)) (1-u'(t)^{2})^{\frac{3}{2}}, \quad \text{for a.e.~$t\in\mathbb{R}$.}
\end{equation}
In particular, a solution $u(t)$ is of class $\mathcal{C}^{2}$ whenever $a(t)$ is a continuous function.

We now collect, for the reader's convenience, the main assumptions on $a(t)$ and $g(u)$ which will be used along the paper. Besides the already mentioned 
\begin{itemize}[leftmargin=40pt,labelsep=22pt,itemsep=6pt]
\item [$(g_{*})$] $g(0) = 0$ and $g(u) > 0$ for $u>0$,
\item [$(a_{\#})$] $\displaystyle \int_{0}^{T} a(t) \,dt <0$, \vspace{3pt}
\end{itemize}
we always assume the following technical condition on the nodal behavior of $a(t)$:
\begin{itemize}[leftmargin=40pt,labelsep=22pt,itemsep=6pt]
\item [$(a_{*})$]
there exist $m\geq 1$ intervals $I^{+}_{1},\ldots,I^{+}_{m}$, closed and pairwise disjoint in the quotient space $\mathbb{R}/T\mathbb{Z}$, such that
\begin{align*}
&\qquad\qquad a(t)\geq0, \; \text{ for a.e.~$t\in I^{+}_{i}$,} \quad a(t)\not\equiv0 \; \text{ on $I^{+}_{i}$,} \quad \text{for $i=1,\ldots,m$,} \\
&\qquad\qquad a(t)\leq0, \; \text{ for a.e.~$t\in(\mathbb{R}/T\mathbb{Z})\setminus\bigcup_{i=1}^{m}I^{+}_{i}$.}
\end{align*}
\end{itemize}
Moreover, as for $g(u)$ we assume one of the following conditions near zero
\begin{itemize}[leftmargin=40pt,labelsep=22pt,itemsep=6pt]
\item [$(g_{0})$] $\displaystyle \lim_{u\to 0^{+}} \dfrac{g(u)}{u} = 0$ and $\displaystyle \lim_{\substack{u\to0^{+} \\ \omega\to1}}\dfrac{g(\omega u)}{g(u)}=1$,
\item [$(g_{0}')$] there exists $\varepsilon>0$ such that $g\in\mathcal{C}^{1}(\mathopen{[}0,\varepsilon\mathclose{[})$ and $g'(0)=0$,
\item [$(g_{0}'')$] there exist $p>1$ and $c_{p}\in \mathopen{]}0,+\infty \mathclose{[}$ such that $ \displaystyle \lim_{u\to0^{+}} \dfrac{g(u)}{u^{p}} = c_{p}$,
\end{itemize}
and one the following conditions at infinity
\begin{itemize}[leftmargin=40pt,labelsep=22pt,itemsep=6pt]
\item [$(g_{\infty})$] $\displaystyle\lim_{\substack{u\to+\infty \\ \omega\to1}}\dfrac{g(\omega u)}{g(u)}=1$,
\item [$(g_{\infty}')$] there exists $\varepsilon>0$ such that $g\in\mathcal{C}^{1}(\mathopen{]}1/\varepsilon,+\infty\mathclose{[})$ and $\displaystyle \lim_{u\to +\infty} \dfrac{g'(u)}{g(u)} = 0$.
\end{itemize}
Some comments about the above conditions are in order. 

Hypotheses $(g_{0})$ and $(g_{0}')$ are alternative assumptions for the behavior of $g(u)$ near zero appearing in our existence result (see Theorem~\ref{th-main-ex}); they have already been used in \cite{BoFeZa-16,FeZa-15ade}. They both imply a superlinear behavior for $g(u)$ as $u \to 0^{+}$ (that is, \eqref{super0} is satisfied). However, while $(g_{0}')$ requires $g(u)$ to be continuously differentiable in a neighborhood of zero, such a regularity condition is dropped in $(g_{0})$, at the expenses of a condition of so-called regular oscillation (that is, the second condition in $(g_{0})$). It can be seen that, in fact, $(g_{0})$ and $(g_{0}')$ are independent (see \cite[Remark~3.3]{FeZa-15ade}, to which we also refer for the notion of regularly oscillating functions and its applications in different contexts). When dealing with subharmonic solutions, we will need the slightly stronger condition $(g_{0}'')$, that is, we require a power-like behavior (of sublinear type) near zero. It is easy to see that $(g_{0}'')$ implies $(g_{0})$; moreover, of course, $(g_{0}'')$ implies $(g_{0}')$ when $g(u)$ is continuously differentiable in a neighborhood of zero.

As for the behavior of $g(u)$ at infinity, we assume either $(g_{\infty})$ or $(g_{\infty}')$. These conditions could appear quite unrelated at first sight; however, they cover most of the functions having a power-like behavior (of either sublinear, linear or superlinear type). Indeed, $(g_{\infty})$ simply requires a regularly oscillating behavior at infinity (and no growth assumptions at all; in particular, any function $g(u)$ satisfying $g(u) \sim K u^{\delta}$ for some $K, \delta > 0$ satisfies $(g_{\infty})$), while $(g_{\infty}')$ asks for the logarithmic derivative of $g(u)$ to be infinitesimal at infinity, a condition which appears pretty natural. It can be seen, however, that they are independent, even when $g(u)$ is of class $\mathcal{C}^{1}$: for instance, the function $g(u) = 2u^{3} + \sin(u^{2})u^{2}$ satisfies $(g_{\infty})$ but not $(g_{\infty}')$, while the function $g(u) = e^{-\sqrt{u}}/u$ satisfies $(g_{\infty}')$ but not $(g_{\infty})$. 

It is worth observing that the model nonlinearities considered in equations \eqref{eq-min1} and \eqref{eq-min2} satisfy all the above assumptions.

In the sequel, it is convenient to extend the nonlinearity $\lambda a(t) g(u)$ appearing in $(\mathscr{E}_{\lambda})$ to the whole real line, by setting
\begin{equation}\label{def-fl}
f_{\lambda}(t,u) := 
\begin{cases}
\, -u, & \text{if $u\leq 0$,} \\
\, \lambda a(t) g(u), & \text{if $u\geq 0$.}
\end{cases}
\end{equation}
In this manner, by the weak maximum principle in Corollary~\ref{cor-weak-max-principle}, any $T$-periodic solution to the equation 
\begin{equation}\label{eq-fl}
(\varphi(u'))' + f_{\lambda}(t,u) = 0
\end{equation} 
is non-negative, thus solving $(\mathscr{E}_{\lambda})$.

\section{Abstract degree setting}\label{section-2}

In this section we introduce the abstract setting of the coincidence degree, in order to deal with the $T$-periodic problem associated with $(\mathscr{E}_{\lambda})$. In what follows, we just assume $a\in L^{1}_{T}$ and $g\in\mathcal{C}(\mathbb{R}^{+})$ with $g(0)=0$.

We proceed in the same spirit of \cite{FeZa-17tmna}, by writing equation \eqref{eq-fl}
as the equivalent first order system
\begin{equation}\label{system}
\begin{cases}
\, x_{1}' = \varphi^{-1}(x_{2}) \\
\, x_{2}' = -f_{\lambda}(t,x_{1}).
\end{cases}
\end{equation}
In what follows, by a $T$-periodic solution to \eqref{system} we mean a vector function $x=(x_{1},x_{2})$ such that, for $i=1,2$, $x_{i}\colon \mathopen{[}0,T\mathclose{]} \to \mathbb{R}$ is an absolutely continuous function (i.e., $x_{i} \in W^{1,1}(\mathopen{[}0,T\mathclose{]}))$ satisfying \eqref{system} for a.e.~$t\in\mathopen{[}0,T\mathclose{]}$ and such that $x_{i}(0)=x_{i}(T)$. As is well known, any $T$-periodic solution to \eqref{system} according to the above definition can be extended to the whole real line to a locally absolutely continuous solution to \eqref{system} such that $x(t+T) = x(t)$ for all $t\in\mathbb{R}$. 

For $i=1,2$, let $X_{i}:= \mathcal{C}(\mathopen{[}0,T\mathclose{]})$ be the Banach space of continuous functions $x_{i} \colon \mathopen{[}0,T\mathclose{]} \to \mathbb{R}$,
endowed with the $\sup$-norm $\|x_{i}\|_{\infty} := \max_{t\in \mathopen{[}0,T\mathclose{]}} |x_{i}(t)|$,
and let $Z_{i}:=L^{1}(\mathopen{[}0,T\mathclose{]})$ be the Banach space of Lebesgue integrable functions $z_{i} \colon \mathopen{[}0,T\mathclose{]} \to \mathbb{R}$, endowed with the norm $\|z_{i}\|_{L^{1}_{T}}:= \int_{0}^{T} |z_{i}(t)|\,dt$.
Next, we define the Banach spaces $X := X_{1}\times X_{2} = \mathcal{C}(\mathopen{[}0,T\mathclose{]},\mathbb{R}^{2})$ and $Z := Z_{1}\times Z_{2} = L^{1}(\mathopen{[}0,T\mathclose{]},\mathbb{R}^{2})$ endowed with the standard norms.

For $i=1,2$, we consider the linear differential operator $L_{i}\colon \mathrm{dom}\,L_{i} \to Z_{i}$ defined as
\begin{equation*}
(L_{i}x_{i})(t):= x'_{i}(t), \quad t\in\mathopen{[}0,T\mathclose{]},
\end{equation*}
where $\mathrm{dom}\,L_{i} := \{ x_{i} \in W^{1,1}(\mathopen{[}0,T\mathclose{]}) \colon x_{i}(0) = x_{i}(T)\} \subseteq X_{i}$.
As is well known, $L_{i}$ is a Fredholm map of index zero, $\ker L_{i}$ and $\mathrm{coker}\,L_{i}$ are made up of constant functions and
\begin{equation*}
\mathrm{Im}\,L_{i} = \biggl{\{} z_{i}\in Z_{i} \colon \int_{0}^{T} z_{i}(t)\,dt = 0 \biggr{\}}.
\end{equation*}
We define the projectors $P_{i} \colon X_{i} \to \ker L_{i}$ and $Q_{i} \colon Z_{i} \to \mathrm{coker}\,L_{i}$ as the average operators
\begin{equation*}
P_{i}x_{i} = Q_{i}x_{i} := \dfrac{1}{T}\int_{0}^{T} x_{i}(t)~\!dt.
\end{equation*}
Let $K_{i} \colon \mathrm{Im}\,L_{i} \to \mathrm{dom}\,L_{i} \cap \ker P_{i}$ be the right inverse of $L_{i}$, namely, given $z_{i}\in Z_{i}$ with $\int_{0}^{T} z_{i}(t)~\!dt =0$, then $x_{i} = K_{i} z_{i}$ is the unique solution of
\begin{equation*}
x_{i}'= z_{i}(t), \quad \text{with $\int_{0}^{T} x_{i}(t)~\!dt = 0$,}
\end{equation*}
which clearly satisfies the boundary condition $x_{i}(0) = x_{i}(T)$. Next, as linear orientation-preserving isomorphism $J_{i} \colon \mathrm{coker}\,L_{i} \to \ker L_{i}$ we take the identity map in $\mathbb{R}$.

Now we define the nonlinear operator $N_{\lambda} \colon X \to Z$ as
the Nemytskii operator induced by the functions $\varphi^{-1}$ and $-f_{\lambda}$, namely
the operator $N_{\lambda}$ has the following components
\begin{equation*}
\begin{cases}
\, N_{\lambda,1}(x_{2})(t):= \varphi^{-1}(x_{2}(t)), \\
\, N_{\lambda,2}(x_{1})(t) := - f_{\lambda}(t,x_{1}(t)),
\end{cases}
\end{equation*}
where $x=(x_{1},x_{2})\in X$ and $t\in\mathopen{[}0,T\mathclose{]}$.
We observe that $N_{\lambda}$ is an $L$-completely continuous operator. Indeed, setting $Q:=(Q_{1},Q_{2})$ and $K:=(K_{1},K_{2})$, we have that $N_{\lambda}$ and $K(\mathrm{Id}_{Z}-Q)N_{\lambda}$ are continuous and $QN_{\lambda}(B)$ and $K(\mathrm{Id}_{Z}-Q)N_{\lambda}(B)$ are relatively compact sets, for each bounded set $B\subseteq X$.

With that position and setting $L :=(L_{1},L_{2})\colon \mathrm{dom}\,L \to Z$ with $\mathrm{dom}\,L:=\mathrm{dom}\,L_{1}\times\mathrm{dom}\,L_{2}$, system \eqref{system} can be written as a coincidence equation
\begin{equation*}
Lx = N_{\lambda}x,\quad x\in \mathrm{dom}\,L,
\end{equation*}
From Mawhin's coincidence degree theory, one can see that the coincidence equation is equivalent to the fixed point problem
\begin{equation*}
x = \Phi_{\lambda}(x), \quad x\in X,
\end{equation*}
where $\Phi_{\lambda} \colon X \to X$ is defined as $\Phi_{\lambda}=(\Phi_{\lambda,1},\Phi_{\lambda,2})$, with $\Phi_{\lambda,i} \colon X \to X_{i}$ for $i=1,2$ given by
\begin{equation*}
\Phi_{\lambda,i}(x):= P_{i}x_{i} + J_{i}Q_{i}N_{\lambda,i}x + K_{i}(\mathrm{Id}_{Z_{i}}-Q_{i})N_{\lambda,i}x, \quad x=(x_{1},x_{2})\in X.
\end{equation*}
We stress that, under the above assumptions, $\Phi_{\lambda}$ is a completely continuous operator.

We can now give the definition of coincidence degree. Let $\Omega\subseteq X$ be an open (possibly unbounded) set such that the solution set
\begin{equation*}
\mathrm{Fix}\,(\Phi_{\lambda},\Omega):= \bigl{\{}x\in {\Omega} \colon x =
\Phi_{\lambda}x \bigr{\}} = \bigl{\{}x\in {\Omega}\cap \mathrm{dom}\,L \colon Lx = N_{\lambda}x\bigr{\}}
\end{equation*}
is compact. The \textit{coincidence degree} $\mathrm{D}_{L}(L-N_{\lambda},\Omega)$ of $L$ and $N_{\lambda}$ in $\Omega$ is defined as
\begin{equation*}
\mathrm{D}_{L}(L-N_{\lambda},\Omega):= \mathrm{deg}_{\mathrm{LS}}(\mathrm{Id}_{X} - \Phi_{\lambda},\Omega,0),
\end{equation*}
where $\mathrm{deg}_{\mathrm{LS}}$ denotes the extension of the Leray--Schauder degree
for locally compact maps defined on open (possibly unbounded) sets (cf.~\cite{Gr-72,Ma-99,Nu-85,Nu-93}). More precisely, given an open bounded set $\mathcal{V}$ with
\begin{equation*}
\mathrm{Fix}\,(\Phi_{\lambda},\Omega) \subseteq \mathcal{V} \subseteq \overline{\mathcal{V}} \subseteq \Omega,
\end{equation*}
we can define
\begin{equation*}
\mathrm{deg}_{\mathrm{LS}}(\mathrm{Id}_{X} - \Phi_{\lambda},\Omega,0):=\mathrm{deg}_{\mathrm{LS}}(\mathrm{Id}_{X} - \Phi_{\lambda},\mathcal{V},0)
\end{equation*}
(the definition is independent of the choice of $\mathcal{V}$).
From the main properties of the Leray--Schauder topological degree, one can derive the corresponding ones for the extension of the coincidence degree.
\begin{itemize}
\item \textit{Additivity.}
Let $\Omega_{1}$, $\Omega_{2}$ be open and disjoint subsets of $\Omega$ such that $\mathrm{Fix}\,(\Phi_{\lambda},\Omega)\subseteq \Omega_{1}\cup\Omega_{2}$.
Then
\begin{equation*}
\mathrm{D}_{L}(L-N_{\lambda},\Omega) = \mathrm{D}_{L}(L-N_{\lambda},\Omega_{1})+ \mathrm{D}_{L}(L-N_{\lambda},\Omega_{2}).
\end{equation*}
\item \textit{Excision.}
Let $\Omega_{0}$ be an open subset of $\Omega$ such that $\mathrm{Fix}\,(\Phi_{\lambda},\Omega)\subseteq \Omega_{0}$.
Then
\begin{equation*}
\mathrm{D}_{L}(L-N_{\lambda},\Omega)=\mathrm{D}_{L}(L-N_{\lambda},\Omega_{0}).
\end{equation*}
\item \textit{Existence theorem.}
If $\mathrm{D}_{L}(L-N_{\lambda},\Omega)\neq0$, then $\mathrm{Fix}\,(\Phi_{\lambda},\Omega)\neq\emptyset$,
hence there exists $u\in {\Omega}\cap \mathrm{dom}\,L$ such that $Lu = N_{\lambda}u$.
\item \textit{Homotopic invariance.}
Let $H\colon\mathopen{[}0,1\mathclose{]}\times \Omega \to X$, $H_{\vartheta}(u) := H(\vartheta,u)$, be a continuous homotopy such that
\begin{equation*}
\Sigma:=\bigcup_{\vartheta\in\mathopen{[}0,1\mathclose{]}} \bigl{\{}u\in \Omega\cap \mathrm{dom}\,L \colon Lu=H_{\vartheta}u\bigr{\}}
\end{equation*}
is a compact set and there exists an open neighborhood $\mathcal{W}$ of $\Sigma$ such that $\overline{\mathcal{W}}\subseteq \Omega$ and
$(K(\mathrm{Id}_{Z}-Q)H)|_{\mathopen{[}0,1\mathclose{]}\times\overline{\mathcal{W}}}$ is a compact map.
Then the map $\vartheta\mapsto \mathrm{D}_{L}(L-H_{\vartheta},\Omega)$ is constant on $\mathopen{[}0,1\mathclose{]}$.
\end{itemize}
For more details and proofs, we refer the reader to \cite{BoFeZa-16,Fe-18,GaMa-77,Ma-79,Ma-93} and the references therein.

In our applications, given an open set $\mathcal{O}\subseteq X$ and an $L$-completely continuous operator $\mathcal{N}$, in order to prove that $\mathrm{D}_{L}(L-\mathcal{N},\mathcal{O})$ is well-defined or equivalently that the set $\bigl{\{}x\in \mathcal{O}\cap \mathrm{dom}\,L \colon Lx = \mathcal{N}x\bigr{\}}$ is compact in $X$, we proceed in this manner. We consider the set $\{x\in \overline{\mathcal{O}}\cap \mathrm{dom}\,L\colon Lx = {\mathcal N} x\}$ and we prove that it is bounded (and thus compact, by the $L$-complete continuity of $\mathcal{N}$) and disjoint from $\partial\mathcal{O}$. When dealing with homotopies, the strategy is analogous.

\medskip

We are now in a position to present two lemmas that allow us to compute the degree on open and unbounded sets of $X$ of the form
\begin{equation}\label{eq-def-omega}
\Omega_{d} := B(0,d) \times \mathcal{C}(\mathopen{[}0,T\mathclose{]}) \subseteq X,
\end{equation}
where $B(0,d)\subseteq X_{1}$ is the open ball centered at zero and with radius $d>0$.

The first one, relying on an extension of the classical Mawhin's continuation theorem illustrated in Appendix~\ref{appendix-B}, is a criteria for non-zero degree in $\Omega_{d}$.

\begin{lemma}\label{lem-deg1}
Let $a \colon \mathbb{R} \to \mathbb{R}$ be a locally integrable $T$-periodic function satisfying $(a_{\#})$, let $g \colon \mathbb{R}^{+} \to \mathbb{R}^{+}$ be a continuous function satisfying $(g_{*})$, and $\lambda>0$. Let $d>0$ and assume that the following property holds.
\begin{itemize}
\item[$(H_{1})$]
If $\vartheta\in \mathopen{]}0,1\mathclose{]}$ and $u(t)$ is a non-negative $T$-periodic solution to
\begin{equation}\label{eq-lem-deg1}
(\varphi(u'))' + \vartheta \lambda a(t) g(u) = 0,
\end{equation}
then $\|u\|_{\infty} \neq d$.
\end{itemize}
Then, it holds that
\begin{equation*}
\mathrm{D}_{L}(L-N_{\lambda},\Omega_{d}) = -1.
\end{equation*}
\end{lemma}

\begin{proof}
We study the equation
\begin{equation}\label{eq-2.ftheta}
(\varphi(u'))' + \vartheta f_{\lambda}(t,u) = 0,
\end{equation}
for $\vartheta\in\mathopen{]}0,1\mathclose{]}$, which can be equivalently written as the system
\begin{equation}\label{system-theta}
\begin{cases}
\, x_{1}' = \varphi^{-1}(x_{2}) \\
\, x_{2}' = -\vartheta f_{\lambda}(t,x_{1})
\end{cases}
\end{equation}
We are going to prove hypotheses $(i)$ and $(ii)$ of Theorem~\ref{th-B.1}.

Let $\vartheta\in\mathopen{]}0,1\mathclose{]}$ and let $\mathcal{S}_{\vartheta}$ be the set of $T$-periodic solutions to \eqref{system-theta} in $\overline{\Omega_{d}}$. We notice that $x = (x_{1},x_{2})\in \mathcal{S}_{\vartheta}$ if and only if $x_{1}(t)$ is a $T$-periodic solution of \eqref{eq-2.ftheta} such that $\|x_{1}\|_{\infty}\leq d$ and $x_{2}(t)=\varphi(x_{1}'(t))$ for all $t\in\mathopen{[}0,T\mathclose{]}$. By the maximum principle in Corollary~\ref{cor-weak-max-principle}, we find that $x_{1}(t)\geq 0$ for all $t\in \mathbb{R}$ and, indeed, $x_{1}(t)$ solves \eqref{eq-lem-deg1}. Next, for all $t\in\mathopen{[}0,T\mathclose{]}$, it holds that
\begin{equation*}
|x_{2}(t)| = |x_{2}(t_{0})| + \biggl{|} \int_{t_{0}}^{t} x_{2}'(\xi) \,d\xi \biggr{|} \leq \lambda \|a\|_{L^{1}_{T}} \max_{u\in\mathopen{[}0,d\mathclose{]}} g(u), \quad \text{for all $t\in\mathopen{[}0,T\mathclose{]}$,}
\end{equation*}
by choosing $t_{0}\in\mathopen{[}0,T\mathclose{]}$ such that $x_{1}'(t_{0})=0$ and so $x_{2}(t_{0})=\varphi(x_{1}'(t_{0}))=0$.
As a consequence, for any $\vartheta\in\mathopen{]}0,1\mathclose{]}$,
\begin{equation*}
\mathcal{S}_{\vartheta} \subseteq B := \biggl{\{} x = (x_{1},x_{2}) \in \overline{\Omega_{d}}\colon \|x_{1}\|_{\infty}\leq d, \, \| x_{2} \|_{\infty}
 \leq \lambda \|a\|_{L^{1}_{T}} \max_{u\in\mathopen{[}0,d\mathclose{]}} g(u) \biggr{\}}.
\end{equation*}
Moreover, via condition $(H_{1})$, we deduce that $\mathcal{S}_{\vartheta}\subseteq \Omega_{d}$.
Therefore, condition $(i)$ of Theorem~\ref{th-B.1} holds.

Secondly, we analyse the validity of condition $(ii)$ of Theorem~\ref{th-B.1}. We notice that the set $B(0,d)\cap\mathbb{R}=\mathopen{]}-d,d\mathclose{[}$ is bounded and so condition $(ii)$ can be reduced to verify that $f^{\#}_{\lambda}$ has no zeros on $\partial(B(0,d)\cap\mathbb{R})=\{\pm d\}$.
Now, from
\begin{equation*}
f_{\lambda}^{\#}(s) = \dfrac{1}{T}\int_{0}^{T} f_{\lambda}(t,s)~\!dt =
\begin{cases}
\, -s, & \text{if $s \leq 0$,} \\
\, \lambda \biggl{(}\dfrac{1}{T} \displaystyle \int_{0}^{T} a(t) \,dt \biggr{)} g(s), & \text{if $s \geq 0$,}
\end{cases}
\end{equation*}
and hypothesis $(a_{\#})$, we have that $f^{\#}_{\lambda}(s)s < 0$ for $s \neq 0$ and immediately $(ii)$ holds.

Now, observing that $f_{\lambda}^{\#}(d)<0<f_{\lambda}^{\#}(-d)$, an application of Theorem~\ref{th-B.1} yields
\begin{equation*}
\mathrm{D}_{L}(L-N_{\lambda},\Omega_{d}) = \mathrm{deg}_{\mathrm{B}}(f^{\#}_{\lambda},\mathopen{]}-d,d\mathclose{[},0) = -1.
\end{equation*}
This concludes the proof.
\end{proof}

The second lemma states a criteria for zero degree in $\Omega_{d}$ (notice that, contrarily to Lemma~\ref{lem-deg1}, here $(a_{\#})$ and $(g_{*})$ are not needed). 

\begin{lemma}\label{lem-deg0}
Let $a \colon \mathbb{R} \to \mathbb{R}$ be a locally integrable $T$-periodic function, let $g \colon \mathbb{R}^{+} \to \mathbb{R}^{+}$ be a continuous function satisfying $g(0) = 0$, and $\lambda>0$. 
Let $d>0$ and assume that there exists $v \in L^{1}(\mathopen{[}0,T\mathclose{]})$, with $v\not\equiv0$, such that the following properties hold.
\begin{itemize}
\item[$(H_{2})$]
If $\alpha \geq 0$ and $u(t)$ is a non-negative $T$-periodic solution to
\begin{equation}\label{eq-lem-deg0}
(\varphi(u'))' + \lambda a(t) g(u) + \alpha v(t) = 0,
\end{equation}
then $\|u\|_{\infty}\neq d$.
\item[$(H_{3})$]
There exists $\alpha_{0} \geq 0$ such that equation \eqref{eq-lem-deg0}, with $\alpha=\alpha_{0}$, has no non-negative $T$-periodic solutions $u(t)$ with $\|u\|_{\infty}\leq d$.
\end{itemize}
Then, it holds that
\begin{equation*}
\mathrm{D}_{L}(L-N_{\lambda},\Omega_{d}) = 0.
\end{equation*}
\end{lemma}

\begin{proof}
We study the equation
\begin{equation}\label{eq-2.falpha}
(\varphi(u'))' + f_{\lambda}(t,u) + \alpha v = 0,
\end{equation}
for $\alpha \geq 0$, which can be equivalently written as the system
\begin{equation*}
\begin{cases}
\, x_{1}' = \varphi^{-1}(x_{2}) \\
\, x_{2}' = - f_{\lambda}(t,u) - \alpha v
\end{cases}
\end{equation*}
and as a coincidence equation in the space $X$
\begin{equation*}
Lx = N_{\lambda}x + \alpha V, \quad x=(x_{1},x_{2})\in \mathrm{dom}\,L,
\end{equation*}
where $V\in X$ is the function $V(t) := (0,-v(t))$ for all $t \in \mathopen{[}0,T\mathclose{]}$.

As a first step, we check that $\mathrm{D}_{L}(L- N_{\lambda} -\alpha V,\Omega_{d})$ is well-defined for any $\alpha \geq 0$. To this aim, suppose that $\alpha\geq 0$ is fixed and consider the set
\begin{equation*}
\mathcal{S}_{\alpha}
:=\bigl{\{}x\in \overline{\Omega_{d}} \cap \mathrm{dom}\,L \colon Lx = N_{\lambda}x + \alpha V\bigr{\}} 
= \bigl{\{}x\in \overline{\Omega_{d}} \colon x = \Phi_{\lambda}x + \alpha V\bigr{\}}.
\end{equation*}
We have that $x = (x_{1},x_{2}) \in \mathcal{S}_{\alpha}$ if and only if $x_{1}(t)$ is a $T$-periodic solution of \eqref{eq-2.falpha} such that $\|x_{1}\|_{\infty}\leq d$ and $x_{2}(t)=\varphi(x_{1}'(t))$ for all $t\in\mathopen{[}0,T\mathclose{]}$. By the maximum principle in Corollary~\ref{cor-weak-max-principle}, we find that $x_{1}(t)\geq 0$ for all $t\in \mathbb{R}$ and, indeed, $x_{1}(t)$ solves \eqref{eq-lem-deg0}. Next, for all $t\in\mathopen{[}0,T\mathclose{]}$, it holds that
\begin{equation*}
|x_{2}(t)| = |x_{2}(t_{0})| + \biggl{|} \int_{t_{0}}^{t} x_{2}'(\xi) \,d\xi \biggr{|} \leq \lambda \|a\|_{L^{1}_{T}} \max_{u\in\mathopen{[}0,d\mathclose{]}} g(u) + \alpha \|v\|_{L^{1}_{T}},
\end{equation*}
by choosing $t_{0}\in\mathopen{[}0,T\mathclose{]}$ such that $x_{1}'(t_{0})=0$ and so $x_{2}(t_{0})=\varphi(x_{1}'(t_{0}))=0$.
As a consequence, $\mathcal{S}_{\alpha}$ is bounded and so the complete continuity of the operator $\Phi_{\lambda}$ ensures the compactness of $\mathcal{S}_{\alpha}$. 
Moreover, via condition $(H_{2})$, we deduce that $\mathcal{S}_{\alpha}\subseteq \Omega_{d}$.
In this manner we have proved that the coincidence degree
$\mathrm{D}_{L}(L-N_{\lambda} - \alpha V,\Omega_{d})$ is well-defined for any $\alpha\geq 0$.

Now, using $\alpha$ as a homotopy parameter, by repeating the same argument as above, we find that the set
\begin{equation*}
\begin{aligned}
\mathcal{S}
:= \bigcup_{\alpha\in \mathopen{[}0,\alpha_{0}\mathclose{]}} \mathcal{S}_{\alpha} \,
&= \bigcup_{\alpha\in \mathopen{[}0,\alpha_{0}\mathclose{]}} \bigl{\{}u\in \overline{\Omega_{d}} \cap \mathrm{dom}\,L \colon Lu = N_{\lambda}u + \alpha V\bigr{\}}
\\ &= \bigcup_{\alpha\in \mathopen{[}0,\alpha_{0}\mathclose{]}} \bigl{\{}u\in \overline{\Omega_{d}} \colon u = \Phi_{\lambda}u + \alpha V\bigr{\}}
\end{aligned}
\end{equation*}
is a compact subset of $\Omega_{d}$. Hence, by the homotopic invariance property of the coincidence degree and condition $(H_{3})$, we have that
\begin{equation*}
\mathrm{D}_{L}(L-N_{\lambda},\Omega_{d}) = \mathrm{D}_{L}(L-N_{\lambda} - \alpha_{0} V,\Omega_{d}) = 0.
\end{equation*}
This concludes the proof.
\end{proof}

\begin{remark}\label{rem-2.1}
It is worth noticing that the discussion in this section (in particular, Lemma~\ref{lem-deg1} and Lemma~\ref{lem-deg0}) still holds true when 
$\varphi \colon I \to \mathbb{R}$ is an increasing homeomorphism defined on an open interval $I \subseteq \mathbb{R}$ containing $0$, with $\varphi(0) = 0$.
$\hfill\lhd$
\end{remark}

\section{$T$-periodic solutions: existence and non-existence}\label{section-3}

In this section we state and prove our existence/non-existence results for positive $T$-periodic solutions to $(\mathscr{E}_{\lambda})$.
In particular, as already described in the introduction, we are going to provide a two-solution theorem, giving a ``small'' solution and a ``large'' one for $\lambda$ sufficiently large, and a non-existence result for $\lambda$ sufficiently small. 

\begin{theorem}\label{th-main-ex}
Let $a \colon \mathbb{R} \to \mathbb{R}$ be a locally integrable $T$-periodic function satisfying $(a_{\#})$ and $(a_{*})$. 
Let $g \colon {\mathbb{R}}^{+} \to {\mathbb{R}}^{+}$ be a continuous function satisfying $(g_{*})$.
Moreover, suppose that either $(g_{0})$ or $(g_{0}')$ holds, and either $(g_{\infty})$ or $(g_{\infty}')$ holds.
Then, there exists $\lambda^{*}>0$ such that for every $\lambda>\lambda^{*}$ there exist at least two positive $T$-periodic solutions to $(\mathscr{E}_{\lambda})$.

More precisely, there exist $\rho^{*}>0$ and $\lambda^{*}>0$ such that for every $\lambda>\lambda^{*}$ there exist two positive $T$-periodic solutions $u_{s}(t)$ and $u_{\ell}(t)$ to $(\mathscr{E}_{\lambda})$ such that $\|u_{s}\|_{\infty}<\rho^{*}<\|u_{\ell}\|_{\infty}$.
\end{theorem}

\begin{theorem}\label{th-main-nex}
Let $a \colon \mathbb{R} \to \mathbb{R}$ be a locally integrable $T$-periodic function satisfying $(a_{\#})$ and $(a_{*})$. 
Let $g \colon {\mathbb{R}}^{+} \to {\mathbb{R}}^{+}$ be a continuously differentiable function satisfying $(g_{*})$ and
\begin{equation*}
\limsup_{u\to +\infty} \dfrac{|g'(u)|}{g(u)^{\eta}} < +\infty, \quad \text{for some $\eta\in\mathopen{[}0,1\mathclose{[}$}.
\leqno{(g_{\infty}'')}
\end{equation*}
Then, there exists $\lambda_{*}>0$ such that for every $\lambda\in\mathopen{]}0,\lambda_{*}\mathclose{[}$ there are no positive $T$-periodic solutions to $(\mathscr{E}_{\lambda})$.
\end{theorem}

It is easy to check that Theorem~\ref{th-1.1} of the introduction can be deduced as a corollary of Theorem~\ref{th-main-ex} and Theorem~\ref{th-main-nex}.

Variants of Theorem~\ref{th-main-ex}, guaranteeing either the existence of a ``small'' solution or the existence of a ``large'' solution, can be given. Precisely, the following results hold true. 

\begin{theorem}\label{th-main-ex-small}
Let $a \colon \mathbb{R} \to \mathbb{R}$ be a locally integrable $T$-periodic function satisfying $(a_{\#})$ and $(a_{*})$. 
Let $g \colon {\mathbb{R}}^{+} \to {\mathbb{R}}^{+}$ be a continuous function satisfying $(g_{*})$.
Moreover, suppose that either $(g_{0})$ or $(g_{0}')$ holds.
Then, there exists $\lambda^{*}>0$ such that for every $\lambda>\lambda^{*}$ there exists at least a positive $T$-periodic solution to $(\mathscr{E}_{\lambda})$.

More precisely, there exist $\rho^{*}>0$ and $\lambda^{*}>0$ such that for every $\lambda>\lambda^{*}$ there exists a positive $T$-periodic solution $u_{s}(t)$ to $(\mathscr{E}_{\lambda})$ such that $\|u_{s}\|_{\infty}<\rho^{*}$.
\end{theorem}

\begin{theorem}\label{th-main-ex-large}
Let $a \colon \mathbb{R} \to \mathbb{R}$ be a locally integrable $T$-periodic function satisfying $(a_{\#})$ and $(a_{*})$. 
Let $g \colon {\mathbb{R}}^{+} \to {\mathbb{R}}^{+}$ be a continuous function satisfying $(g_{*})$ and
\begin{equation*}
\limsup_{u\to 0^{+}} \dfrac{g(u)}{u} < +\infty.
\end{equation*}
Moreover, suppose that either $(g_{\infty})$ or $(g_{\infty}')$ holds.
Then, there exists $\lambda^{*}>0$ such that for every $\lambda>\lambda^{*}$ there exists at least a positive $T$-periodic solution to $(\mathscr{E}_{\lambda})$.

More precisely, there exist $\rho^{*}>0$ and $\lambda^{*}>0$ such that for every $\lambda>\lambda^{*}$ there exists a positive $T$-periodic solution $u_{\ell}(t)$ to $(\mathscr{E}_{\lambda})$ such that $\|u_{\ell}\|_{\infty}>\rho^{*}$.
\end{theorem}

We stress that in Theorem~\ref{th-main-ex-small} no assumptions for $g(u)$ at infinity are required (actually, by a careful reading of the proof, one can see that $g(u)$ could be defined only on a right neighborhood of zero). On the other hand, in Theorem~\ref{th-main-ex-large} the ratio $g(u)/u$ is assumed to be bounded for $u$ near $0$: this is required to guarantee the positivity of the solution, via a strong maximum principle (cf.~Theorem~\ref{strong-max-principle}). If such a condition is dropped, only a non-trivial non-negative $T$-periodic solution can be found.

The rest of the section is divided into two subsections. In the first one, Section~\ref{section-3.1}, we state and prove the technical lemmas needed for the application of the abstract degree lemmas of Section~\ref{section-2}. From these results the proofs of Theorem~\ref{th-main-ex}, Theorem~\ref{th-main-ex-small} and Theorem~\ref{th-main-ex-large} easily follow and are given in Section~\ref{section-3.2}, together with the proof of Theorem~\ref{th-main-nex}.

\subsection{Technical lemmas}\label{section-3.1}

This section is devoted to some technical lemmas. Notice that each of them is given under the minimal set of assumptions for $a(t)$ and $g(u)$.

The first lemma will be used for fixing the constants $\rho^{*}$ and $\lambda^{*}$ appearing in the statement of Theorem~\ref{th-main-ex} and for the computation of the topological degree in $\Omega_{\rho^{*}}$ via Lemma~\ref{lem-deg0}.

\begin{lemma}\label{lem-rho}
Let $a \colon \mathbb{R} \to \mathbb{R}$ be a locally integrable $T$-periodic function satisfying $(a_{*})$.
Let $g \colon {\mathbb{R}}^{+} \to {\mathbb{R}}^{+}$ be a continuous function satisfying $(g_{*})$.
There exist $\rho^{*}> 0$ and $\lambda^{*}> 0$ such that, for every $\lambda > \lambda^{*}$, $\alpha \geq 0$ and $i \in \{1,\ldots,m\}$, there are no non-negative solutions $u(t)$ to
\begin{equation}\label{eq-lem-rho}
(\varphi(u'))' + \lambda a(t)g(u) + \alpha = 0,
\end{equation}
with $u(t)$ defined for all $t \in I^{+}_{i}$ and such that $\max_{t\in I^{+}_{i}} u(t) = \rho^{*}$.
\end{lemma}

\begin{proof}
Let us fix $\rho^{*}>0$ such that
\begin{equation*}
\rho^{*}<\dfrac{ |I^{+}_{i} |}{4} \quad \text{ and } \quad \int_{\sigma_{i}+2\rho^{*}}^{\tau_{i}-2\rho^{*}} a(t)\,dt > 0, \quad \text{for every $i = 1,\ldots,m$,}
\end{equation*}
where we have set $I^{+}_{i}=\mathopen{[}\sigma_{i},\tau_{i}\mathclose{]}$, and, accordingly, let us define
\begin{equation*}
\lambda^{*} := \max_{i=1,\ldots,m}\dfrac{2 \varphi(1/2)}{ \min \bigl{\{} g(u) \colon u \in \mathopen{[}2(\rho^{*})^{2}/ |I^{+}_{i} |,\rho^{*} \mathclose{]} \bigr{\}} \displaystyle{\int_{\sigma_{i}+2\rho^{*}}^{\tau_{i}-2\rho^{*}} a(t)\,dt}}.
\end{equation*}

Let $\lambda > \lambda^{*}$, $\alpha \geq 0$ and $i \in \{1,\ldots,m\}$. We suppose by contradiction that there exists a non-negative solution $u(t)$ to \eqref{eq-lem-rho}, defined on $I^{+}_{i}$ and such that $\max_{t \in I^{+}_{i}} u(t) = \rho^{*}$. 
We first observe that, since $u(t)$ is concave on $I^{+}_{i}$, we have the estimate (cf.~\cite[p.~807]{BoFeZa-18})
\begin{equation*}
u(t) \geq \frac{\rho^{*}}{|I^{+}_{i}|}\min\{ t - \sigma_{i}, \tau_{i} - t\}, \quad \text{for all $t \in I^{+}_{i}$,}
\end{equation*}
and, consequently, 
\begin{equation*}
u(t) \geq \frac{2(\rho^{*})^{2}}{| I^{+}_{i} |}, \quad \text{for all $t \in \mathopen{[}\sigma_{i}+2\rho^{*},\tau_{i}-2\rho^{*}\mathclose{]}$.}
\end{equation*}
Moreover, we claim that
\begin{equation}\label{eq-3.2}
|u'(t)| \leq \dfrac{u(t)}{2\rho^{*}}, \quad \text{for all $t \in \mathopen{[}\sigma_{i}+2\rho^{*},\tau_{i}-2\rho^{*}\mathclose{]}$.}
\end{equation}
To prove this, let us fix $t\in \mathopen{[}\sigma_{i}+2\rho^{*},\tau_{i}-2\rho^{*}\mathclose{]}$. The result is trivially true if $u'(t) = 0$. If $u'(t) > 0$, again by the concavity of $u(t)$, we obtain
\begin{equation*}
u(t) \geq u(t) - u(\sigma_{i}) = \int_{\sigma_{i}}^{t} u'(\xi)\,d\xi \geq 2\rho^{*} u'(t), \quad \text{for all $t \in \mathopen{[}\sigma_{i}+2\rho^{*},\tau_{i}\mathclose{]}$.}
\end{equation*}
In the case $u'(t) < 0$, the argument is analogous and therefore \eqref{eq-3.2} follows. In particular, since $u(t) \leq \rho^{*}$ we immediately deduce
\begin{equation*}
|u'(t)| \leq \dfrac{1}{2}, \quad \text{for all $t \in \mathopen{[}\sigma_{i}+2\rho^{*},\tau_{i}-2\rho^{*}\mathclose{]}$.}
\end{equation*}
Integrating equation \eqref{eq-lem-rho} on $\mathopen{[}\sigma_{i}+2\rho^{*},\tau_{i}-2\rho^{*}\mathclose{]}$ and using the fact that 
$\varphi$ is odd, we thus obtain
\begin{equation*}
\lambda \min \biggl{\{} g(u) \colon u \in \biggl{[}\frac{2 (\rho^{*})^{2}}{| I^{+}_{i} |},\rho^{*} \biggr{]}\biggr{\}} \int_{\sigma_{i}+2\rho^{*}}^{\tau_{i}-2\rho^{*}} a(t)\,dt \leq 2 \varphi \biggl{(} \dfrac{1}{2}\biggr{)}, 
\end{equation*}
contradicting the fact that $\lambda > \lambda^{*}$.
\end{proof}

\begin{remark}\label{rem-rho}
Using the same argument in the proof of Lemma~\ref{lem-rho}, one can show the following: if $u(t)$ is a non-negative solution to $(\mathscr{E}_{\lambda})$ such that, for some $i\in\{1,\ldots,m\}$, 
\begin{equation*}
c_{1} \leq \max_{t \in I^{+}_{i}} u(t) \leq c_{2} \leq \rho^{*}, \quad \text{for some $c_{1},c_{2} > 0$,}
\end{equation*}
then the following inequality
\begin{equation*}
\lambda \min \biggl{\{} g(u) \colon u \in \biggl{[}\frac{2 c_{1} \rho^{*}}{| I^{+}_{i} |},c_{2} \biggr{]}\biggr{\}} \int_{\sigma_{i}+2\rho^{*}}^{\tau_{i}-2\rho^{*}} a(t)\,dt \leq 2 \varphi \biggl{(} \dfrac{c_{2}}{2\rho^{*}}\biggr{)}
\end{equation*}
holds. The above formula will have a crucial role in the proofs of Theorem~\ref{th-conv-s1} and Theorem~\ref{th-conv-s2}.
$\hfill\lhd$
\end{remark}

All the four forthcoming lemmas deal with (non-negative) $T$-periodic solutions to
\begin{equation}\label{eq-3-theta}
(\varphi(u'))' + \vartheta\lambda a(t)g(u) = 0, \quad \vartheta\in\mathopen{]}0,1\mathclose{]},
\end{equation}
and they will be used for the computation of the topological degree via Lemma~\ref{lem-deg1}. In particular, the first two look at assumptions $(g_{0})$ and $(g_{\infty})$, respectively.

\begin{lemma}\label{lem-r-ro}
Let $a \colon \mathbb{R} \to \mathbb{R}$ be a locally integrable $T$-periodic function satisfying $(a_{\#})$. 
Let $g \colon {\mathbb{R}}^{+} \to {\mathbb{R}}^{+}$ be a continuous function satisfying $(g_{*})$ and $(g_{0})$.
Let $\lambda>0$.
There exists $r_{0}>0$ such that, for every $\vartheta\in \mathopen{]}0,1\mathclose{]}$, every non-negative $T$-periodic solution $u(t)$ to \eqref{eq-3-theta} with $\|u\|_{\infty} \leq r_{0}$ is such that $u\equiv0$.
\end{lemma}

\begin{proof}
By contradiction, we assume that there exists a sequence $(u_{n}(t))_{n}$ of non-negative $T$-periodic solutions to \eqref{eq-3-theta} for $\vartheta=\vartheta_{n}$ satisfying $0<\|u_{n}\|_{\infty}\to0$.
Letting $t^{*}_{n}\in \mathopen{[}0,T\mathclose{]}$ be such that $u_{n}(t^{*}_{n}) = \|u_{n}\|_{\infty} =: r_{n}$, we define
\begin{equation*}
v_{n}(t): = \dfrac{u_{n}(t)}{r_{n}}, \quad t\in\mathbb{R},
\end{equation*}
and observe that $v_{n}(t)$ is a non-negative $T$-periodic solution to
\begin{equation*}
\Biggl{(} \dfrac{v_{n}'}{\sqrt{1-(u_{n}')^{2}}}\Biggr{)}' + \vartheta_{n} \lambda a(t) q(u_{n}(t)) v_{n} = 0,
\end{equation*}
where $q(u) := g(u)/u$ for $u > 0$ and $q(0) := 0$. Multiplying the above equation by $v_{n}$ and integrating by parts on $\mathopen{[}0,T\mathclose{]}$, we obtain
\begin{equation*}
\int_{0}^{T} v_{n}'(t)^{2} \,dt \leq \int_{0}^{T} \dfrac{v_{n}'(t)^{2}}{\sqrt{1-u_{n}'(t)^{2}}} = \vartheta_{n} \lambda \int_{0}^{T} a(t)q(u_{n}(t))v_{n}(t)^{2} \,dt.
\end{equation*}
Therefore, using the first condition in $(g_{0})$ and recalling that $\| v_{n} \|_{\infty} \leq 1$, we obtain $\int_{0}^{T} v_{n}'(t)^{2} \,dt \to 0$.
As a consequence, for any $t \in \mathopen{[}0,T\mathclose{]}$,
\begin{equation*}
|v_{n}(t) - 1 | = |v_{n}(t) - v_{n}(t^{*}_{n}) | \leq \int_{0}^{T} | v_{n}'(\xi) | \,d\xi \leq \sqrt{T} \biggl{(}\int_{0}^{T} v_{n}'(\xi)^{2} \,d\xi \biggr{)}^{\frac{1}{2}} \to 0,
\end{equation*}
namely $v_{n}(t) \to 1$ uniformly in $t \in \mathopen{[}0,T\mathclose{]}$. 

Integrating now the equation for $u_{n}$ on $\mathopen{[}0,T\mathclose{]}$, we obtain
\begin{equation*}
0 = \int_{0}^{T} a(t) g(u_{n}(t))~\!dt = \int_{0}^{T} a(t) g(r_{n})~\!dt + \int_{0}^{T} a(t)\bigl{(}g(r_{n} v_{n}(t)) - g(r_{n})\bigr{)}~\!dt
\end{equation*}
and hence, dividing by $g(r_{n}) > 0$,
\begin{equation*}
0 < - \int_{0}^{T} a(t)~\!dt \leq \|a\|_{L^{1}_{T}} \sup_{t\in \mathopen{[}0,T\mathclose{]}}\biggl{|}\dfrac{g(r_{n} v_{n}(t))}{g(r_{n})} - 1\biggr{|}.
\end{equation*}
Using the second condition in $(g_{0})$ and recalling that $v_{n}(t) \to 1$ uniformly, we find that the right-hand side of the above inequality tends to zero, a contradiction.
\end{proof}

\begin{lemma}\label{lem-R-ro}
Let $a \colon \mathbb{R} \to \mathbb{R}$ be a locally integrable $T$-periodic function satisfying $(a_{\#})$. 
Let $g \colon {\mathbb{R}}^{+} \to {\mathbb{R}}^{+}$ be a continuous function satisfying $(g_{*})$ and $(g_{\infty})$.
There exists $R_{0}>0$ such that, for every $\lambda>0$ and for every $\vartheta\in \mathopen{]}0,1\mathclose{]}$, every non-negative $T$-periodic solution $u(t)$ to \eqref{eq-3-theta} satisfies $\|u\|_{\infty}< R_{0}$.
\end{lemma}

\begin{proof}
By contradiction, we assume that there exists a sequence $(u_{n}(t))_{n}$ of non-negative $T$-periodic solutions to \eqref{eq-3-theta} for $\vartheta=\vartheta_{n}$ and $\lambda=\lambda_{n}$ satisfying $\|u_{n}\|_{\infty}\to+\infty$. Letting $t^{*}_{n}\in \mathopen{[}0,T\mathclose{]}$ be such that $u_{n}(t^{*}_{n}) = \|u_{n}\|_{\infty} =: R_{n}$, we define
\begin{equation*}
v_{n}(t): = \dfrac{u_{n}(t)}{R_{n}}, \quad t\in\mathbb{R}.
\end{equation*}
Since $\|u_{n}'\|_{\infty} \leq 1$, we easily find $\| v_{n}' \|_{\infty} \to 0$ and, consequently,
\begin{equation*}
|v_{n}(t) - 1 | = |v_{n}(t) - v_{n}(t^{*}_{n}) | \leq \int_{0}^{T} | v_{n}'(\xi) | \,d\xi \to 0,
\end{equation*}
namely $v_{n}(t) \to 1$ uniformly in $t \in \mathopen{[}0,T\mathclose{]}$. Integrating the equation for $u_{n}$ and dividing by $g(R_{n}) > 0$, we thus obtain
\begin{equation*}
0 < - \int_{0}^{T} a(t)~\!dt \leq \|a\|_{L^{1}_{T}} \sup_{t\in \mathopen{[}0,T\mathclose{]}}\biggl{|}\dfrac{g(R_{n} v_{n}(t))}{g(R_{n})} - 1\biggr{|}.
\end{equation*}
Using $(g_{\infty})$, a contradiction easily follows.
\end{proof}

The last two lemmas give the same conclusions of Lemma~\ref{lem-r-ro} and Lemma~\ref{lem-R-ro}, under the alternative assumptions at zero and at infinity $(g_{0}')$ and $(g_{\infty}')$, respectively. The common strategy for their proofs is based on the following change of variable
\begin{equation}\label{change-var}
z(t) := \dfrac{\varphi(u'(t))}{\vartheta \lambda g(u(t))}, \quad t\in\mathbb{R}.
\end{equation}
One can easily check that if $u(t)$ is a $T$-periodic positive solution to \eqref{eq-3-theta}, where $g(u)$ is continuously differentiable on the range of $u(t)$, then $z(t)$ is a $T$-periodic solution to the first order equation
\begin{equation*}
z' + a(t) + \dfrac{g'(u(t))\varphi(u'(t))u'(t)}{\vartheta\lambda g(u(t))^{2}} = 0.
\end{equation*}
Such an equation can be written either as
\begin{equation}\label{eq-z1}
z' + a(t) + \vartheta\lambda g'(u(t)) \sqrt{1-u'(t)^{2}} z^{2}= 0
\end{equation}
or as
\begin{equation}\label{eq-z2}
z' + a(t) + \dfrac{g'(u(t))u'(t)}{g(u(t))} z = 0.
\end{equation}
Precisely, equation \eqref{eq-z1} will be used in Lemma~\ref{lem-r-C1} while equation \eqref{eq-z2} will be used in Lemma~\ref{lem-R-C1}.
For further convenience, we also observe that $z(t)$ vanishes at least once on $\mathopen{[}0,T\mathclose{]}$ (actually, at least twice), due to the $T$-periodicity of $u(t)$.

\begin{lemma}\label{lem-r-C1}
Let $a \colon \mathbb{R} \to \mathbb{R}$ be a locally integrable $T$-periodic function satisfying $(a_{\#})$. 
Let $g \colon {\mathbb{R}}^{+} \to {\mathbb{R}}^{+}$ be a continuous function satisfying $(g_{*})$ and $(g_{0}')$. Let $\lambda>0$.
There exists $r_{0}>0$ such that, for every $\vartheta\in \mathopen{]}0,1\mathclose{]}$, every non-negative $T$-periodic solution $u(t)$ to \eqref{eq-3-theta} with $\|u\|_{\infty} \leq r_{0}$ is such that $u\equiv0$.
\end{lemma}

\begin{proof}
Let $M>\|a\|_{L^{1}_{T}}$.
By contradiction, we assume that there exists a sequence $(u_{n}(t))_{n}$ of non-negative $T$-periodic solutions to \eqref{eq-3-theta} for $\vartheta=\vartheta_{n}$ satisfying $0<\|u_{n}\|_{\infty}\to0$. We perform the change of variable as in \eqref{change-var} and we claim that
\begin{equation*}
\|z_{n}\|_{\infty} \leq M.
\end{equation*}
We suppose by contradiction that this is not true. Then, recalling the fact that $z_{n}(t)$ vanishes at some point $\tilde{t}_{n}\in\mathopen{[}0,T\mathclose{]}$, we can find a maximal interval $J_{n}\subseteq \mathopen{[}0,T\mathclose{]}$ either of the form $\mathopen{[}\tilde{t}_{n},\hat{t}_{n}\mathclose{]}$ or of the form $\mathopen{[}\hat{t}_{n},\tilde{t}_{n}\mathclose{]}$, such that $|z_{n}(t)| \leq M$ for all $t\in J_{n}$ and $|z_{n}(t)| > M$ for some $t\notin J_{n}$. By the maximality of the interval $J_{n}$, we also know that $|z_{n}(\hat{t}_{n})| = M$.
Then, integrating on $J_{n}$ and using equation \eqref{eq-z1}, we obtain
\begin{equation*}
M = |z_{n}(\hat{t}_{n})| \leq \int_{J_{n}} |z_{n}'(t)| \, dt \leq \|a\|_{L^{1}_{T}} + \lambda T \sup_{t \in \mathopen{[}0,T\mathclose{]}} |g '(u_{n}(t))| M^{2}.
\end{equation*}
Using $(g_{0}')$ and passing to the limit we thus obtain $M\leq\|a\|_{L^{1}_{T}}$, contradicting the choice of $M$.

Now, we integrate \eqref{eq-z1} in $\mathopen{[}0,T\mathclose{]}$ in order to obtain
\begin{equation*}
0 < -\int_{0}^{T} a(t)\,dt \leq \lambda T \sup_{t \in \mathopen{[}0,T\mathclose{]}} |g'(u_{n}(t))| M^{2}.
\end{equation*}
and a contradiction is reached using again $(g_{0}')$.
\end{proof}

\begin{lemma}\label{lem-R-C1}
Let $a \colon \mathbb{R} \to \mathbb{R}$ be a locally integrable $T$-periodic function satisfying $(a_{\#})$.
Let $g \colon {\mathbb{R}}^{+} \to {\mathbb{R}}^{+}$ be a continuous function satisfying $(g_{*})$ and $(g_{\infty}')$.
There exists $R_{0}>0$ such that, for every $\lambda>0$ and for every $\vartheta\in \mathopen{]}0,1\mathclose{]}$, every non-negative $T$-periodic solution $u(t)$ to \eqref{eq-3-theta} satisfies $\|u\|_{\infty} < R_{0}$.
\end{lemma}

\begin{proof}
By contradiction, we assume that there exists a sequence $(u_{n}(t))_{n}$ of non-negative $T$-periodic solutions to \eqref{eq-3-theta} for $\vartheta=\vartheta_{n}$ and $\lambda=\lambda_{n}$ satisfying $\|u_{n}\|_{\infty}\to+\infty$. Since $\|u_{n}'\|_{\infty}<1$ it follows that $u_{n}(t)\to+\infty$ uniformly in $t$, so that
\begin{equation}\label{eq-3unif}
\dfrac{g'(u_{n}(t))}{g(u_{n}(t))}\to0, \quad \text{uniformly in $t$,}
\end{equation}
by assumption $(g_{\infty}')$.
We perform the change of variable as in \eqref{change-var} and we claim that the sequence $(\|z_{n}\|_{\infty})_{n}$ is bounded.
Indeed, recalling that $z_{n}(t)$ vanishes at least once and using equation \eqref{eq-z2}, we obtain
\begin{equation*}
\|z_{n}\|_{\infty} \leq \int_{0}^{T} |z_{n}'(t)| \,dt \leq \|a\|_{L^{1}_{T}} + T \sup_{t \in \mathopen{[}0,T\mathclose{]}} \dfrac{|g'(u_{n}(t))|}{g(u_{n}(t))} \|z_{n}\|_{\infty}.
\end{equation*}
Using \eqref{eq-3unif}, the claim easily follows.

Now, we integrate \eqref{eq-z2} in $\mathopen{[}0,T\mathclose{]}$ in order to obtain
\begin{equation*}
0 < -\int_{0}^{T} a(t)\,dt \leq T \sup_{t \in \mathopen{[}0,T\mathclose{]}} \dfrac{|g'(u_{n}(t))|}{g(u_{n}(t))} \|z_{n}\|_{\infty}
\end{equation*}
and a contradiction is reached using again \eqref{eq-3unif}.
\end{proof}

\subsection{Proofs}\label{section-3.2}

In this section we give the proofs of our non-existence/existence theorems.

\begin{proof}[Proof of Theorem~\ref{th-main-ex}]
We first apply Lemma~\ref{lem-rho} so as to find the constants $\rho^{*} > 0$ and $\lambda^{*} > 0$, and fix $\lambda > \lambda^{*}$. 

We claim that Lemma~\ref{lem-deg0} applies to the set $\Omega_{\rho^{*}}$ (see the definition in \eqref{eq-def-omega}) taking as $v(t)$ the indicator function $\mathbbm{1}_{\bigcup_{i} I^{+}_{i}}(t)$ of the set $\bigcup_{i} I^{+}_{i}$. To verify that assumption $(H_{2})$ holds true, we first observe that, 
since $v(t) = 0$ for any $t \in (\mathbb{R}/T\mathbb{Z})\setminus \bigcup_{i} I^{+}_{i}$, any non-negative $T$-periodic solution $u(t)$ to \eqref{eq-lem-deg0} is convex therein; therefore, its maximum is attained on $\bigcup_{i} I^{+}_{i}$. Then, $(H_{2})$ plainly follows from Lemma~\ref{lem-rho}. As for assumption $(H_{3})$, we integrate equation \eqref{eq-lem-deg0} on $\mathopen{[}0,T\mathclose{]}$ and pass to the absolute value in order to obtain
\begin{equation*}
\alpha \| v \|_{L^{1}_{T}} \leq \lambda \| a \|_{L^{1}_{T}} \max_{u \in \mathopen{[}0,\rho^{*}\mathclose{]}} g(u),
\end{equation*}
whence a contradiction follows for $\alpha$ sufficiently large. From Lemma~\ref{lem-deg0} we thus obtain
\begin{equation*}
\mathrm{D}_{L}(L-N_{\lambda},\Omega_{\rho^{*}}) = 0.
\end{equation*}

Next, we use either Lemma~\ref{lem-r-ro} or Lemma~\ref{lem-r-C1} (depending on whether $(g_{0})$ or $(g_{0}')$ is satisfied) as well as either Lemma~\ref{lem-R-ro} or Lemma~\ref{lem-R-C1} (depending on whether $(g_{\infty})$ or $(g_{\infty}')$ is satisfied) so as to find $r_{0} \in \mathopen{]}0,\rho^{*} \mathclose{[}$ and $R_{0} > \rho^{*}$ such that the corresponding conclusions hold true. Then, Lemma~\ref{lem-deg1} applies both to $\Omega_{r_{0}}$ and to $\Omega_{R_{0}}$
(indeed, $(H_{1})$ is trivially satisfied) yielding
\begin{equation*}
\mathrm{D}_{L}(L-N_{\lambda},\Omega_{r_{0}}) = -1 \quad \text{ and } \quad 
\mathrm{D}_{L}(L-N_{\lambda},\Omega_{R_{0}}) = -1.
\end{equation*}

By the additivity property of the coincidence degree, we obtain that
\begin{equation}\label{eq-deg-small}
\mathrm{D}_{L}(L-N_{\lambda},\Omega_{\rho^{*}} \setminus \overline{\Omega_{r_{0}}}) = 1
\end{equation}
and
\begin{equation}\label{eq-deg-large}
\mathrm{D}_{L}(L-N_{\lambda},\Omega_{R_{0}} \setminus \overline{\Omega_{\rho^{*}}}) = -1.
\end{equation}
As a consequence, we get the existence of a $T$-periodic solution $x_{s}(t) = (x_{s,1}(t),x_{s,2}(t))$
to \eqref{system} in $\Omega_{\rho^{*}} \setminus \overline{\Omega_{r_{0}}}$ as well as the existence of a 
$T$-periodic solution $x_{\ell}(t) = (x_{\ell,1}(t),x_{\ell,2}(t))$ to \eqref{system}
in $\Omega_{R_{0}} \setminus \overline{\Omega_{\rho^{*}}}$. Then, $u_{s}(t) := x_{s,1}(t)$ and $u_{\ell}(t) := x_{\ell,1}(t)$ are $T$-periodic solutions to \eqref{eq-fl} and satisfy 
\begin{equation*}
r_{0} < \| u_{s} \|_{\infty} < \rho^{*} < \| u_{\ell} \|_{\infty} < R_{0}.
\end{equation*}
By the maximum principles in Corollary~\ref{cor-weak-max-principle} and Theorem~\ref{strong-max-principle}, they are actually positive $T$-periodic solutions to $(\mathscr{E}_{\lambda})$ and the proof is concluded.
\end{proof}

\begin{proof}[Proof of Theorem~\ref{th-main-nex}]
As a preliminary step, we define a continuous function $\alpha \colon \mathbb{R}^{+} \to \mathopen{[}2-\eta,2\mathclose{]} \subseteq \mathopen{]}1,2\mathclose{]}$, with $\alpha(u) = 2$ for $u \in \mathopen{[}0,1\mathclose{]}$ and $\alpha(u) = 2-\eta$ for $u \in \mathopen{[}2,+\infty\mathclose{[}$. Due to the fact that $g'(u)$ is continuously differentiable on $\mathbb{R}^{+}$ together with assumption $(g_{\infty}'')$, we have
\begin{equation}\label{eq-alfa}
D_{\alpha} := \sup_{u \in \mathbb{R}^{+}}\frac{|g'(u)|}{g(u)^{2-\alpha(u)}} < +\infty.
\end{equation}

Arguing by contradiction, we now assume that there exists a sequence of positive $T$-periodic solutions $u_{n}(t)$ to 
$(\mathscr{E}_{\lambda})$ with $\lambda = \lambda_{n} \to 0^{+}$. Using the change of variable \eqref{change-var} (notice that $\vartheta = 1$) we thus obtain the equation
\begin{equation*}
z_{n}' + a(t) + \dfrac{g'(u_{n}(t))\varphi(u_{n}'(t))u_{n}'(t)}{\lambda_{n} g(u_{n}(t))^{2}} = 0
\end{equation*}
and passing to the absolute value, for any $t \in \mathopen{[}0,T\mathclose{]}$, we deduce that
\begin{equation}\label{eq-ineq}
\begin{aligned}
&| z_{n}'(t) | \leq |a(t)| + \dfrac{|g'(u_{n}(t))|\varphi(u_{n}'(t))u_{n}'(t)}{\lambda_{n} g(u_{n}(t))^{2}} 
= |a(t)| +	\\
& + \lambda_{n}^{\alpha(u_{n}(t))-1}\dfrac{|g'(u_{n}(t))| |u_{n}'(t) |^{2-\alpha(u_{n}(t))} 
| 1 - u_{n}'(t)^{2}|^{\frac{\alpha(u_{n}(t))-1}{2}}}{ g(u_{n}(t))^{2-\alpha(u_{n}(t))}} | z_{n}(t) |^{\alpha(u_{n}(t))}.
\end{aligned}
\end{equation}

Similarly as in the proof of Lemma~\ref{lem-r-C1}, we claim that 
\begin{equation*}
\|z_{n}\|_{\infty} \leq M,
\end{equation*}
where $M>\|a\|_{L^{1}_{T}} + 1$. By contradiction, suppose that this is not true. Using the fact that $z_{n}(\tilde{t}_{n}) = 0$ for some point $\tilde{t}_{n}\in\mathopen{[}0,T\mathclose{]}$, we find a maximal interval $J_{n}\subseteq \mathopen{[}0,T\mathclose{]}$ either of the form $\mathopen{[}\tilde{t}_{n},\hat{t}_{n}\mathclose{]}$ or of the form $\mathopen{[}\hat{t}_{n},\tilde{t}_{n}\mathclose{]}$, such that $|z_{n}(t)| \leq M$ for all $t\in J_{n}$ and $|z_{n}(t)| > M$ for some $t\notin J_{n}$. By the maximality of the interval $J_{n}$, we also know that $|z_{n}(\hat{t}_{n})| = M$.
Then, integrating on $J_{n}$ and using \eqref{eq-alfa} and \eqref{eq-ineq}, we obtain
\begin{equation*}
M = |z_{n}(\hat{t}_{n})| \leq \int_{J_{n}} |z_{n}'(t)| \, dt \leq \|a\|_{L^{1}_{T}} + 
\lambda_{n}^{\alpha(u_{n}(t))-1} T D_{\alpha} M^{\alpha(u_{n}(t))}.
\end{equation*}
Since $\lambda_{n} < 1$ for $n$ large and $M > 1$, we finally obtain
\begin{equation*}
M \leq \|a\|_{L^{1}_{T}} + \lambda_{n} T D_{\alpha} M^{2},
\end{equation*}
a contradiction with the choice of $M$, for $n$ large enough.

Now, we integrate \eqref{eq-z1} in $\mathopen{[}0,T\mathclose{]}$ in order to obtain, arguing as above,
\begin{equation*}
0 < -\int_{0}^{T} a(t)\,dt \leq \lambda_{n} T D_{\alpha} M^{2}.
\end{equation*}
and a contradiction is reached as $n \to \infty$.
\end{proof}

\begin{proof}[Proof of Theorem~\ref{th-main-ex-small} and Theorem~\ref{th-main-ex-large}]
We use the very same arguments as in the proof of Theorem~\ref{th-main-ex} to prove that 
the degree formula \eqref{eq-deg-small} holds true under the assumptions of Theorem~\ref{th-main-ex-small} and that \eqref{eq-deg-large} holds true under the assumptions of Theorem~\ref{th-main-ex-large}. This gives rise to a positive $T$-periodic solution $u_{s}(t)$ with $r_{0} < \| u_{s} \|_{\infty} < \rho^{*}$ in the former case and to a 
non-negative $T$-periodic solution $u_{\ell}(t)$ with $\rho^{*} < \| u_{\ell} \|_{\infty} < R_{0}$ in the latter one; in this case we further use the boundedness of $g(u)/u$ for $u$ near zero to apply the strong maximum principle in Theorem~\ref{strong-max-principle}.
\end{proof}

\begin{remark}\label{rem-3.2}
We observe that all the results in this section are valid for the equation
\begin{equation*}
(\varphi(u'))' + \lambda a(t)g(u) = 0,
\end{equation*}
where $\varphi\colon I \to \mathbb{R}$ is an increasing homeomorphism defined on a bounded interval $I \subseteq \mathbb{R}$ containing the origin, with 
$\varphi(0) = 0$ and
\begin{equation}\label{cond-phi}
0 < \liminf_{\xi \to 0}\frac{\varphi(\xi)}{\xi} \leq \limsup_{\xi \to 0}\frac{\varphi(\xi)}{\xi} < +\infty.
\end{equation}
Indeed, a careful check of the proofs shows that:
\begin{itemize}
\item the abstract degree theoretical setting is still suitable (see Remark \ref{rem-2.1}) and Lemma~\ref{lem-rho} can be proved: for all this, condition \eqref{cond-phi} does not play a role and, actually, even the boundedness of $I$ is not necessary;
\item Lemma~\ref{lem-R-ro} and Lemma~\ref{lem-R-C1} can be established in the very same way, using the fact that $I$ is bounded;
\item Lemma~\ref{lem-r-ro} and Lemma~\ref{lem-r-C1} can be proved with minor modifications of the arguments, using in an essential way condition \eqref{cond-phi}.
\end{itemize}
Then, the conclusions of Theorem~\ref{th-main-ex}, Theorem~\ref{th-main-ex-small} and Theorem~\ref{th-main-ex-large} follow by observing that the maximum principles in Appendix \ref{appendix-A} are valid in this setting (as for the strong maximum principle, one has to use once more condition \eqref{cond-phi} for verifying both $(\varphi_{*})$ and $(ii)$ of Theorem~\ref{strong-max-principle}). Theorem~\ref{th-main-nex} can also be established with minor changes in the proof.
$\hfill\lhd$
\end{remark}

\section{Asymptotic analysis for $\lambda\to+\infty$}\label{section-4}

In this section we study the asymptotic behaviour of both ``small'' and ``large'' $T$-periodic solutions to $(\mathscr{E}_{\lambda})$
when $\lambda \to +\infty$. In what follows, assumption $(a_{*})$ is going to play a crucial role and, to simplify the notation, for $i=1,\ldots,m$, we set
\begin{equation*}
I^{+}_{i} = \mathopen{[}\sigma_{i},\tau_{i}\mathclose{]} \quad \text{ and } \quad I^{-}_{i} = \mathopen{[}\tau_{i},\sigma_{i+1}\mathclose{]},
\end{equation*}
where $\sigma_{1}<\tau_{1}<\sigma_{2}<\tau_{2}<\ldots<\sigma_{m}<\tau_{m}<\sigma_{m+1} = \sigma_{1} + T$.
 
We first state and prove a result about the convergence (to zero) of ``small'' $T$-periodic solutions.

\begin{theorem}\label{th-conv-s1}
Under the assumptions of Theorem~\ref{th-main-ex-small}, let $\{u_{s,\lambda}(t)\}_{\lambda>\lambda^{*}}$ be a family of positive $T$-periodic solutions to $(\mathscr{E}_{\lambda})$ with $\|u_{s,\lambda}\|_{\infty}<\rho^{*}$. Then, for every $p \in \mathopen{[}1,+\infty\mathclose{[}$, it holds
\begin{equation*}
u_{s,\lambda}\to0 \quad \text{in $W^{1,p}_{T}$, as $\lambda\to+\infty$.}
\end{equation*}
In particular, $u_{s,\lambda}(t)$ converges to zero uniformly in $t$, as $\lambda\to+\infty$.
\end{theorem}

\begin{proof}
We first observe that the theorem is proved if we show that any sequence $\lambda_{n} \to +\infty$ admits a subsequence $\lambda_{n_{k}} \to +\infty$ such that the corresponding sequence $(u_{s,\lambda_{n_{k}}}(t))_{k}$ goes to $0$ in $W^{1,p}_{T}$.
Hence, let us take a sequence $\lambda_{n} \to +\infty$ and set 
$u_{n}(t) := u_{s,\lambda_{n}}(t)$. Since $\|u_{n}\|_{\infty}<\rho^{*}$ and $\|u'_{n}\|_{\infty}<1$, we obtain that 
$\|u_{n}\|_{W^{1,\infty}_{T}}$ is bounded. Regarding $L^{\infty}_{T}$ as the dual space of $L^{1}_{T}$, the sequential version of the Banach--Alaoglu theorem (see \cite[Corollary~3.30]{Br-11}) implies that, up to subsequences, there exists $u_{\infty}\in W^{1,\infty}_{T}$ such that $u_{n} \to u_{\infty}$ in the weak$^{*}$ topology of $W^{1,\infty}_{T}$,
namely $u_{n} \to u_{\infty}$ uniformly and 
\begin{equation}\label{eq-4.1-conv}
\int_{0}^{T} u'_{n}(t) \psi(t)\,dt \to \int_{0}^{T} u_{\infty}'(t) \psi(t)\,dt, \quad \text{for every } \psi \in L^{1}_{T}.
\end{equation}

We claim that the sequence $\rho_{n} := \| u_{n} \|_{\infty}$ goes to zero, implying, by uniform convergence, $u_{\infty} \equiv 0$. By contradiction, suppose that there exists $\rho_{*} > 0$ such that
\begin{equation}\label{eq-rho-n}
\rho_{n} \geq \rho_{*} > 0, \quad \text{for $n$ sufficiently large.}
\end{equation}
By the convexity of $u_{n}(t)$ on $\bigcup_{i} I^{-}_{i}$, there exists $i_{n} \in \{1,\ldots,m\}$ such that
$\rho_{n} = \max_{t \in I^{+}_{i_{n}}} u_{n}(t)$. By Remark~\ref{rem-rho} with $c_{1} = \rho_{n}$ and $c_{2} = \rho^{*}$, we obtain
\begin{equation*}
\lambda_{n} \min \biggl{\{} g(u) \colon u \in \biggl{[}\frac{2 \rho_{n} \rho^{*}}{| I^{+}_{i_{n}} |},\rho^{*} \biggr{]}\biggr{\}} \int_{\sigma_{i_{n}}+2\rho^{*}}^{\tau_{i_{n}}-2\rho^{*}} a(t)\,dt \leq 2 \varphi \biggl{(} \dfrac{1}{2}\biggr{)}.
\end{equation*}
Using \eqref{eq-rho-n}, a contradiction is then obtained for $n$ large enough.

We finally claim that 
\begin{equation*}
\int_{0}^{T} | u_{n}'(t) |^{p} \,dt \to 0,
\end{equation*}
implying that $u_{n} \to 0$ in the strong topology of $W^{1,p}_{T}$ and thus concluding the proof.
Let us first observe that, since $\|u_{n}'\|_{\infty}<1$, it is sufficient to prove that $\int_{0}^{T} |u_{n}'(t)| \,dt \to 0$; precisely, we are going to prove that for any $i\in\{1,\ldots,m\}$ and $\odot\in\{+,-\}$, it holds that
\begin{equation*}
\int_{I^{\odot}_{i}} |u_{n}'(t)| \,dt \to 0.
\end{equation*}
After passing to a subsequence, two cases may occur: either $u_{n}'(t)$ has constant sign on $I^{\odot}_{i}$ for any $n$, or $u_{n}'(t)$ is sign-changing on $I^{\odot}_{i}$ for any $n$.
In the first case, the conclusion is immediate, choosing $\psi$ in \eqref{eq-4.1-conv} to be the indicator function $\mathbbm{1}_{I^{\odot}_{i}}$ of the interval $I^{\odot}_{i}$. In the second case, by a convexity argument we observe that there exists $\hat{t}_{n}\in I^{\odot}_{i}$ such that $u_{n}'(t)\,\mathrm{sign}(t-\hat{t}_{n})$ has constant sign on $I^{\odot}_{i}$. Setting $I^{\odot}_{i}=\mathopen{[}t_{1},t_{2}\mathclose{]}$, we have
\begin{equation*}
\int_{I^{\odot}_{i}} |u_{n}'(t)| \,dt = \biggl{|} \int_{t_{1}}^{\hat{t}_{n}} u_{n}'(t) \,dt \biggr{|} + \biggl{|} \int_{\hat{t}_{n}}^{t_{2}} u_{n}'(t) \,dt \biggr{|}
\end{equation*}
We are going to show that both the integrals in the right-hand side go to zero. Up to a subsequence, we assume that $\hat{t}_{n}\to\hat{t}\in I^{\odot}_{i}$, so that
\begin{equation*}
\int_{t_{1}}^{\hat{t}_{n}} u_{n}'(t) \,dt = \int_{0}^{T} u_{n}'(t) \bigl{(}\mathbbm{1}_{\mathopen{[}t_{1},\hat{t}_{n}\mathclose{]}}(t)-\mathbbm{1}_{\mathopen{[}t_{1},\hat{t}\mathclose{]}}(t)\bigr{)} \,dt + \int_{0}^{T} u_{n}'(t) \mathbbm{1}_{\mathopen{[}t_{1},\hat{t}\mathclose{]}}(t) \,dt.
\end{equation*}
The first integral in the right-hand side goes to zero via the dominated convergence theorem, and the second one goes to zero using again \eqref{eq-4.1-conv}. One can proceed in a similar manner for the remaining term. The proof is thus completed.
\end{proof}

We notice that the case $p = \infty$ (corresponding to $\mathcal{C}_{T}^{1}$-convergence) is not considered in Theorem~\ref{th-conv-s1}.
We manage to obtain this stronger conclusion, which will be essential in the argument leading to the existence of subharmonic solutions (see the proof of Lemma~\ref{lem-mu}), under an additional assumption on the behavior of $g(u)$ near zero.

\begin{theorem}\label{th-conv-s2}
Under the assumptions of Theorem~\ref{th-main-ex-small}, let us further suppose that $g(u)$ is non-decreasing in a right neighborhood of $0$ and that $(g_{0}'')$ holds true. Let $\{u_{s,\lambda}(t)\}_{\lambda>\lambda^{*}}$ be a family of positive $T$-periodic solutions to $(\mathscr{E}_{\lambda})$ with $\|u_{s,\lambda}\|_{\infty}<\rho^{*}$. Then, it holds
\begin{equation*}
u_{s,\lambda}\to0 \quad \text{in $\mathcal{C}^{1}_{T}$, as $\lambda\to+\infty$,}
\end{equation*}
and, for a suitable constant $C > 0$, 
\begin{equation}\label{eq-us}
| u_{s,\lambda}''(t) | \leq C |a(t) |, \quad \text{for a.e.~$t \in \mathbb{R}$ and for all $\lambda > \lambda^{*}$.}
\end{equation}
\end{theorem}

\begin{proof}
We first show that
\begin{equation}\label{eq-stima}
s_{p} := \limsup_{\lambda\to+\infty} \lambda^{\frac{1}{p}} \|u_{s,\lambda}\|_{\infty} < +\infty,
\end{equation}
where $p > 1$ is given by $(g_{0}'')$. By contradiction, let us assume that, for a sequence $\lambda_{n} \to +\infty$ and setting $u_{n}(t) := u_{s,\lambda_{n}}(t)$, it holds that
\begin{equation*}
\lim_{n \to \infty} (\lambda_{n})^{\frac{1}{p}} \| u_{n} \|_{\infty} = +\infty.
\end{equation*}
Then, for $n$ large enough,
\begin{equation*}
\| u_{n} \|_{\infty} \geq \dfrac{1}{(\lambda_{n})^{\frac{1}{p}}}.
\end{equation*}
Letting $i_{n} \in \{1,\ldots,m\}$ be such that $\max_{t \in I^{+}_{i_{n}}} u_{n}(t) = \| u_{n} \|_{\infty}$ and 
using Remark~\ref{rem-rho} with $c_{1} = (\lambda_{n})^{-\frac{1}{p}}$ and $c_{2} = \| u_{n} \|_{\infty}$, we obtain 
\begin{equation*}
\lambda_{n} \min \biggl{\{} g(u) \colon u \in \biggl{[}\frac{2\rho^{*} }{(\lambda_{n})^{\frac{1}{p}} | I^{+}_{i_{n}} |},\| u_{n} \|_{\infty} \biggr{]}\biggr{\}} \int_{\sigma_{i_{n}}+2\rho^{*}}^{\tau_{i_{n}}-2\rho^{*}} a(t)\,dt \leq 2 \varphi \biggl{(} \dfrac{\| u_{n} \|_{\infty}}{2\rho^{*}}\biggr{)}.
\end{equation*} 
Using the fact that $g(u)$ is non-decreasing in a right neighborhood of zero, we find
\begin{equation*}
\lambda_{n} g\Biggl{(}\frac{2\rho^{*} }{(\lambda_{n})^{\frac{1}{p}} | I^{+}_{i_{n}} |} \Biggr{)} \int_{\sigma_{i_{n}}+2\rho^{*}}^{\tau_{i_{n}}-2\rho^{*}} a(t)\,dt \leq 2 \varphi \biggl{(} \dfrac{\| u_{n} \|_{\infty}}{2\rho^{*}}\biggr{)},
\end{equation*}
and, finally, 
\begin{equation*}
\lambda_{n} g\Biggl{(}\frac{2\rho^{*} }{(\lambda_{n})^{\frac{1}{p}} \max_{i} | I^{+}_{i} |} \Biggr{)} \min_{i} \int_{\sigma_{i}+2\rho^{*}}^{\tau_{i}-2\rho^{*}} a(t)\,dt \leq 2 \varphi \biggl{(} \dfrac{\| u_{n} \|_{\infty}}{2\rho^{*}}\biggr{)}.
\end{equation*} 
Since $\| u_{n} \|_{\infty} \to 0$ by Theorem~\ref{th-conv-s1}, the right-hand side tends to zero, contradicting assumption $(g_{0}'')$.

Now, recalling the expression for $u''_{s,\lambda}(t)$ given by \eqref{eq-usecondo}
and using again $(g_{0}'')$ together with \eqref{eq-stima}, for $\lambda$ large enough and for a.e.~$t \in \mathbb{R}$ we obtain
\begin{align*}
|u_{s,\lambda}''(t) | & \leq \lambda \max_{t \in \mathbb{R}} g(u_{s,\lambda}(t)) |a (t)| \leq (c_{p}+1) \lambda \| u_{s,\lambda} \|_{\infty}^{p} |a (t)| \\
& \leq (c_{p}+1) (s_{p}+1)^{p} |a (t)|.
\end{align*} 
On one hand, this immediately gives 
\eqref{eq-us}; on the other hand, it implies the equicontinuity of the family $\{ u'_{s,\lambda}(t) \}_{\lambda > \lambda^{*}}$. Since
$\| u'_{s,\lambda} \|_{\infty} < 1$, Ascoli--Arzel\`{a} theorem can be applied, finally giving $u_{s,\lambda} \to 0$ in $\mathcal{C}^{1}_{T}$, as $\lambda\to+\infty$.
\end{proof}

The next result provides the convergence of ``large'' $T$-periodic solutions to $(\mathscr{E}_{\lambda})$ (see Figure~\ref{fig-03} for a graphical example).

\begin{theorem}\label{th-conv-l}
Under the assumptions of Theorem~\ref{th-main-ex-large} and assuming further that $a(t)\not\equiv0$ in any interval of $\mathbb{R}$, let $\{u_{\ell,\lambda}(t)\}_{\lambda>\lambda^{*}}$ be a family of positive $T$-periodic solutions to $(\mathscr{E}_{\lambda})$ with $\|u_{\ell,\lambda}\|_{\infty}>\rho^{*}$. 
Then,
there exists a (not identically zero) Lipschitz continuous and $T$-periodic function $u_{\infty}(t)$, with 
\begin{equation*}
u_{\infty}'(t) \in \{-1,0,1\}, \quad \text{for a.e.~$t \in \mathbb{R}$,}
\end{equation*}
such that, up to subsequences, $u_{\ell,\lambda} \to u_{\infty}$ in the strong topology of $W^{1,p}_{T}$, for any $p \in \mathopen{[}1,+\infty\mathclose{[}$, as $\lambda \to +\infty$.
\end{theorem}

\begin{proof}
We first observe that, from Lemma~\ref{lem-R-ro} or Lemma~\ref{lem-R-C1}, it holds that
\begin{equation*}
\|u_{\ell,\lambda}\|_{\infty} < R_{0}, \quad \text{for all $\lambda > \lambda^{*}$.}
\end{equation*}
We can thus argue exactly as in the proof of Theorem~\ref{th-conv-s1} to infer (via the Banach--Alaoglu theorem) that, for a subsequences $\lambda_{n} \to +\infty$, $u_{n}(t) := u_{\ell,\lambda_{n}}(t)$ converges to a Lipschitz continuous and $T$-periodic function $u_{\infty}(t)$ in the weak$^{*}$ topology of $W^{1,\infty}_{T}$. Notice that $u_{n}(t) \not\equiv 0$, since $u_{n} \to u_{\infty}$ uniformly and $\|u_{n}\|_{\infty}>\rho^{*}$ for every $n$. The rest of the proof consists in defining 
a suitable $T$-periodic function $v_{\infty}(t)$, with $v_{\infty}(t) \in \{-1,0,1\}$ 
for almost every $t \in \mathbb{R}$, such that, for any $p \in \mathopen{[}1,+\infty\mathclose{[}$,
\begin{equation*}
\int_{0}^{T} | u_{n}'(t) - v_{\infty}(t) |^{p} \,dt \to 0.
\end{equation*}
This implies, by the uniqueness of the limit, that $v_{\infty}(t) = u_{\infty}'(t)$ for a.e.~$t$ and $u_{n} \to u_{\infty}$ in the strong topology of $W^{1,p}_{T}$, thus concluding the proof.

Actually, we are going to define the function $v_{\infty}(t)$ separately in each $I^{\odot}_{i}$, where $i\in\{1,\ldots,m\}$ and $\odot\in\{+,-\}$, and then prove that
\begin{equation}\label{eq-v}
\int_{I^{\odot}_{i}} | u_{n}'(t) - v_{\infty}(t) |^{p} \,dt \to 0.
\end{equation}

We begin with the case $\odot = +$. As a preliminary observation, we notice that, by uniform convergence, $u_{\infty}(t)$ is concave on $I^{+}_{i}$. Therefore, one and only one of the following situations occurs: either $u_{\infty}(t) \equiv0$ on $I^{+}_{i}$, or $u_{\infty}(t)\neq0$ for all $t\in\mathopen{]}\sigma_{i},\tau_{i}\mathclose{[}$. In the first case, the arguments used in the proof of Theorem~\ref{th-conv-s1} yield to
\begin{equation*}
\int_{I^{+}_{i}} | u_{n}'(t) |^{p} \,dt \to 0,
\end{equation*}
so that \eqref{eq-v} follows with the position $v_{\infty}(t) := 0$ for $t \in \mathopen{]}\sigma_{i},\tau_{i}\mathclose{[}$. To treat the other case, let us further distinguish two situations (up to subsequences, they are the only possible ones).
\begin{itemize}[leftmargin=*]
\item [$\bullet$] $u_{n}'(t)$ has constant sign on $I^{+}_{i}$, for any $n$. Let us first suppose that $u_{n}'(t)\geq0$; then, for all $t\in\mathopen{]}\sigma_{i},\tau_{i}\mathclose{[}$, it holds that
\begin{equation*}
-\varphi(u_{n} '(t))\leq \varphi(u_{n} '(\tau_{i}))-\varphi(u_{n} '(t)) = - \lambda_{n} \int_{t}^{\tau_{i}} a(\xi)g(u_{n} (\xi)) \,d\xi.
\end{equation*} 
Since $u_{\infty}(t)\neq0$ for every $t\in\mathopen{]}\sigma_{i},\tau_{i}\mathclose{[}$,
the right-hand side goes to $-\infty$. Then $\varphi(u_{n} '(t))\to+\infty$ and thus $u_{n} '(t)\to1$.
By the dominated convergence theorem, we deduce that
\begin{equation*}
\int_{I^{+}_{i}} | u_{n}'(t) - 1 |^{p} \,dt \to 0, 
\end{equation*}
so that \eqref{eq-v} holds true with the position $v_{\infty}(t) := 1$ for $t \in \mathopen{]}\sigma_{i},\tau_{i}\mathclose{[}$. If $u_{n}'(t)\leq0$ for all $t\in\mathopen{[}\sigma_{i},\tau_{i}\mathclose{]}$, we proceed in a similar manner, showing that $u_{n} '(t)\to-1$, so that \eqref{eq-v} holds true with 
$v_{\infty}(t) := -1$.
\item [$\bullet$] $u_{n}'(t)$ is sign-changing on $I^{+}_{i}$, for any $n$. Then, by the concavity of $u_{n}(t)$ on $I^{+}_{i}$, 
there exists $\hat{t}_{n}\in\mathopen{]}\sigma_{i},\tau_{i}\mathclose{[}$ such that $u_{n}'(t)\geq 0$ on $\mathopen{[}\sigma_{i},\hat{t}_{n}\mathclose{]}$ and $u_{n}'(t)\leq 0$ on $\mathopen{]}\hat{t}_{n},\tau_{i}\mathclose{]}$; moreover, up to a subsequence, $\hat{t}_{n}\to\hat{t}\in I^{+}_{i}$. 
For all $t\in\mathopen{[}\sigma_{i},\hat{t}\mathclose{[}$, taking $n$ so large that $t < \hat{t}_{n}$ and using the fact that $u'_{n}(\hat{t}_{n}) = 0$, it holds that
\begin{equation*}
\varphi(u_{n} '(t)) = \lambda_{n} \int_{t}^{\hat{t}_{n}} a(\xi)g(u_{n} (\xi)) \,d\xi.
\end{equation*}
Then, since $u_{\infty}(t)\neq0$ for all $t\in \mathopen{]}\sigma_{i},\tau_{i}\mathclose{[}$, we deduce $u_{n} '(t)\to1$. Similarly, one can obtain that for all $t\in\mathopen{]}\hat{t},\tau_{i}\mathclose{[}$ it holds that $u_{n}'(t)\to-1$. 
Then, defining
\begin{equation*}
v_{\infty}(t) :=
\begin{cases}
\, 1, & \text{if $t\in\mathopen{]}\sigma_{i},\hat{t}\mathclose{[}$,} \\
\, -1, & \text{if $t\in\mathopen{]}\hat{t},\tau_{i}\mathclose{[}$,}
\end{cases}
\end{equation*}
we obtain that \eqref{eq-v} is fulfilled, by the dominated convergence theorem. 
\end{itemize}

We now deal with the case $\odot = -$. By uniform convergence $u_{\infty}(t)$ is convex on $I^{-}_{i}$; therefore, one and only one of the following situations occurs: either $u_{\infty}(t)\neq0$ for all $t\in\mathopen{]}\tau_{i},\sigma_{i+1}\mathclose{[}$, or
there exist $\kappa_{i},\nu_{i}\in I^{-}_{i}$ with $\kappa_{i}\leq \nu_{i}$ and $\mathopen{[}\kappa_{i},\nu_{i}\mathclose{]}\cap \mathopen{]}\tau_{i},\sigma_{i+1}\mathclose{[}\neq\emptyset$ such that $u_{\infty}(t)=0$ for all $t\in \mathopen{[}\kappa_{i},\nu_{i}\mathclose{]}$ and $u_{\infty}(t)\neq0$ for all $t\in\mathopen{[}\tau_{i},\sigma_{i+1}\mathclose{]}\setminus \mathopen{[}\kappa_{i},\nu_{i}\mathclose{]}$.

In the case $u_{\infty}(t)\neq0$ for all $t\in\mathopen{]}\tau_{i},\sigma_{i+1}\mathclose{[}$, we argue similarly as in the case $\odot = +$, by distinguish two situations (we give some details, since they are slightly different with respect to the previous case).
\begin{itemize}[leftmargin=*]
\item [$\bullet$] $u_{n}'(t)$ has constant sign on $I^{-}_{i}$, for any $n$. Assuming, for instance, that $u_{n}'(t)\geq0$, then, for all 
$t\in\mathopen{]}\tau_{i},\sigma_{i+1}\mathclose{[}$, it holds that
\begin{equation*}
\varphi(u_{n} '(t))\geq \varphi(u_{n} '(t))-\varphi(u_{n} '(\tau_{i})) = \lambda_{n} \int_{\tau_{i}}^{t} a(\xi)g(u_{n} (\xi)) \,d\xi.
\end{equation*} 
Since $u_{\infty}(t)\neq0$ for every $t\in\mathopen{]}\tau_{i},\sigma_{i+1}\mathclose{[}$, the right-hand side goes to $+\infty$. Then $\varphi(u_{n} '(t))\to+\infty$ and thus $u_{n} '(t)\to1$.
Arguing similarly if $u_{n}'(t)\leq0$ and using once again the dominated convergence theorem, we finally find that \eqref{eq-v} holds true with 
$v_{\infty}(t) := 1$ for $t \in \mathopen{]}\tau_{i},\sigma_{i+1}\mathclose{[}$ in the former case and $v_{\infty}(t) := -1$ in the latter one.
\item [$\bullet$] $u_{n}'(t)$ is sign-changing on $I^{-}_{i}$, for any $n$. Then, by the convexity of $u_{n}(t)$ on $I^{-}_{i}$, 
there exists $\hat{t}_{n} \in \mathopen{]}\tau_{i},\sigma_{i+1}\mathclose{[}$ such that $u_{n}'(t)\leq 0$ on $\mathopen{[}\tau_{i},\hat{t}_{n}\mathclose{]}$ and $u_{n}'(t)\geq 0$ on $\mathopen{[}\hat{t}_{n},\sigma_{i+1}\mathclose{]}$; moreover, up to a subsequence, $\hat{t}_{n}\to\hat{t}\in I^{-}_{i}$.
For all $t\in\mathopen{]}\tau_{i},\hat{t}\mathclose{[}$, taking $n$ so large that $t < \hat{t}_{n}$ and using the fact that $u'_{n}(\hat{t}_{n}) = 0$, it holds that
\begin{equation*}
\varphi(u_{n} '(t)) = -\lambda_{n} \int_{t}^{\hat{t}_{n}} a(\xi)g(u_{n} (\xi)) \,d\xi \to -\infty,
\end{equation*}
and thus $u_{n} '(t)\to - 1$. Similarly, one can obtain that for all $t\in\mathopen{]}\hat{t},\sigma_{i+1}\mathclose{[}$ it holds that $u_{n}'(t)\to 1$. The conclusion then follows as in the case $\odot = +$, by defining 
$v_{\infty}(t) = - 1$ on $\mathopen{]}\tau_{i},\hat{t}\mathclose{[}$ and $v_{\infty}(t) := 1$ on $\mathopen{]}\hat{t},\sigma_{i+1}\mathclose{[}$.
\end{itemize}

Finally, in the case $u_{\infty}(t)=0$ for all $t\in \mathopen{[}\kappa_{i},\nu_{i}\mathclose{]}$ and $u_{\infty}(t)\neq0$ for all $t\in\mathopen{[}\tau_{i},\sigma_{i+1}\mathclose{]}\setminus \mathopen{[}\kappa_{i},\nu_{i}\mathclose{]}$ we simply argue exactly as before separately in the intervals $\mathopen{[}\tau_{i},\kappa_{i}\mathclose{]}$ and $\mathopen{[}\nu_{i},\sigma_{i+1}\mathclose{]}$ and use the arguments of the proof of Theorem~\ref{th-conv-s1} to show that $\int_{\kappa_{i}}^{\nu_{i}} | u_{n}'(t) |^{p} \,dt \to 0$.
\end{proof}

\begin{figure}[htb]
\centering
\begin{tikzpicture}[scale=1]
\begin{axis}[
  tick label style={font=\scriptsize},
          scale only axis,
  enlargelimits=false,
  xtick={0,1,2,3,4},
  xticklabels={$0$, $\frac{\pi}{2}$, , , $2\pi$},
  ytick={0,4},
  xlabel={\small $t$},
  ylabel={\small $u(t)$},
  max space between ticks=50,
                minor x tick num=1,
                minor y tick num=7,  
every axis x label/.style={
below,
at={(3.5cm,0cm)},
  yshift=-3pt
  },
every axis y label/.style={
below,
at={(0cm,2.5cm)},
  xshift=-3pt},
  y label style={rotate=90,anchor=south},
  width=7cm,
  height=5cm,  
  xmin=0,
  xmax=4,
  ymin=0,
  ymax=4]
\addplot graphics[xmin=0,xmax=4,ymin=0,ymax=4] {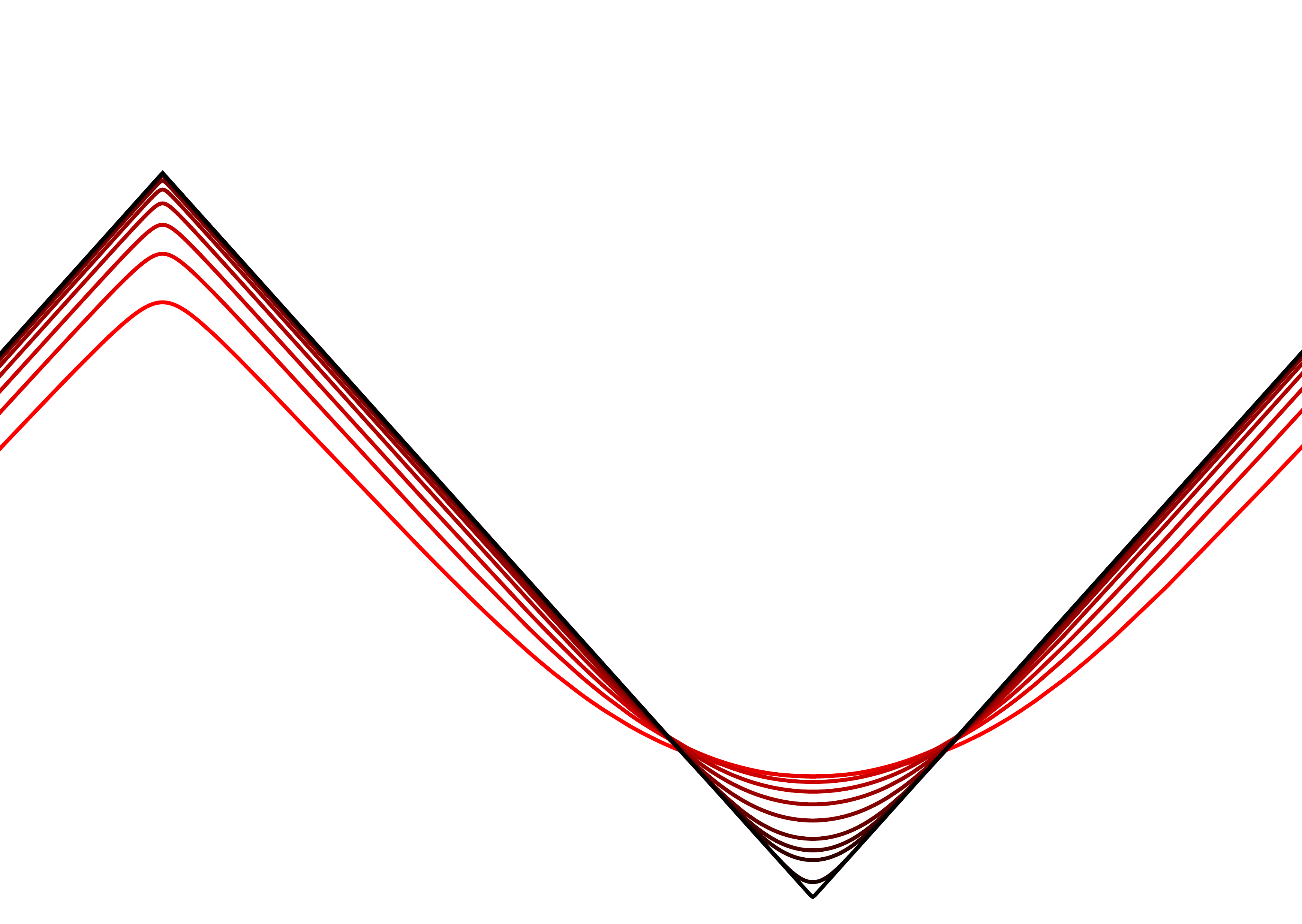};
\end{axis}
\end{tikzpicture} 
\caption{Graphs of the ``large'' $2\pi$-periodic solution $u_{\ell,\lambda}(t)$ to $(\mathscr{E}_{\lambda})$ with $g(u)=u^{3}$, $a(t)=\cos(t-\pi/4)-\sqrt{2}/2$ and $\lambda\in\{1.2,1.5,2,3,5,10,25,50,100,10^{3},10^{5}\}$.}         
\label{fig-03}
\end{figure}

\begin{remark}\label{rem-4.1}
As a by-product of the proof of Lemma~\ref{th-conv-l}, $u_{\infty}'(t)$ can vanish on a non-degenerate interval only if $u_{\infty} \equiv 0$ on that interval. This is a consequence of the assumption that $a(t)$ is not identically zero on any interval of the real line:
if such a condition is dropped, intervals on which $u_{\infty}(t)$ is a non-zero constant may appear (see Figure \ref{fig-04} for a numerical example).

We do not know if Lemma~\ref{th-conv-l} holds true under the more general assumption $(a_{*})$. A careful inspection of the proof shows that the $W^{1,p}_T$-convergence to a non-zero Lipschitz profile $u_{\infty}(t)$ is still guaranteed and that such a profile is a piecewise affine function. However, 
it seems delicate to prove that the derivative must be equal to $-1,0,1$ on intervals on which the weight function $a(t)$ vanishes. Probably, this would require a sharper analysis of the behavior of the solutions and, possibly, some ad hoc assumptions on the behavior of $a(t)$ ``near the zeros''.
$\hfill\lhd$
\end{remark}

\begin{figure}[htb]
\centering
\begin{tikzpicture}[scale=1]
\begin{axis}[
  tick label style={font=\scriptsize},
          scale only axis,
  enlargelimits=false,
  xtick={0,10},
  xticklabels={$0$, $10$},
  ytick={0,4},
  xlabel={\small $t$},
  ylabel={\small $u(t)$},
  max space between ticks=50,
                minor x tick num=9,
                minor y tick num=7,  
every axis x label/.style={
below,
at={(3.5cm,0cm)},
  yshift=-3pt
  },
every axis y label/.style={
below,
at={(0cm,2.5cm)},
  xshift=-3pt},
  y label style={rotate=90,anchor=south},
  width=7cm,
  height=5cm,  
  xmin=0,
  xmax=10,
  ymin=0,
  ymax=4]
\addplot graphics[xmin=0,xmax=10,ymin=0,ymax=4] {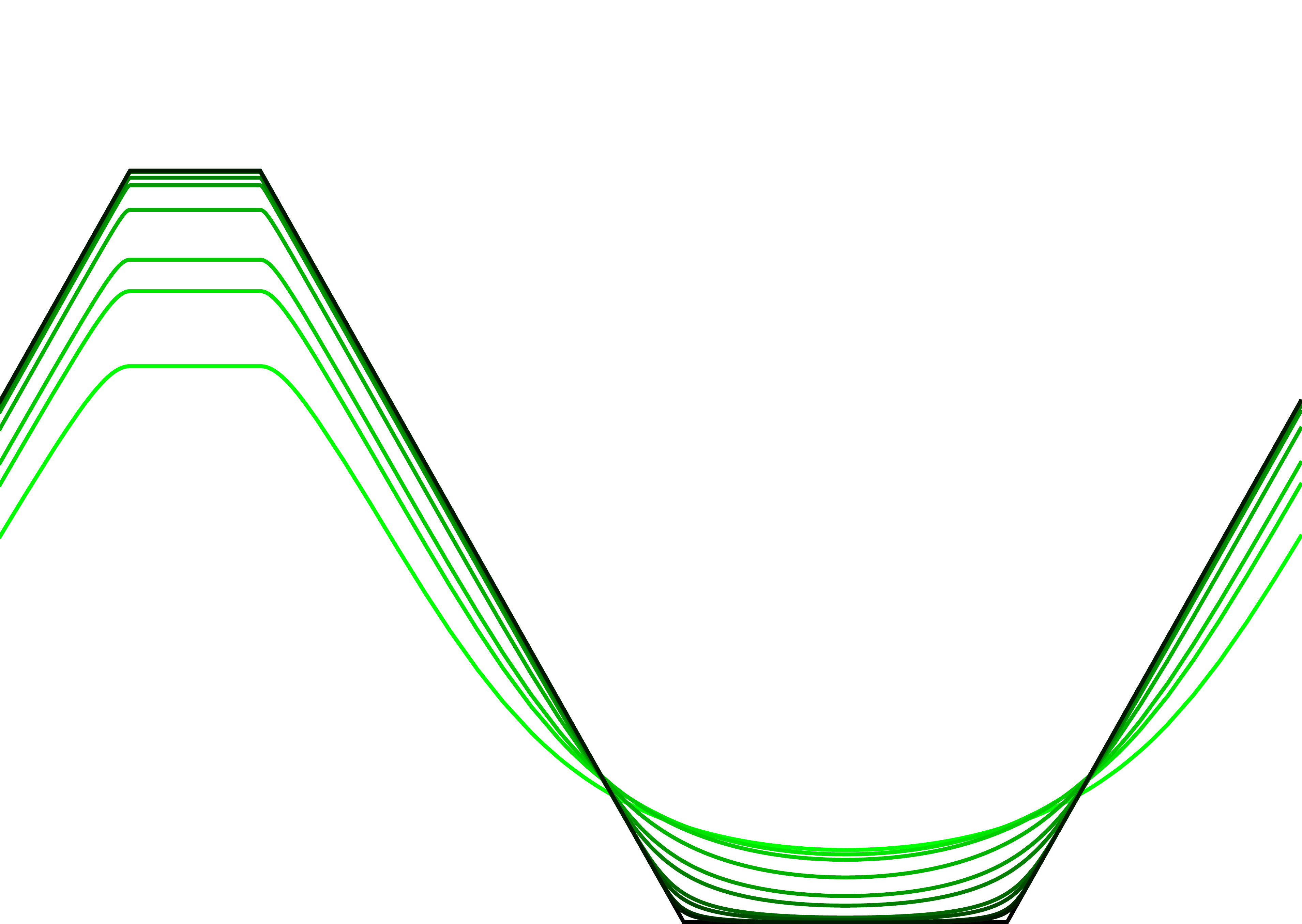};
\end{axis}
\end{tikzpicture} 
\caption{Graphs of the ``large'' $T$-periodic solution $u_{\ell,\lambda}(t)$ to $(\mathscr{E}_{\lambda})$ with $T=10$, $g(u)=u^{2}$, $a(t)=1$ on $\mathopen{[}0,1\mathclose{]}\cup\mathopen{[}2,3\mathclose{]}$, $a(t)=0$ on $\mathopen{[}1,2\mathclose{]}$ and $a(t)=-2$ on $\mathopen{[}3,10\mathclose{]}$, and $\lambda\in\{0.6,0.8,1,2,5,10,50,100,500,1000,10^{6}\}$.}      
\label{fig-04}
\end{figure}

\section{Subharmonic solutions}\label{section-5}

In this section, we present our result for positive subharmonic solutions to equation $(\mathscr{E}_{\lambda})$.

Before giving the statement, let us make the notion of subharmonic solutions precise. 
We say that a solution $u(t)$ to $(\mathscr{E}_{\lambda})$, defined on the whole real line, is a \textit{subharmonic solutions of order $k$} 
(with $k \geq 2$ an integer number) if $u(t)$ is $kT$-periodic, but not $lT$-periodic for any integer $l = 1,\ldots,k-1$, that is, 
$kT$ is the minimal period of $u(t)$ in the class of the integer multiples of $T$. If $T$ is the minimal period of $a(t)$ and $u(t)$ is a positive subharmonic solution of order $k$, by writing the equation as $\lambda a(t) = - (\varphi(u'(t)))'/g(u(t))$, it is easy to see that the minimal period of $u(t)$ is actually $kT$. Let us also observe
that if $u(t)$ is a subharmonic solution of order $k$ to $(\mathscr{E}_{\lambda})$, then the $k - 1$ functions
$u(\cdot + lT )$, for $l = 1, \ldots, k-1$, are subharmonic solutions of order $k$, as well; these solutions, though distinct, are considered equivalent from the point of view of the counting of subharmonics. Accordingly, given $u_{1}(t), u_{2}(t)$
subharmonic solutions of order $k$ to $(\mathscr{E}_{\lambda})$, we say that $u_{1}(t)$ and $u_{2}(t)$ are not in
the same periodicity class if $u_{1}(\cdot) \neq u_{2}(\cdot + lT )$ for any integer $l = 1, \ldots , k - 1$.

With this in mind, the following result holds true. Notice that we require $a \in L^{\infty}_T$, together with extra conditions on the behavior of $g(u)$ near zero, while no assumptions for $g(u)$ at infinity are imposed.

\begin{theorem}\label{th-sub}
Let $a \colon \mathbb{R} \to \mathbb{R}$ be a locally bounded $T$-periodic function satisfying $(a_{\#})$ and $(a_{*})$. 
Let $g \colon {\mathbb{R}}^{+} \to {\mathbb{R}}^{+}$ be a twice continuously differentiable function satisfying $(g_{*})$
and $g''(u) > 0$ for all $u \in \mathopen{]}0,\varepsilon\mathclose{[}$, for some $\varepsilon> 0$; moreover, assume that
$(g_{0}'')$ is fulfilled and that
\begin{equation*} 
\lim_{u \to 0^{+}} \dfrac{g'(u)}{g''(u)} = 0.
\end{equation*} 
Then, there exists $\Lambda^{*} \geq \lambda^{*}$ such that for every $\lambda>\Lambda^{*}$ equation $(\mathscr{E}_{\lambda})$ possesses positive subharmonic solutions of order $k$ for every integer $k$ large enough.

More precisely, there exists $\Lambda^{*} \geq \lambda^{*}$ such that for every family \linebreak[4] $\{u_{s,\lambda}(t)\}_{\lambda>\lambda^{*}}$ of positive $T$-periodic solutions to $(\mathscr{E}_{\lambda})$ with $\|u_{s,\lambda}\|_{\infty}<\rho^{*}$ as in Theorem~\ref{th-main-ex-small} and for every $\lambda > \Lambda^{*}$ the following holds true: there exist an integer $k^{*} \geq 1$ and a sequence of integers $(m_k)_{k \geq k^{*}}$ with $m_k \to +\infty$ such that for every integer $k \geq k^{*}$ and for every integer $j$ relatively prime with $k$ and such that $1 \leq j \leq m_{k}$, equation $(\mathscr{E}_{\lambda})$ has two positive subharmonic solutions $u_{s,\lambda,k,j}^{(i)}(t)$ $(i = 1, 2)$ of order $k$ (not belonging
to the same periodicity class) such that $u_{s,\lambda,k,j}^{(i)}(t)-u_{s,\lambda}(t)$ has exactly $2j$ zeros in the interval $\mathopen{[}0,kT\mathclose{[}$.
\end{theorem}

\begin{remark}\label{rem-5.1}
We notice that, if $g(u)$ is twice continuously differentiable and $(g_{0}'')$ is satisfied, then
$g'(0) = 0$. Therefore, under the assumptions of Theorem~\ref{th-sub}, condition $(g_{0}')$ is satisfied and Theorem~\ref{th-main-ex-small} can be applied.
$\hfill\lhd$
\end{remark}

As already mentioned in the introduction, the proof of Theorem~\ref{th-sub} relies on the Poincar\'{e}--Birkhoff fixed point theorem, on the lines of \cite{BoFe-18,BoZa-13}. Before giving the proof, we need the following crucial lemma. 

\begin{lemma}\label{lem-mu}
Under the assumptions of Theorem~\ref{th-sub}, there exists $\Lambda^{*} \geq \lambda^{*}$ such that for every family $\{u_{s,\lambda}(t)\}_{\lambda>\lambda^{*}}$ of positive $T$-periodic solutions to $(\mathscr{E}_{\lambda})$ with $\|u_{s,\lambda}\|_{\infty}<\rho^{*}$ and for every $\lambda>\Lambda^{*}$ the principal eigenvalue $\mu_{0}$ of the linear problem
\begin{equation}\label{eq-mu}
\bigl{(} \varphi'(u_{s,\lambda}'(t))w'\bigr{)}' + \bigl{(} \mu + \lambda a(t)g'(u_{s,\lambda}(t)) \bigr{)} w = 0
\end{equation} 
is strictly negative (see Appendix \ref{appendix-C} for the definition of $\mu_{0}$; notice that the definition is well-posed since $\varphi'(\xi) \geq 1$ for any $\xi$).
\end{lemma}

\begin{proof}
First, we are going to establish a preliminary integral relationship.
Let $\mu_{0}$ be the principal eigenvalue of the $T$-periodic problem for \eqref{eq-mu} and let $w_{\lambda}(t)$ be an associated positive $T$-periodic solution. Without loss of generality, we also assume that $\| w_{\lambda} \|_{\infty} = 1$. By multiplying equation $(\mathscr{E}_{\lambda})$ by $g'(u_{s,\lambda}) w_{\lambda}$ and equation \eqref{eq-mu} (for $\mu = \mu_{0}$) by $g(u_{s,\lambda})$, we obtain
\begin{equation*} 
\bigl{(} \varphi(u_{s,\lambda}')\bigr{)}' g'(u_{s,\lambda}) w_{\lambda} + \lambda a(t)g(u_{s,\lambda}) g'(u_{s,\lambda}) w_{\lambda} = 0
\end{equation*}
and
\begin{equation*} 
\bigl{(} \varphi'(u_{s,\lambda}')w_{\lambda}'\bigr{)}' g(u_{s,\lambda})+ \bigl{(} \mu_{0} + \lambda a(t)g'(u_{s,\lambda}) \bigr{)} g(u_{s,\lambda})w_{\lambda}= 0, 
\end{equation*}
where, above and from now on, we omit the dependence on $t$ in all the considered functions.
Subtracting the first equation from the second one and integrating on $\mathopen{[}0,T\mathclose{]}$, we obtain
\begin{equation*}
\mu_{0} \int_{0}^{T} w_{\lambda} g(u_{s,\lambda}) = \int_{0}^{T} \Bigl{[} g'(u_{s,\lambda}) w_{\lambda} (\varphi(u_{s,\lambda}'))' - (\varphi'(u_{s,\lambda}')w_{\lambda}')' g(u_{s,\lambda})\Bigr{]}.
\end{equation*}
Integrating by parts, we write the right-hand side as
\begin{equation*}
\begin{aligned}
& \int_{0}^{T} \Bigl{[} - \varphi(u_{s,\lambda}') \bigl{(}g''(u_{s,\lambda}) u_{s,\lambda}'w_{\lambda} + g'(u_{s,\lambda})w_{\lambda}' \bigr{)}
+ \varphi'(u_{s,\lambda}')w_{\lambda}' g'(u_{s,\lambda})u_{s,\lambda}'\Bigr{]} 
\\ &= - \int_{0}^{T} \varphi(u_{s,\lambda}') u_{s,\lambda}' g''(u_{s,\lambda}) w_{\lambda} 
+ \int_{0}^{T} g'(u_{s,\lambda})w_{\lambda}' \bigl{(} \varphi'(u_{s,\lambda}')u_{s,\lambda}' - \varphi(u_{s,\lambda}')\bigr{)}. 
\end{aligned}
\end{equation*}
Integrating again by parts, we write the second integral in the right-hand side of the above equality as
\begin{equation*}
-\int_{0}^{T} w_{\lambda} \Bigl{[} g''(u_{s,\lambda})u_{s,\lambda}' \bigl{(} \varphi'(u_{s,\lambda}')u_{s,\lambda}' - \varphi(u_{s,\lambda}')\bigr{)} + g'(u_{s,\lambda}) \varphi''(u_{s,\lambda}')u_{s,\lambda}' u_{s,\lambda}'' \Bigr{]}. 
\end{equation*}
Using the explicit expression of $\varphi$, simple computations finally lead to
\begin{equation*}
\begin{aligned}
& \mu_{0} \int_{0}^{T} w_{\lambda} g(u_{s,\lambda}) = \\
& = -\int_{0}^{T} \dfrac{(u_{s,\lambda}')^{2}}{(1-(u_{s,\lambda}')^{2})^{\frac{3}{2}}} g''(u_{s,\lambda}) w_{\lambda}
- \int_{0}^{T} \dfrac{3(u_{s,\lambda}')^{2}}{(1-(u_{s,\lambda}')^{2})^{\frac{5}{2}}} u_{s,\lambda}'' g'(u_{s,\lambda}) w_{\lambda}.
\end{aligned}
\end{equation*}

We are now in a position to conclude. By contradiction, assume that there exist a sequence $\lambda_{n} \to +\infty$ and a sequence $u_{n}(t) := u_{s,\lambda_{n}}(t)$ of positive $T$-periodic solutions to $(\mathscr{E}_{\lambda})$ with $\|u_{n}\|_{\infty}<\rho^{*}$ such that the principal eigenvalue $\mu_{0,n}$ is non-negative. Let us also denote by $w_{n}(t) := w_{\lambda_{n}}(t)$ a corresponding positive eigenfunction with $\|w_{n}\|_{\infty} = 1$.
Since $g'(0) = 0$ and $g''(u) > 0$ for $u \in \mathopen{]}0,\varepsilon \mathclose{]}$, the function
$g(u)$ is strictly increasing on $\mathopen{[}0,\varepsilon \mathclose{]}$; accordingly, Theorem~\ref{th-conv-s2} can be applied yielding
$u_{n}\to0$ in $\mathcal{C}_{T}^{1}$ and $|u_{n}''(t)|\leq C \| a \|_{\infty}$ for a.e.~$t \in \mathbb{R}$ and $n$ large enough.
Using the fact that $g'(u)/g''(u) \to 0$ for $u \to 0^{+}$, we thus have, for every $t \in \mathbb{R}$ and $n$ large enough, 
\begin{equation*}
| g'(u_{n}(t)) | < \frac{g''(u_{n}(t))}{3C\| a \|_{\infty}}, \quad \text{for all $t \in \mathbb{R}$,}
\end{equation*}
so that, for a.e.~$t \in \mathbb{R}$ and $n$ large enough,
\begin{equation*}
\Biggl{|} \dfrac{3 u_{n}'(t)^{2}}{(1-u_{n}'(t)^{2})^{\frac{5}{2}}} u_{n}''(t) g'(u_{n}(t)) w_{n}(t) \Biggr{|}
< \dfrac{u_{n}'(t)^{2}}{(1-u_{n}'(t)^{2})^{\frac{3}{2}}} g''(u_{n}(t)) w_{n}(t).
\end{equation*}
Recalling the integral relationship obtained in the first part of the proof, we thus find
\begin{equation*}
\mu_{0,n} \int_{0}^{T} w_{n}(t) g(u_{n}(t)) \,dt > 0,
\end{equation*}
a contradiction.
\end{proof}

\begin{remark}\label{rem-5.2}
From a variational viewpoint, Lemma~\ref{lem-mu} implies that any ``small'' solution $u_{s,\lambda}(t)$, when regarded as a critical point of the action functional
\begin{equation*}
\mathcal{J}(u) = \int_{0}^{T} \Bigl{(} 1 - \sqrt{1-u'(t)^{2}} - \lambda a(t)G(u(t))\Bigr{)} \,dt, \quad u \in \mathcal{X},
\end{equation*}
is not a local minimum for $\lambda > \Lambda^{*}$ (here, $\mathcal{X}:=\{u\in W^{1,\infty}_{T} \colon \|u'\|_{\infty}<1\}$ and $G(u) := \int_{0}^{u} g(s^+)\,ds$). Indeed, it is easy to see that, if $w(t)$ is a positive $T$-periodic solution to \eqref{eq-mu} with $\mu = \mu_{0} < 0$, then $d^{2} \mathcal{J}(u_{s,\lambda})[w,w] < 0$. The variational characterization of ``large'' solutions, on the contrary, seems to be a delicate and challenging problem.
$\hfill\lhd$
\end{remark}

We are now in a position to give the proof of Theorem~\ref{th-sub}.

\begin{proof}[Proof of Theorem~\ref{th-sub}]
We split the proof in some steps.

\smallskip
\noindent
\textit{Definition of the Poincar\'{e} map.} Let us consider the planar system \eqref{system}
and observe that, given any integer $k \geq 1$, any $kT$-periodic solution $(x_{1}(t),x_{2}(t))$ gives rise to a positive $T$-periodic solution $u(t) := x_{1}(t)$ to $(\mathscr{E}_{\lambda})$, by the maximum principles in Appendix~\ref{appendix-A} (just replacing $T$ with $kT$). We also highlight that system \eqref{system} has Hamiltonian structure, namely 
\begin{equation*} 
x_{1}' = \partial_{x_{2}}H(t,x_{1},x_{2}), \qquad x_{2}' = - \partial_{x_{1}}H(t,x_{1},x_{2}),
\end{equation*} 
where $H(t,x_{1},x_{2}) = \int_{0}^{x_{2}} \varphi^{-1}(s)\,ds + \int_{0}^{x_{1}} f_{\lambda}(t,s)\,ds$, with $f_{\lambda}(t,u)$ given by \eqref{def-fl}.

We claim that, for any $t_{0} \in \mathbb{R}$ and any $(x_{0,1},x_{0,2}) \in \mathbb{R}^{2}$, the solution 
$x(\cdot;t_{0},(x_{0,1},x_{0,2}))$ of the Cauchy problem associated with \eqref{system} and with initial condition $(x_{1}(t_{0}),x_{2}(t_{0})) = (x_{0,1},x_{0,2})$ is unique and globally defined on the whole real line. Indeed, the uniqueness directly comes from the local Lipschitz continuity of the right-hand side of \eqref{system}. As for the global continuability, we observe that, due to $| x_{1}'(t) | < 1$, the estimates
\begin{equation*} 
|x_{1}(t)| \leq |x_{1}(t_{0})| + |t - t_{0} | =: m_{t_0}(t),
\end{equation*} 
and, consequently,
\begin{align*} 
|x_{2}(t)| 
&\leq |x_{2}(t_{0})| + \biggl{|} \int_{t_{0}}^{t}f_{\lambda}(\xi,x_{1}(\xi))\,d\xi \biggr{|} 
\\& \leq |x_{2}(t_{0})| + \max\biggl{\{} |t - t_{0} | m_{t_0}(t), \lambda \max_{0 \leq u \leq m_{t_0}(t)} g(u) \int_{t_{0}}^{t} a(\xi) \,d\xi \biggr{\}}
\end{align*} 
hold true for any $t$ in the maximal interval of definition of $(x_{1}(t),x_{2}(t))$, which therefore cannot explode in finite time.

As a consequence, for any integer $k \geq 1$ the Poincar\'{e} map
\begin{equation*} 
\Phi^{k} \colon \mathbb{R}^{2} \to \mathbb{R}^{2}, \quad (x_{0,1},x_{0,2}) \mapsto \bigl{(}x_{1}(kT;0,(x_{0,1},x_{0,2})),x_{2}(kT;0,(x_{0,1},x_{0,2}))\bigr{)}
\end{equation*} 
can be defined. The standard theory of initial value problems ensures that such a map is a global homeomorphism; moreover, by the Hamiltonian structure of system \eqref{system}, $\Phi^{k}$ is an area-preserving map. We also observe that the $T$-periodic solution $u_{s,\lambda}(t)$ to $(\mathscr{E}_{\lambda})$ gives rise to a fixed point $(\bar{x}_{1}, \bar{x}_{2}) := (u_{s,\lambda}(0),\varphi(u'_{s,\lambda}(0))$ of the map 
$\Phi^{k}$. 

For a more direct application of the Poincar\'{e}--Birkhoff fixed point theorem, we find convenient to move such a fixed point to the origin via a linear change of variable. Namely, we deal with the area-preserving homeomorphism 
\begin{equation*} 
\Psi^{k} \colon \mathbb{R}^{2} \to \mathbb{R}^{2}, \quad (y_{1,0},y_{2,0}) \mapsto \Phi^{k}(y_{1,0} + \bar{x}_{1},y_{2,0} + \bar{x}_{2}) - (\bar{x}_{1}, \bar{x}_{2}),
\end{equation*} 
which in turn can be meant as the Poincar\'{e} operator (at time $kT$) of the system
\begin{equation}\label{eq-hs2}
\begin{cases}
\, y_{1}' = \varphi^{-1}(y_{2} + \varphi(u'_{s,\lambda}(t))) - u_{s,\lambda}'(t) \\
\, y_{2}' = -f_{\lambda}(t,y_{1} + u_{s,\lambda}(t)) + f_{\lambda}(t,u_{s,\lambda}(t)),
\end{cases}
\end{equation}
arising from \eqref{system} after the change of variable
\begin{equation*} 
y_{1}(t) = x_{1}(t) - u_{s,\lambda}(t), \qquad y_{2}(t) = x_{2}(t) - \varphi(u'_{s,\lambda}(t)).
\end{equation*} 
We notice that $(0,0)$ is now a solution of the (Hamiltonian) system \eqref{eq-hs2}; as a consequence, any non-trivial solution $(y_{1}(t),y_{2}(t))$ satisfies $(y_{1}(t),y_{2}(t)) \neq 0$ for every $t \in \mathbb{R}$ and thus has an associated winding number around the origin. As shown in \cite{Bo-11,MaReZa-02}, via the Poincar\'{e}--Birkhoff theorem, in this setting $kT$-periodic solutions can be provided whenever a suitable twist condition between the winding numbers of solutions to \eqref{eq-hs2} departing from a small and a large circle centered at the origin is satisfied.

\smallskip
\noindent
\textit{Proof of the twist condition.} We pass to (clockwise) polar coordinates
\begin{equation*} 
(y_{1}(t),y_{2}(t)) = \ell(t) (\cos \theta(t), - \sin \theta(t)), \quad \ell(t) > 0,
\end{equation*} 
and we claim that
\begin{itemize}
\item [$(i)$] there exist an integer $k^{*} \geq 1$ and a sequence of integers $(m_k)_{k \geq k^{*}}$ with $m_k \to +\infty$ such that for any integer $k \geq k^{*}$ there exists a number $\ell_{*} > 0$ such that any solution to \eqref{eq-hs2} with $\ell(0) = \ell_{*}$ satisfies
\begin{equation*}
\theta(kT) - \theta(0) > 2 m_{k} \pi;
\end{equation*}
\item [$(ii)$] for every integer $k \geq k^{*}$ there exists a number $\ell^{*} > 0$ such that every solution to \eqref{eq-hs2} with $\ell(0) = \ell^{*}$ satisfies
\begin{equation}\label{eq-twist-l}
\theta(kT) - \theta(0) < 2\pi.
\end{equation}
\end{itemize} 
To prove $(i)$, we first recall that from \cite[Lemma~3]{MaReZa-02}\footnote{Some care is needed in applying this result to our setting, since the discussion in \cite{MaReZa-02} is developed in the context of Hamiltonian systems with Hamiltonian function continuous in the time variable $t$. However, the continuous dependence arguments used in the proof of \cite[Lemma~3]{MaReZa-02} still hold true, with minor changes, in a Carath\'{e}odory setting. This has been already observed and used, dealing with a semilinear second order equation, in \cite{BoFe-18} (see, precisely, the proof of Claim~(A1) of Proposition~2.1).} it is enough to show that the same property holds for the solutions of the (linear) system obtained by linearizing \eqref{eq-hs2} around the constant solution $(0,0)$, that is,
\begin{equation*} 
\begin{cases}
\, z_{1}' = (\varphi^{-1})'(\varphi(u'_{s,\lambda}(t)))z_{2} \\
\, z_{2}' = -\lambda a(t) g'(u_{s,\lambda}(t)) z_{1}. 
\end{cases}
\end{equation*} 
Now, we observe that this system corresponds to the linear Sturm--Liouville equation 
\begin{equation*}
(\varphi'(u_{s,\lambda}'(t))w')' + \lambda a(t)g'(u_{s,\lambda}(t)) w = 0.
\end{equation*} 
By Lemma~\ref{lem-mu}, the principal eigenvalue associated with the above equation is strictly negative and the thesis thus follows from Corollary~\ref{cor-weak-max-principle} in Appendix~\ref{appendix-A}.

To prove $(ii)$, we first observe that, due to the global continuability of the solutions to \eqref{eq-hs2}, for any $k \geq k^{*}$ there exists $\ell^{*}> 0$ such that any solution to \eqref{eq-hs2} with $\ell(0) = \ell^{*}$ satisfies 
\begin{equation}\label{eq-elastic}
\ell(t) \geq kT, \quad \text{for all $t \in \mathopen{[}0,kT\mathclose{]}$}
\end{equation}
(see the ``elastic property'' in \cite[p.~88]{Bo-11}). Assuming by contradiction that \eqref{eq-twist-l} is violated, we deduce the existence of $t_{1},t_{2} \in \mathopen{[}0,kT\mathclose{]}$ such that
\begin{equation*}
(y_{1}(t_{1}),y_{2}(t_{1})) \in \mathopen{]}-\infty,0\mathclose{[} \times\{0\} \quad \text{ and } \quad (y_{1}(t_{2}),y_{2}(t_{2})) \in \mathopen{]}0,+\infty\mathclose{[} \times\{0\} .
\end{equation*}
Then, from \eqref{eq-elastic} we obtain $\ell(t_{i}) = | y_{1}(t_{i}) | \geq kT$, for $i=1,2$, so that
\begin{equation*}
y_{1}(t_{2}) - y_{1}(t_{1}) \geq 2kT.
\end{equation*}
However, since from the first equation in \eqref{eq-hs2} we have $| y_{1}'(t) | < 2$ for any $t$, we find
\begin{equation*}
y_{1}(t_{2}) - y_{1}(t_{1}) \leq \biggl{|} \int_{t_{1}}^{t_{2}} |y_{1}'(t) | \,dt \biggr{|} < 2kT
\end{equation*}
and a contradiction is reached.

\smallskip
\noindent
\textit{Conclusion.} From the Poincar\'{e}--Birkhoff theorem (in the generalized
version for non-invariant annuli, see \cite{Bo-11,BoZa-13,FoSaZa-12,FoUr-17,MaReZa-02} and the references therein for a detailed description of this technique), it follows that, for any $k \geq k^{*}$ and any $1 \leq j \leq m_{k}$, there exist two $kT$-periodic solutions 
$(y_{1,\lambda,k,j}^{(i)}(t),y_{2,\lambda,k,j}^{(i)}(t))$ to \eqref{eq-hs2} ($i=1,2$), not belonging to the same periodicity class and such that the corresponding angular coordinate satisfies 
\begin{equation*}
\theta(kT) - \theta(0) = 2j\pi.
\end{equation*}
On one hand, this easily implies that, when $j$ and $k$ are relatively prime, these solutions cannot be $lT$-periodic for any $l = 1,\ldots,k-1$. On the other hand, since the angular coordinate $\theta(t)$ is strictly increasing at any $t_{0}$ such that $\theta(t_{0}) \in \tfrac{\pi}{2} + \mathbb{Z}$ (following from the facts that $y_{1}' > 0$ if and only if $y_{2} > 0$, and that $y_{1}' < 0$ if and only if $y_{2} < 0$), it gives that $y_{1,\lambda,k,j}^{(i)}(t)$
has exactly $2j$ zeros on $\mathopen{[}0,kT\mathclose{[}$. Going back to $x_{1,\lambda,k,j}^{(i)}(t) = y_{1,\lambda,k,j}^{(i)}(t) + u_{s,\lambda}(t)$ and defining $u_{s,\lambda,k,j}^{(i)}(t) := x_{1,\lambda,k,j}^{(i)}(t)$, the proof is thus concluded.
\end{proof}

\appendix
\section{Maximum principles}\label{appendix-A}

This appendix is devoted to maximum principles for the boundary value problem 
\begin{equation}\label{pb-appB}
\begin{cases}
\,(\varphi(u'))'+f(t,u)=0 \\
\,\mathfrak{B}(u,u')=(0,0),
\end{cases}
\end{equation}
where 
\begin{itemize}
\item $\varphi\colon I \to J$ is an increasing homeomorphism between two open intervals $I,J\subseteq\mathbb{R}$ both containing $0$, with $\varphi(0)=0$;
\item $f\colon \mathopen{[}0,T\mathclose{]}\times \mathbb{R}\to \mathbb{R}$ is an $L^{1}$-Ca\-ra\-th\'{e}o\-dory function, that is, $f(t,u)$ is measurable in $t$ for every $u$, continuous in $u$ for a.e.~$t$ and, for every $r > 0$, there exists $\zeta_{r} \in L^{1}(\mathopen{[}0,T\mathclose{]},\mathbb{R}^{+})$ such that $| f(t,u) | \leq \zeta_{r}(t)$ for a.e.~$t$ and $|u| \leq r$;
\item $\mathfrak{B}(u,u')$ defines the boundary conditions, precisely
\begin{equation}\label{BC}
\begin{aligned}
\mathfrak{B}(u,u')\in\Bigl{\{}&\bigl{(}u(0),u(T)\bigr{)}, \bigl{(}u(0),u'(T)\bigr{)}, \bigl{(}u'(0),u(T)\bigr{)},
\\ &\bigl{(}u'(0),u'(T)\bigr{)},\bigl{(}u(0)-u(T),u'(0)-u'(T)\bigr{)}\Bigr{\}}.
\end{aligned}
\end{equation}
\end{itemize}
We notice that we cover the Dirichlet, Neumann and periodic boundary value problems.

The first result is a weak maximum principle that ensures the non-negativity of non-constant solutions to \eqref{pb-appB}. 

\begin{theorem}[Weak maximum principle]\label{weak-max-principle}
Let $f\colon \mathopen{[}0,T\mathclose{]}\times \mathbb{R}\to \mathbb{R}$ be an $L^{1}$-Ca\-ra\-th\'{e}o\-dory function such that
\begin{equation*}
f(t,u)\geq0, \quad \text{a.e.~$t\in\mathopen{[}0,T\mathclose{]}$, for all $u<0$.}
\end{equation*}
Then, if $u(t)$ is a solution to \eqref{pb-appB}, either $u(t)$ is non-negative on $\mathopen{[}0,T\mathclose{]}$ (namely $u(t)\geq0$ for all $t\in\mathopen{[}0,T\mathclose{]}$) or $u\equiv u_{0}$ with $u_{0}<0$.
\end{theorem}

We remark that the second alternative occurs if and only if the boundary conditions are of Neumann and periodic type along with $f(t,u_{0})\equiv 0$.

\begin{proof}
Let $u(t)$ be a non-constant solution of \eqref{pb-appB} and let $\hat{t}\in\mathopen{[}0,T\mathclose{]}$ be a minimum point of $u(t)$. By contradiction, we suppose that $u(\hat{t}) < 0$. We first claim that $u'(\hat{t}) = 0$: indeed, this is straightforward if $\hat{t} \in \mathopen{]}0,T\mathclose{[}$, while it follows from the boundary conditions if $\hat{t} \in \{0,T\}$ (in the periodic case, one observes that
if $u'(0) = u'(T) \neq 0$, then $u(0) = u(T)$ cannot be an extreme value of $u(t)$). 
Let $\mathopen{]}t_{1},t_{2}\mathclose{[}\subseteq\mathopen{]}0,T\mathclose{[}$ be the maximal open interval containing $\hat{t}$ with $u(t)<0$, for all $t\in\mathopen{]}t_{1},t_{2}\mathclose{[}$.
From
\begin{equation*}
\varphi(u'(\hat{t}_{2}))-\varphi(u'(\hat{t}_{1})) = -\int_{\hat{t}_{1}}^{\hat{t}_{2}} f(\xi,u(\xi)) \,d\xi \leq 0,
\end{equation*}
for all $\hat{t}_{1},\hat{t}_{2}$ with $t_{1} \leq \hat{t}_{1} < \hat{t}_{2} \leq t_{2}$, and the fact that $\varphi$ is an increasing homeomorphism, we immediately deduce that the map $t\mapsto u'(t)$ is non-increasing on $\mathopen{]}t_{1},t_{2}\mathclose{[}$. Thus,
$u'(t) \geq 0$ for $t \in \mathopen{]}t_{1},\hat{t}\mathclose{]}$ and $u'(t) \leq 0$ for $t \in \mathopen{[}\hat{t},t_{2}\mathclose{[}$.
Therefore, since $\hat{t}$ is a minimum point of $u(t)$, we have $u(t) \equiv u(\hat{t})$ for every $t \in \mathopen{]}t_{1},t_{2}\mathclose{[}$ and so $ \mathopen{[}t_{1},t_{2}\mathclose{]}= \mathopen{[}0,T\mathclose{]}$.
Summing up, $u(t)$ is constant on the whole $\mathopen{[}0,T\mathclose{]}$, a contradiction.
\end{proof}

A straightforward corollary of Theorem~\ref{weak-max-principle} is the following.

\begin{corollary}\label{cor-weak-max-principle}
Let $f\colon \mathopen{[}0,T\mathclose{]}\times \mathbb{R}\to \mathbb{R}$ be an $L^{1}$-Ca\-ra\-th\'{e}o\-dory function such that $f(t,u)>0$, for a.e.~$t\in\mathopen{[}0,T\mathclose{]}$, for all $u<0$. Then, every solution $u(t)$ to \eqref{pb-appB} is non-negative on $\mathopen{[}0,T\mathclose{]}$.
\end{corollary}

The second result is a strong maximum principle that ensures the positivity of non-trivial non-negative $T$-periodic solutions to \eqref{pb-appB}.
To state it, we need to introduce the following condition: 
\begin{itemize}
\item[$(\varphi_{*})$] for all $\sigma>1$ it holds that
\begin{equation*}
\limsup_{\xi\to0^{+}}\dfrac{\varphi^{-1}(\sigma \xi)}{\varphi^{-1}(\xi)} < + \infty.
\end{equation*}
\end{itemize}
According to \cite{GaMaZa-93}, such an assumption can be meant as an upper $\sigma$-condition for $\varphi^{-1}$ as $\xi \to 0^+$; notice that
it is fulfilled by $\varphi$ defined in \eqref{phi-L} (corresponding to the Minkowksi-curvature operator). With this in mind, the following result (to be applied to equation \eqref{eq-fl}, with $\varphi$ defined as in \eqref{phi-L}, $a \in L^{1}(\mathopen{[}0,T\mathclose{]})$ and $g(u)/u$ bounded in a right neighborhood of zero) holds true.

\begin{theorem}[Strong maximum principle]\label{strong-max-principle}
Assume that $\varphi$ satisfies $(\varphi_{*})$ and let $f\colon \mathopen{[}0,T\mathclose{]}\times \mathbb{R}\to \mathbb{R}$ be an $L^{1}$-Ca\-ra\-th\'{e}o\-dory function such that 
\begin{itemize}
\item [$(i)$] $f(t,0)=0$, for a.e.~$t\in\mathopen{[}0,T\mathclose{]}$;
\item [$(ii)$] there exists $\gamma\in L^{1}(\mathopen{[}0,T\mathclose{]},{\mathbb{R}}^{+})$ such that
\begin{equation*}
\limsup_{u\to0^{+}}\dfrac{|f(t,u)|}{\varphi(u)} \leq \gamma(t), \quad \text{uniformly a.e.~$t\in\mathopen{[}0,T\mathclose{]}$.}
\end{equation*}
\end{itemize}
Then, every non-trivial non-negative solution $u(t)$ to \eqref{pb-appB} satisfies:
\begin{itemize} 
\item $u(t)>0$ for all $t\in\mathopen{]}0,T\mathclose{[}$, if $\mathfrak{B}(u,u') = (u(0),u(T))$;
\item $u(t)>0$ for all $t\in\mathopen{]}0,T\mathclose{]}$, if $\mathfrak{B}(u,u') = (u(0),u'(T))$;
\item $u(t)>0$ for all $t\in\mathopen{[}0,T\mathclose{[}$, if $\mathfrak{B}(u,u') = (u'(0),u(T))$;
\item $u(t)>0$ for all $t\in\mathopen{[}0,T\mathclose{]}$, if $\mathfrak{B}(u,u') = (u'(0),u'(T))$
or $\mathfrak{B}(u,u') = (u(0)-u(T),u'(0)-u'(T))$.
\end{itemize}
\end{theorem}

\begin{proof}
First of all, we remark the following two facts. From hypothesis $(\varphi_{*})$, we deduce that for every $\sigma>1$ there exist $\eta>0$ and $C_{\sigma}>0$ such that
\begin{equation*}
\varphi^{-1}(\sigma \xi) \leq C_{\sigma} \varphi^{-1}(\xi), \quad \text{for all $\xi$ with $\sigma \xi\in\mathopen{[}0,\eta\mathclose{]}\cap J$.}
\end{equation*}
From hypotheses $(i)$ and $(ii)$, we deduce that there exists $\delta>0$ such that
\begin{equation*}
|f(t,u)|\leq \gamma_{1}(t)\varphi(u), \quad \text{a.e.~$t\in\mathopen{[}0,T\mathclose{]}$, for all $u\in\mathopen{[}0,\delta\mathclose{]}\cap I$.}
\end{equation*}
Without loss of generality we can assume $\delta\leq\eta$. Let us set $K_{1}:=C_{\sigma}$ with $\sigma=\|\gamma_{1}\|_{L^{1}}$.

Let $u(t)$ be a non-trivial non-negative solution $u(t)$ to \eqref{pb-appB}. We are going to show that if $u(t_0) = u'(t_0) = 0$ for some $t_0 \in \mathopen{[}0,T\mathclose{]}$, then $u\equiv0$ on $\mathopen{[}0,T\mathclose{]}$. It is easy to see that, arguing by contradiction, this implies the thesis for any boundary condition.

We proceed similarly as in \cite[Section~2]{MaNjZa-95}, where a maximum principle for a Dirichlet problem involving a $p$-Laplacian operator is proposed, by proving that there exists $\varepsilon>0$, independent of $u(t)$ and $t_{0}$, such that $u(t)=0$, for all $t\in\mathopen{[}t_{0}-\varepsilon,t_{0}+\varepsilon\mathclose{]}$ (or for all $t\in\mathopen{[}t_{0},t_{0}+\varepsilon\mathclose{]}$, or for all $t\in\mathopen{[}t_{0}-\varepsilon,t_{0}\mathclose{]}$, if $t_{0}=0$ or $t_{0}=T$, respectively). By repeating the argument a finite number of times, we conclude.

For the sake of simplicity in the exposition, we assume that $t_{0}\in\mathopen{]}0,T\mathclose{[}$.
We fix $\varepsilon\in\mathopen{]}0, 1/K_{1} \mathclose{[}$ and we assume by contradiction that $u\not\equiv 0$ on 
$\mathopen{[}t_{0}-\varepsilon,t_{0}+\varepsilon\mathclose{]}$. Then, there exists $\varepsilon_{1}\in\mathopen{]}0, \varepsilon\mathclose{]}$ such that 
\begin{equation*}
0 < M:= \max_{t \in\mathopen{[}t_{0}-\varepsilon_{1},t_{0}+\varepsilon_{1}\mathclose{]}}u(t) \leq\delta, \quad \text{for all $t\in\mathopen{[}t_{0}-\varepsilon_{1},t_{0}+\varepsilon_{1}\mathclose{]}$.}
\end{equation*}
Then, for all $t\in\mathopen{[}t_{0}-\varepsilon_{1},t_{0}+\varepsilon_{1}\mathclose{]}$, the following holds
\begin{equation*}
\begin{aligned}
u(t) &= \int_{t_{0}}^{t} \varphi^{-1} \biggl{(} \int_{t_{0}}^{\xi} -f(s,u(s)) \,ds \biggr{)} \,d\xi 
\leq \int_{t_{0}}^{t} \varphi^{-1} \biggl{(} \int_{t_{0}}^{\xi} \gamma_{1}(s)\varphi(M) \,ds \biggr{)} \,d\xi \\
& \leq \int_{t_{0}}^{t} \varphi^{-1} \bigl{(} \|\gamma_{1}\|_{L^{1}} \varphi(M) \bigr{)} \,d\xi
\leq \varepsilon_{1} K_{1} M.
\end{aligned}
\end{equation*}
Then $M \leq \varepsilon_{1} K_{1}M$
and so $1\leq \varepsilon_{1} K_{1} \leq \varepsilon K_{1} < 1$, a contradiction.
Hence the proof is completed.
\end{proof}

\section{A continuation theorem}\label{appendix-B}

In this second appendix we present a continuation theorem for the first order system
\begin{equation}\label{B-system}
\begin{cases}
\, x_{1}' = \varphi^{-1}(x_{2}) \\
\, x_{2}' = - \vartheta f_{\lambda}(t,x_{1}),
\end{cases}
\end{equation}
where $\vartheta\in\mathopen{[}0,1\mathclose{]}$ and 
\begin{itemize}
\item $\varphi\colon I \to \mathbb{R}$ is an increasing homeomorphism defined on an open (possibly bounded) interval $I\subseteq\mathbb{R}$ containing $0$, with $\varphi(0)=0$;
\item $f\colon \mathopen{[}0,T\mathclose{]}\times \mathbb{R}\to \mathbb{R}$ is an $L^{1}$-Ca\-ra\-th\'{e}o\-dory function.
\end{itemize}
We refer to Section~\ref{section-2} for the abstract coincidence degree setting and in particular for the definitions of the Banach spaces $X$, $Z$ and of the linear operators $L$, $J$, $Q$ and of the nonlinear Nemytskii operator $N$ induced by $\varphi^{-1}$ and $-f$ (actually, in 
Section~\ref{section-2} the case $f = f_{\lambda}$, with $f_{\lambda}$ defined in \eqref{def-fl}, was considered, but the arguments work the same for a general $L^{1}$-Ca\-ra\-th\'{e}o\-dory function $f$).

Our purpose is to provide a generalization of the classical Mawhin's theorem (cf.~\cite[Th\'{e}or\`{e}me~2]{Ma-69} or \cite[Theorem~4.1]{Ma-93}) in the framework of $T$-periodic problem associated with \eqref{B-system}. More precisely, the result we are going to illustrate allows us to reduce the computation of the coincidence degree $\mathrm{D}_{L}(L-N,\Omega)$ on an open (possibly unbounded) set of the form $\Omega:=\Omega_{1}\times\Omega_{2}$ to the computation of the finite-dimensional Brouwer degree (henceforth denoted by $\mathrm{deg}_{\mathrm{B}}$) of the average map $f^{\#} \colon \mathbb{R} \to \mathbb{R}$ defined by
\begin{equation}\label{def-fsharp}
f^{\#}(s) := \dfrac{1}{T} \int_{0}^{T}f(t,s) \,dt.
\end{equation}

In \cite[Section~3]{FeZa-17tmna} the authors have provided analogous results for the $T$-periodic problem associated with a cyclic feedback system of the form
\begin{equation*}
\begin{cases}
\, x_{1}' = g_{1}(x_{2}) \\
\, x_{2}' = g_{2}(x_{3}) \\
\qquad \vdots \\
\, x_{n-1}' = g_{n-1}(x_{n}) \\
\, x_{n}' = -\vartheta h(t,x_{1},\ldots,x_{n}),
\end{cases}
\end{equation*}
which clearly includes \eqref{B-system} as a special case. Our continuation theorem can be meant as a variant of \cite[Theorem~3.10]{FeZa-17tmna}, which however deals with a bounded set $\Omega$ and a homeomorphism $\varphi$ of the whole real line.
Due to these differences, and since \cite[Theorem~3.10]{FeZa-17tmna} has been obtained as a corollary of a series of more abstract results, we give here the precise statement together with a self-contained proof.

\begin{theorem}\label{th-B.1}
Let $\Omega:=\Omega_{1}\times\Omega_{2} \subseteq X$ be an open set such that $0\in\Omega_{2}$.
Suppose that the following conditions hold.
\begin{itemize}
\item[$(i)$] There exists a bounded set $B \subseteq X$ such that, for every $\vartheta\in\mathopen{]}0,1\mathclose{]}$, if $x\in\overline{\Omega}$ is a $T$-periodic solution to \eqref{B-system}, then $x \in B \cap \Omega$.
\item [$(ii)$] The set $(f^{\#})^{-1}(0)\cap\Omega_{1}$ is compact, where $f^{\#}$ is defined in \eqref{def-fsharp}.
\end{itemize}
Then
\begin{equation*}
\mathrm{D}_{L}(L-N,\Omega) = \mathrm{deg}_{\mathrm{B}}(f^{\#},\Omega_{1}\cap\mathbb{R},0).
\end{equation*}
\end{theorem}

\begin{proof}
Let $\eta\colon \mathbb{R}^{2} \to \mathbb{R}^{2}$ be defined as
\begin{equation*}
\eta(s):=\bigl{(}-\varphi^{-1}(s_{2}),f^{\#}(s_{1})\bigr{)}, \quad s=(s_{1},s_{2})\in\mathbb{R}^{2}.
\end{equation*}
We first claim that
\begin{equation}\label{eq-step1}
\mathrm{D}_{L}(L-N,\Omega) = \mathrm{deg}_{\mathrm{B}}(\eta,\Omega \cap \mathbb{R}^{2},0).
\end{equation}
To show this, we introduce the homotopy $\mathcal{N}\colon X\times\mathopen{[}0,1\mathclose{]}\to Z$ defined as
\begin{equation*}
\mathcal{N}(x,\vartheta) := (N_{1}x_{2},\vartheta N_{2}x_{1}+(1-\vartheta) Q_{2}N_{2}x_{1}), \quad x=(x_{1},x_{2})\in X.
\end{equation*}
Letting
\begin{equation*}
\mathcal{S}_{\vartheta}:=\bigl{\{}x\in\Omega\cap \mathrm{dom}\,L \colon Lx=\mathcal{N}(x,\vartheta)\bigr{\}},
\end{equation*}
we prove that the set $\Sigma := \bigcup_{\vartheta\in\mathopen{[}0,1\mathclose{]}} \mathcal{S}_{\vartheta}$ is a compact subset of $\Omega$.

Let us first consider the case $\vartheta \in \mathopen{]}0,1\mathclose{]}$. Then, $x = (x_{1},x_{2}) \in \mathcal{S}_{\theta}$ solves the system 
\begin{equation*}
\begin{cases}
\, L_{1} x_{1} = N_{1} x_{2} \\
\, L_{2} x_{2} = \vartheta N_{2} x_{1} + (1-\vartheta) Q_{2}N_{2}x_{1}.
\end{cases}
\end{equation*}
By applying the operator $Q_{2}$ to the second equation, we obtain $Q_{2} N_{2}x_{1} = 0$, so that $x$ is a $T$-periodic solution to \eqref{B-system}. By assumption $(i)$, we deduce $\mathcal{S}_{\theta} \subseteq B$ for every $\vartheta \in \mathopen{]}0,1\mathclose{]}$.

On the other hand, let $\vartheta = 0$. Then, $x = (x_{1},x_{2}) \in \mathcal{S}_0$ satisfies 
\begin{equation*}
\begin{cases}
\, L_{1} x_{1} = N_{1} x_{2} \\
\, L_{2} x_{2} = Q_{2}N_{2}x_{1}.
\end{cases}
\end{equation*}
Since $Q_{2} N_{2}x_{1} \in \mathrm{coker}\,L_{2}$, from the second equation we get $Q_{2}N_{2}x_{1}=0$ and $x_{2} \in \ker L_{2}$.
From the first equation, we have $Q_{1}N_{1}x_{2}=0$ and, moreover, since 
\begin{equation*}
\mathrm{Im}\,L_{1}\cap N_{1}(\Omega_{2} \cap \ker L_{2}) = \{0\},
\end{equation*}
we immediately obtain $x_{1}\in\ker L_{1}$. Summing up,
\begin{equation*}
\mathcal{S}_{0} = \bigl{\{}x\in\overline{\Omega}\cap \ker L \colon Q\mathcal{N}(x,0)=0\bigr{\}}.
\end{equation*}
Remarking that
\begin{equation}\label{eq-JQN}
\begin{aligned}
&Q_{1}N_{1} (s_{2})= \varphi^{-1}(s_{2}), &\text{for all $s_{2}\in\Omega_{2}\cap \ker L_{2} \cong \Omega_{2}\cap \mathbb{R}$,} \\
&Q_{2}N_{2} (s_{1})= -f^{\#}(s_{1}), &\text{for all $s_{1}\in\Omega_{1}\cap \ker L_{1} \cong \Omega_{1}\cap \mathbb{R}$,}
\end{aligned}
\end{equation}
and using the fact that $\varphi$ is a homeomorphism with $\varphi(0)=0$, we have
\begin{equation*}
\mathcal{S}_{0} = \bigl{\{}(\omega,0) \in\Omega\cap \mathbb{R}^{2}\colon \omega\in(f^{\#})^{-1}(0)\bigr{\}} = ((f^{\#})^{-1}(0)\cap\Omega_{1})\times\{0\}.
\end{equation*}
By hypothesis $(ii)$, we then deduce that $\mathcal{S}_{0}$ is compact.

By the above discussion, the set $\Sigma$ is bounded. From the $L$-complete continuity of $\mathcal{N}(\cdot,\vartheta)$, we deduce that $\Sigma$ is compact and, using again assumptions $(i)$ and $(ii)$, we finally obtain $\Sigma \subseteq \Omega$.

Via the homotopic invariance property of the degree, we then find
\begin{equation*}
\mathrm{D}_{L}(L-N,\Omega) = \mathrm{D}_{L}(L-\mathcal{N}(\cdot,1),\Omega) = \mathrm{D}_{L}(L-\mathcal{N}(\cdot,0),\Omega).
\end{equation*}
Next, using the reduction formula of the degree and the previous considerations about the solution set for $\vartheta = 0$, we obtain
\begin{equation*}
\mathrm{D}_{L}(L-N,\Omega) = \mathrm{deg}_{\mathrm{B}}(-JQN|_{\Omega\cap\ker L},\Omega \cap \ker L,0),
\end{equation*}
thus proving \eqref{eq-step1} in view of \eqref{eq-JQN}.

We now claim that
\begin{equation}\label{eq-step2}
\mathrm{deg}_{\mathrm{B}}(\eta,\Omega \cap \mathbb{R}^{2},0) = - \mathrm{deg}_{\mathrm{B}}(-\varphi^{-1},\Omega_{2}\cap\mathbb{R},0) \, \mathrm{deg}_{\mathrm{B}}(f^{\#},\Omega_{1}\cap\mathbb{R},0).
\end{equation}
To show this, let $\tilde{\eta}\colon \mathbb{R}^{2} \to \mathbb{R}^{2}$ be defined as
\begin{equation*}
\tilde{\eta}(s):=\bigl{(}f^{\#}(s_{1}),-\varphi^{-1}(s_{2})\bigr{)}, \quad s=(s_{1},s_{2})\in\mathbb{R}^{2}.
\end{equation*}
We notice that $\tilde{\eta}(s) = (P \eta) (s)$ for all $s \in \mathbb{R}^{2}$,
where
\begin{equation*}
P =
\begin{pmatrix}
0 & 1 \\
1 & 0
\end{pmatrix}
\in \mathbb{R}^{2\times 2}
\end{equation*}
is a permutation matrix.
Therefore, by a standard property of the Brouwer degree of a composition of maps, we have
\begin{align*}
\mathrm{deg}_{\mathrm{B}}(\tilde{\eta},\Omega \cap \mathbb{R}^{2},0) & = \mathrm{deg}_{\mathrm{B}}(P \eta,\Omega \cap \mathbb{R}^{2},0) \\
& = \mathrm{sign}(\det (P)) \, \mathrm{deg}_{\mathrm{B}}(\eta,\Omega \cap \mathbb{R}^{2},0) \\
& = - \mathrm{deg}_{\mathrm{B}}(\eta,\Omega \cap \mathbb{R}^{2},0)
\end{align*}
and the claim follows from the product property (cf.~\cite[Theorem~9.7]{Br-14}) of the Brouwer degree.

From \eqref{eq-step1} and \eqref{eq-step2}, we then have
\begin{equation*}
\mathrm{D}_{L}(L-N,\Omega) = - \mathrm{deg}_{\mathrm{B}}(-\varphi^{-1},\Omega_{2}\cap\mathbb{R},0) \, \mathrm{deg}_{\mathrm{B}}(f^{\#},\Omega_{1}\cap\mathbb{R},0).
\end{equation*}
Observing that $\mathrm{deg}_{\mathrm{B}}(-\varphi^{-1},\Omega_{2}\cap\mathbb{R},0)=-1$, coming from the fact that $-\varphi^{-1}$ is a decreasing homeomorphisms and $0\in\Omega_{2}$, we conclude.
\end{proof}

\section{The principal eigenvalue of Sturm--Liouville operators}\label{appendix-C}

In this section we provide a dynamical characterization of the principal eigenvalue of the $T$-periodic problem associated with the linear equation
\begin{equation}\label{eq-sturm}
(p(t)w')' + (\mu + q(t)) w = 0,
\end{equation}
where $p \in L^{\infty}_{T}$ with $p_{*} := \mathrm{ess}\inf p(t) > 0$ and $q \in L^{1}_{T}$. As a consequence, we will be able to prove that, whenever such a principal eigenvalue is strictly negative, solutions to the equation $(p(t)w')' + q(t)w = 0$ perform more and more rotations around the origin (of the phase-plane) as the time interval becomes larger and larger (see Corollary~\ref{cor-mu} below). 

In the particular case $p(t) \equiv 1$, in \cite[p.~6]{BoFe-18} such a result has been explained to follow as a consequence of the oscillation theory for Hill's equation (precisely, by exploiting the relationship between the disconjugacy of the equation and its Moser rotation number). 
The fact that such a statement is still valid in the more general case of equation \eqref{eq-sturm} is highly expected; however, it is not easy to find an appropriate reference in the literature. For this reason, we have decided to prove it in a 
self-contained way, by following the so-called rotation number approach to the periodic spectrum 
(cf.~\cite{GaZh-00,Za-03}) and thus giving an alternative proof even for the case $p(t) \equiv 1$.

In order to state our results, we rewrite equation \eqref{eq-sturm} (of course, solutions are meant as locally absolutely continuous functions 
$w(t)$ such that $p(t)w'(t)$ is locally absolutely continuous and the differential equation is satisfied almost everywhere) as the equivalent first order planar system 
\begin{equation}\label{sys-sturm}
\begin{cases}
\, z_{1}' = \dfrac{z_{2}}{p(t)}\\
\, z_{2}' = -(\mu+q(t))z_{1},
\end{cases}
\end{equation}
we pass to (clockwise) polar coordinates 
\begin{equation}\label{modpol}
(z_{1}(t),z_{2}(t)) = \ell(t) (\cos \theta(t), -\sin \theta(t)), \quad \ell(t) > 0,
\end{equation}
and we denote by $\theta_{\mu}(t;\theta_{0})$ the angular coordinate of the solution to \eqref{sys-sturm} with initial condition $(z_{1}(0),z_{2}(0)) = (\cos\theta_{0},-\sin\theta_{0})$. For further convenience we also observe that $\theta_{\mu}(t;\theta_{0})$ solves the differential equation
\begin{equation}\label{eq-theta}
\theta_{\mu}'(t;\theta_{0}) = \frac{\sin^{2} \theta_{\mu}(t;\theta_{0})}{p(t)} + (\mu + q(t)) \cos^{2} \theta_{\mu}(t;\theta_{0})
=: \Theta(t,\theta_{\mu}(t;\theta_{0});\mu).
\end{equation} 
With this notation, the following preliminary result can be stated.

\begin{lemma}\label{lem-f}
The function $f\colon\mathbb{R} \to \mathbb{R}$ defined by 
\begin{equation*}
f(\mu) := \min_{\theta_{0} \in \mathopen{[}0,2\pi\mathclose{[}} \bigl{(}\theta_{\mu}(T;\theta_{0}) - \theta_{0}\bigr{)}
\end{equation*}
is continuous, strictly increasing and such that
\begin{equation}\label{eq-f-}
\lim_{\mu \to - \infty} f(\mu) < 0
\end{equation}
and
\begin{equation}\label{eq-f+}
\lim_{\mu \to + \infty} f(\mu) = +\infty.
\end{equation}
\end{lemma}

\begin{proof}
We split the proof in some steps.

\smallskip
\noindent
\textit{Continuity.}
This follows from the facts that the map $(\mu,\theta_{0}) \mapsto \theta_{\mu}(T;\theta_{0})$ is continuous (as a consequence of the continuous dependence of solutions to Cauchy problems on parameters and initial values) and that the minimum with respect to $\theta_{0}$ is taken on the compact set $\mathbb{R}/2\pi\mathbb{Z}$. 

\smallskip
\noindent
\textit{Monotonicity.}
We are going to show that, for any $\theta_{0} \in \mathopen{[}0,2\pi\mathclose{[}$ and for any $\mu_{1}, \mu_{2} \in \mathbb{R}$ with $\mu_{1} > \mu_{2}$, it holds that 
\begin{equation}\label{eq-theta-mu}
\theta_{\mu_{1}}(T;\theta_{0}) > \theta_{\mu_{2}}(T;\theta_{0}).
\end{equation}
To prove this, we observe that, using \eqref{eq-theta}, 
the function $\theta(t) := \theta_{\mu_{1}}(t;\theta_{0}) - \theta_{\mu_{2}}(t;\theta_{0})$ solves
\begin{align*}
\theta'(t) & = \Theta(t,\theta_{\mu_{1}}(t;\theta_{0});\mu_{1}) - \Theta(t,\theta_{\mu_{2}}(t;\theta_{0});\mu_{2}) \\
& = \Theta(t,\theta_{\mu_{1}}(t;\theta_{0});\mu_{1}) - \Theta(t,\theta_{\mu_{1}}(t;\theta_{0});\mu_{2}) \\
& \quad + \Theta(t,\theta_{\mu_{1}}(t;\theta_{0});\mu_{2}) - \Theta(t,\theta_{\mu_{2}}(t;\theta_{0});\mu_{2}) \\
& = \alpha(t) + \beta(t) \theta(t), 
\end{align*}
where 
\begin{equation*}
\alpha(t) = (\mu_{1} - \mu_{2}) \cos^{2} \theta_{\mu_{1}}(t;\theta_{0})
\end{equation*}
and 
\begin{equation*}
\beta(t) = \frac{\partial\Theta}{\partial \theta}(t,\xi(t);\mu_{2}), \quad \text{for some $\xi(t) \in \mathopen{[} \theta_{\mu_{1}}(t;\theta_{0}),\theta_{\mu_{2}}(t;\theta_{0}) \mathclose{]}$.}
\end{equation*}
Therefore, for every $t \in \mathbb{R}$,
\begin{equation*}
\theta(t) = \int_{0}^{t} \alpha(s) e^{\int_s^{t} \beta(\tau)\,d\tau} \,ds.
\end{equation*}
Observing that, by \eqref{eq-theta}, the function $t \mapsto \cos^{2} \theta_{\mu_{1}}(t;\theta_{0})$ cannot vanish on any subinterval of the real line, we obtain $\alpha(t) \geq 0$ for any $t \in \mathopen{[}0,T\mathclose{]}$ with $\int_{0}^{T} \alpha(t)\,dt > 0$. As a consequence, $\theta(T) > 0$ and 
\eqref{eq-theta-mu} is thus proved.

\smallskip
\noindent
\textit{Verification of \eqref{eq-f-}.}
We are going to show that, if $\mu$ is negative and large enough, then the solution 
$(z_{1}(t),z_{2}(t))$ of \eqref{sys-sturm} with $(z_{1}(0),z_{2}(0)) = (1,0)$ satisfies
\begin{equation*}
\min_{t \in \mathopen{[}0,T\mathclose{]}}z_{1}(t) > 0 \quad \text{ and } \quad z_{2}(T) > 0.
\end{equation*}
This easily implies $\theta_{\mu}(T;0) \in \mathopen{]}-\tfrac{\pi}{4},0\mathclose{[}$ and, finally, $f(\mu) < 0$ for $\mu$ negative and large enough. Below, we follow closely the arguments used in the proof of \cite[Lemma~2]{Za-03}.

We start by proving that $z_{1}(t) > 0$ for all $t \in \mathopen{[}0,T\mathclose{]}$. By contradiction, assume that this is not true and let $t_{1} \in \mathopen{]}0,T\mathclose{]}$ be the first 
zero of $z_{1}(t)$. Then, for any $t \in\mathopen{[}0,t_{1}\mathclose{]}$,
\begin{align*}
0 \leq z_{1}(t) & \leq \int_{t}^{t_{1}} | z_{1}'(s) | \,ds \leq \sqrt{t_{1}-t} \, \biggl{(}\int_{t}^{t_{1}} z_{1}'(s)^{2} \,ds \biggr{)}^{\frac{1}{2}} 
\\ & \leq \sqrt{T} \biggl{(}\int_{0}^{t_{1}} z_{1}'(s)^{2} \,ds \biggr{)}^{\frac{1}{2}}. 
\end{align*}
Integrating by parts on $\mathopen{[}0,t_{1}\mathclose{]}$ the right-hand side of the equality $- z_{1} z_{2}' = (\mu + q(t))z_{1}^{2}$ and using the above inequality, we find
\begin{equation*}
\int_{0}^{t_{1}} p(t)z_{1}'(t)^{2} \,dt = \int_{0}^{t_{1}} (\mu + q(t)) z_{1}(t)^{2} \,dt \leq T \int_{0}^{t_{1}} (\mu + q(t))\,dt \int_{0}^{t_{1}} z_{1}'(t)^{2} \,dt.
\end{equation*}
As a consequence
\begin{equation*}
0 < p_{*} \leq T \int_{0}^{t_{1}} (\mu + q(t))\,dt \leq T \int_{0}^{T} (\mu + q(t))^{+} \,dt,
\end{equation*}
where $(\mu + q(t))^{+}$ is the positive part of $\mu + q(t)$. Since, by the monotone convergence theorem, $\int_{0}^{T} (\mu + q(t))^{+} \,dt \to 0$ as $\mu \to -\infty$, a contradiction is reached. 

We now prove that $z_{2}(T) > 0$. We have
\begin{align*}
z_{2}(T) & = \int_{0}^{T} z_{2}'(t)\,dt = - \int_{0}^{T} (\mu + q(t)) z_{1}(t) \,dt \\
& \geq \min_{t \in \mathopen{[}0,T\mathclose{]}}z_{1}(t) \int_{0}^{T} (\mu + q(t))^{-} \,dt - \max_{t \in \mathopen{[}0,T\mathclose{]}}z_{1}(t) \int_{0}^{T} (\mu + q(t))^{+} \,dt,
\end{align*}
where $(\mu + q(t))^{-}$ is the negative part of $\mu + q(t)$. On the other hand,
\begin{equation*}
 \max_{t \in \mathopen{[}0,T\mathclose{]}}z_{1}(t) - \min_{t \in \mathopen{[}0,T\mathclose{]}}z_{1}(t) \leq \int_{0}^{T} | z_{1}'(t) | \,dt \leq \sqrt{T} \biggl{(}\int_{0}^{T} z_{1}'(t)^{2} \,dt \biggr{)}^{\frac{1}{2}}
\end{equation*}
and using the equation we can further estimate
\begin{align*}
\int_{0}^{T} z_{1}'(t)^{2} \,dt & \leq \frac{1}{p_{*}} \int_{0}^{T} p(t)z_{1}'(t)^{2} \,dt = 
\frac{1}{p_{*}} \int_{0}^{T} (\mu + q(t)) z_{1}(t)^{2} \,dt \\ & \leq \frac{1}{p_{*}} \biggl{(}\max_{t \in \mathopen{[}0,T\mathclose{]}}z_{1}(t) \biggr{)}^{2}
\int_{0}^{T} (\mu + q(t))^{+} \,dt.
\end{align*}
By combining the above two inequalities, we find
\begin{equation*}
\left(1 - \sqrt{\frac{T}{p_{*}} \int_{0}^{T} (\mu + q(t))^{+} \,dt} \right) \max_{t \in \mathopen{[}0,T\mathclose{]}}z_{1}(t) \leq \min_{t \in \mathopen{[}0,T\mathclose{]}}z_{1}(t). 
\end{equation*}
In conclusion,
\begin{align*}
z_{2}(T) & \geq \min_{t \in \mathopen{[}0,T\mathclose{]}}z_{1}(t) \left[\rule{0cm}{25pt}\right.\int_{0}^{T} (\mu + q(t))^{-} \,dt \\ 
& \quad \left.- \int_{0}^{T} (\mu + q(t))^{+} \,dt \left(1 - \sqrt{\frac{T}{p_{*}} \int_{0}^{T} (\mu + q(t))^{+} \,dt} \right)^{-1}\rule{0cm}{25pt}\right].
\end{align*}
Since, by the monotone convergence theorem, $\int_{0}^{T} (\mu + q(t))^{-} \,dt \to +\infty$ and $\int_{0}^{T} (\mu + q(t))^{+} \,dt \to 0$ as $\mu \to -\infty$, the term in the square bracket is strictly positive and the conclusion follows.

\smallskip
\noindent
\textit{Verification of \eqref{eq-f+}.}
For this part of the proof, we use a trick based on the introduction of a system of ``deformed'' polar coordinates in the phase-plane. More precisely, instead of \eqref{modpol}, for $\mu \geq 1$ we now write 
\begin{equation*}
(z_{1}(t),z_{2}(t)) = \ell_{\mu}(t)(\cos \vartheta_{\mu}(t),- \sqrt{\mu}\sin \vartheta_{\mu}(t)), \quad \ell_{\mu}(t) > 0,
\end{equation*}
The angular coordinates $\vartheta_{\mu}$ and $\theta$ are different in general, but they fulfill the property
\begin{equation*}
\vartheta_{\mu} = k \pi \; \Longleftrightarrow \; \theta = k \pi, \quad \text{for all $k \in \mathbb{Z}$}
\end{equation*}
(see \cite[Section~2]{Bo-11} and the references therein for further details). As a consequence of this observation, if we denote by $\vartheta_{\mu}(t;\vartheta_{0})$ the angular coordinate of the solution to system \eqref{sys-sturm} with initial condition $(z_{1}(0),z_{2}(0)) = (\cos\vartheta_{0},-\sqrt{\mu}\sin\vartheta_{0})$, \eqref{eq-f+} will be proved if we show that
\begin{equation*}
\lim_{\mu \to +\infty} \bigl{(}\vartheta_{\mu}(T;\vartheta_{0}) - \vartheta_{0} \bigr{)}= +\infty, \quad \text{uniformly in $\vartheta_{0} \in \mathopen{[}0,2\pi\mathclose{[}$.}
\end{equation*}
To this end, we first observe that $\vartheta_{\mu}(t;\vartheta_{0})$ solves the differential equation
\begin{equation*}
\vartheta_{\mu}'(t;\vartheta_{0}) = \sqrt{\mu} \biggl{[} \frac{\sin^{2} \vartheta_{\mu}(t;\vartheta_{0})}{p(t)} + \biggl{(}1 + \frac{q(t)}{\mu}\biggr{)} \cos^{2} \vartheta_{\mu}(t;\vartheta_{0}) \biggr{]}.
\end{equation*} 
Then, by integrating on $\mathopen{[}0,T\mathclose{]}$, we have
\begin{align*}
\vartheta_{\mu}(T;\vartheta_{0}) - \vartheta_{0} & \geq \sqrt{\mu} \int_{0}^{T} \biggl{(} \frac{\sin^{2} \vartheta_{\mu}(t;\vartheta_{0})}{p(t)} + \cos^{2} \vartheta_{\mu}(t;\vartheta_{0})\biggr{)} \,dt \\
& \quad - \frac{1}{\sqrt{\mu}} \int_{0}^{T} | q(t) | \cos^{2} \vartheta_{\mu}(t;\vartheta_{0})\,dt \\
& \geq \sqrt{\mu} \, T \min\biggl{\{} \frac{1}{\| p \|_{\infty}},1\biggr{\}} - \frac{\| q \|_{L^{1}_{T}}}{\sqrt{\mu}}.
\end{align*}
Then, taking the limit for $\mu \to +\infty$, the conclusion follows.
\end{proof}

We can now state and prove our result about the existence and dynamical characterization of the principal eigenvalue.

\begin{theorem}\label{th-mu}
There exists a unique value $\mu_{0} \in \mathbb{R}$ such that equation \eqref{eq-sturm} has a positive $T$-periodic solution; $\mu_{0}$ can be characterized as the unique zero of the function $f$ defined in Lemma~\ref{lem-f}.
\end{theorem}

\begin{proof}
With standard arguments it can be proved that, if $\hat{\mu}, \check{\mu} \in \mathbb{R}$ are such that equation \eqref{eq-sturm} has a positive $T$-periodic solution for $\mu = \hat{\mu}$ and $\mu = \check{\mu}$, then $\hat{\mu} = \check{\mu}$. Namely, the principal eigenvalue (if it exists) is unique. 

Let $\mu_{0} \in \mathbb{R}$ be the unique value such that $f(\mu_{0}) = 0$ (by Lemma~\ref{lem-f}).
To conclude the proof, we need to show that there exists an associated positive $T$-periodic solution.

To this end, we define the function $\eta(\theta_{0}) := \theta_{\mu_{0}}(T;\theta_{0}) - \theta_{0}$ and we first claim that
\begin{equation}\label{eq-eta}
\eta'(\theta_{0}) = \frac{1}{\ell_{\mu_{0}}(T;\theta_{0})^{2}} - 1,
\end{equation}
where, with an obvious notation, $\ell_{\mu_{0}}(t;\theta_{0})$ denotes the radial coordinate of the solution $(z_{1}(t;\theta_{0}),z_{2}(t;\theta_{0})$ to \eqref{sys-sturm} (for $\mu = \mu_{0}$) with initial condition $(z_{1}(0),z_{2}(0)) = (\cos\theta_{0},-\sin\theta_{0})$. To prove \eqref{eq-eta}, we first observe that
$\xi(t) := \tfrac{\partial \theta_{\mu_{0}}}{\partial \theta_{0}}(t;\theta_{0})$ satisfies $\xi(0) = 1$ and 
\begin{equation*}
\xi'(t) = 2\biggl{(} \dfrac{1}{p(t)} - (\mu_{0} + q(t)) \biggr{)} \sin \theta_{\mu_{0}}(t;\theta_{0}) \cos \theta_{\mu_{0}}(t;\theta_{0}) \xi(t), 
\end{equation*}
which in turn, by deriving the identity
\begin{equation*}
\ell_{\mu_{0}}(t;\theta_{0})^{2} = |z_{1}(t;\theta_{0})|^{2}+ |z_{1}(t;\theta_{0})|^{2},
\end{equation*}
can be written as
\begin{equation*}
\xi'(t) = - 2 \,\dfrac{\ell_{\mu_{0}}'(t;\theta_{0})}{\ell_{\mu_{0}}(t;\theta_{0})} \,\xi(t).
\end{equation*}
By elementary integration, we obtain $\xi(t) = (\ell_{\mu_{0}}(t;\theta_{0}))^{-2}$, from which \eqref{eq-eta} plainly follows.

We are now in a position to conclude. Let $\theta_{0}^{*} \in \mathopen{[}0,2\pi\mathclose{[}$ be a minimum point of the function $\eta$. Then,
\begin{equation*}
\eta(\theta_{0}^{*}) = f(\mu_{0}) = 0 \quad \text{ and } \quad \eta'(\theta_{0}^{*}) = 0,
\end{equation*}
implying, by the definition of $\eta$ and \eqref{eq-eta}, that
\begin{equation*}
\theta_{\mu_{0}}(T;\theta_{0}) = \theta_{0} \quad \text{ and } \quad \ell_{\mu_{0}}(T;\theta_{0}) = 1.
\end{equation*}
Since $\ell_{\mu_{0}}(0;\theta_{0}) = 1$, we have thus found a $T$-periodic solution to \eqref{sys-sturm}, having zero winding number around the origin.
As is well known (as a consequence of the fact that $\theta_{\mu_{0}}'$ is strictly increasing whenever 
$\theta_{\mu_{0}} \in \tfrac{\pi}{2} + \mathbb{Z}$, compare with the end of the proof of Theorem~\ref{th-sub}) this gives rise to a one-signed $T$-periodic solution to \eqref{eq-sturm}, thus concluding the proof. 
\end{proof}

As a consequence of Lemma~\ref{lem-f} and Theorem~\ref{th-mu}, we can finally state and prove the following corollary.

\begin{corollary}\label{cor-mu}
Assume $\mu_{0} < 0$. Then, it holds that
\begin{equation*}
\lim_{k \to +\infty} \bigl{(}\theta_{0}(kT;\theta_{0}) - \theta_{0}\bigr{)} = +\infty, \quad \text{uniformly in $\theta_{0} \in \mathopen{[}0,2\pi\mathclose{[}$.}
\end{equation*}
\end{corollary}

\begin{proof}
Since the function $f$ defined in Lemma~\ref{lem-f} is strictly increasing, the assumption $\mu_{0} < 0$ implies that, for a suitable $\eta > 0$,
\begin{equation*}
\theta_{0}(T;\theta_{0}) - \theta_{0} \geq \eta, \quad \text{for all $\theta_{0} \in \mathbb{R}$.}
\end{equation*}
Now we write
\begin{equation*}
\theta_{0}(kT;\theta_{0}) - \theta_{0} = \sum_{j=1}^{k} \bigl{(}\theta_{0}(jT;\theta_{0}) - \theta_{0}((j-1)T;\theta_{0})\bigr{)}
\end{equation*}
and we observe that, since equation \eqref{eq-theta} is $T$-periodic in $t$ and $2\pi$-periodic in $\theta_{0}$, it holds that
\begin{equation*}
\theta_{0}(jT;\theta_{0}) - \theta_{0}((j-1)T;\theta_{0}) = \theta_{0}(T;\theta_{0,j}) - \theta_{0,j},
\end{equation*}
where $\theta_{0,j} = \theta_{0}((j-1)T;\theta_{0})$. Hence, for any $\theta_{0} \in \mathbb{R}$,
\begin{equation*}
\theta_{0}(kT;\theta_{0}) - \theta_{0} \geq k \eta,
\end{equation*}
thus implying the conclusion.
\end{proof}

\begin{remark}\label{rem-C1}
Higher eigenvalues $\mu_{k}', \mu_{k}''$, with $k \geq 1$, can be defined as the unique zeros of the functions
\begin{equation*}
\mu \mapsto \max_{\theta_{0} \in \mathopen{[}0,2\pi\mathclose{[}} \bigl{(}\theta_{\mu}(T;\theta_{0}) - \theta_{0} - 2k\pi\bigr{)}, \quad
\mu \mapsto \min_{\theta_{0} \in \mathopen{[}0,2\pi\mathclose{[}} \bigl{(}\theta_{\mu}(T;\theta_{0}) - \theta_{0} - 2k\pi\bigr{)},
\end{equation*}
respectively; of course, they give rise to eigenfunctions having exactly $2k$ zeros on $\mathopen{[}0,T\mathclose{[}$. With some extra work with respect to the arguments used in the proof of Theorem~\ref{th-mu}, it can be shown that the above defined eigenvalues form a sequence
\begin{equation*}
\mu_{0} < \mu_{1}' \leq \mu_{1}'' < \mu_{2}' \leq \mu_{2}'' < \ldots < \mu_{k}' \leq \mu_{k}'' < \ldots
\end{equation*}
and that they are the only eigenvalues for the $T$-periodic problem associated with \eqref{eq-sturm}. For more details, we refer to \cite{GaZh-00,Za-03} (in the case $p(t) \equiv 1$).
$\hfill\lhd$
\end{remark}

\bibliographystyle{elsart-num-sort}
\bibliography{BoFe-biblio}

\end{document}